\documentclass[11pt,b5paper,notitlepage]{article}
\usepackage[b5paper, margin={0.5in,0.65in}]{geometry}

\usepackage{amsmath,amscd,amssymb,amsthm,mathrsfs,amsfonts,layout,indentfirst,graphicx,caption,mathabx, stmaryrd,appendix,calc,imakeidx,upgreek,appendix} 
\usepackage[dvipsnames]{xcolor}
\usepackage{palatino}  
\usepackage{slashed} 
\usepackage{mathrsfs} 
\usepackage{extarrows} 
\usepackage{enumitem} 
\usepackage[all]{xy}
\usepackage{fancyhdr} 
\usepackage{verbatim}

\usepackage{tikz-cd}
\usepackage[nottoc]{tocbibind}   

\usepackage{lipsum}
\let\OLDthebibliography\thebibliography
\renewcommand\thebibliography[1]{
	\OLDthebibliography{#1}
	\setlength{\parskip}{0pt}
	\setlength{\itemsep}{2pt} 
}

\allowdisplaybreaks  
\usepackage{latexsym}
\usepackage{chngcntr}
\usepackage[colorlinks,linkcolor=blue,anchorcolor=blue, linktocpage,
]{hyperref}
\hypersetup{ urlcolor=cyan,
	citecolor=[rgb]{0,0.5,0}}

\usepackage{simpler-wick}

\usepackage{thmtools}
\newtheoremstyle{introTheorems}
  {\topsep}
  {\topsep}
  {\itshape}
  {0pt}
  {\bfseries}
  {}
  { }
  {\thmname{#1}
  \textnormal{\thmnote{#3}.}
  }
\theoremstyle{introTheorems}
\newtheorem{introTheorem}{Theorem}

\DeclareFontFamily{U}{rcjhbltx}{}
\DeclareFontShape{U}{rcjhbltx}{m}{n}{<->s*[1.15]rcjhbltx}{}   
\DeclareSymbolFont{hebrewletters}{U}{rcjhbltx}{m}{n}

\DeclareMathSymbol{\lamed}{\mathord}{hebrewletters}{108}
\DeclareMathSymbol{\mem}{\mathord}{hebrewletters}{109}
\DeclareMathSymbol{\ayin}{\mathord}{hebrewletters}{96}
\DeclareMathSymbol{\tsadi}{\mathord}{hebrewletters}{118}
\DeclareMathSymbol{\qof}{\mathord}{hebrewletters}{113}
\DeclareMathSymbol{\shin}{\mathord}{hebrewletters}{152}

\counterwithin{figure}{section}

\theoremstyle{definition}
\newtheorem{df}{Definition}[section]

\newtheorem{rem}[df]{Remark}
\newtheorem{ass}[df]{Assumption}
\newtheorem{cv}[df]{Convention}

\theoremstyle{plain}
\newtheorem{thm}[df]{Theorem}

\newtheorem{pp}[df]{Proposition}
\newtheorem{co}[df]{Corollary}
\newtheorem{lm}[df]{Lemma}

\newcommand{\fk}{\mathfrak}
\newcommand{\mc}{\mathcal}
\newcommand{\wtd}{\widetilde}
\newcommand{\wht}{\widehat}
\newcommand{\wch}{\widecheck}
\newcommand{\ovl}{\overline}

\newcommand{\tr}{\mathrm{t}} 

\newcommand{\End}{\mathrm{End}} 
\newcommand{\idt}{\mathbf{1}}
\newcommand{\Hom}{\mathrm{Hom}}

\newcommand{\Res}{\mathrm{Res}}

\newcommand{\opp}{\mathrm{op}}

\newcommand{\SV}{\mathscr{V}}
\newcommand{\Span}{\mathrm{Span}}

\newcommand{\Vect}{\mathcal Vect}
\newcommand{\scr}{\mathscr}

\newcommand{\yk}{\mathfrak y}

\newcommand{\im}{\mathbf{i}}

\newcommand{\sgm}{\varsigma}
\newcommand{\SX}{{S_{\fk X}}}

\newcommand{\mbb}{\mathbb}

\newcommand{\blt}{\bullet}

\newcommand{\Vbb}{\mathbb V}

\newcommand{\Xbb}{\mathbb X}
\newcommand{\Ybb}{\mathbb Y}
\newcommand{\Wbb}{\mathbb W}
\newcommand{\Mbb}{\mathbb M}

\newcommand{\Cbb}{\mathbb C}
\newcommand{\Nbb}{\mathbb N}
\newcommand{\Zbb}{\mathbb Z}
\newcommand{\Pbb}{\mathbb P}
\newcommand{\Rbb}{\mathbb R}
\newcommand{\Ebb}{\mathbb E}
\newcommand{\cbf}{\mathbf c}

\newcommand{\Ker}{\mathrm{Ker}}

\newcommand{\pr}{\mathrm {pr}}
\newcommand{\SXtd}{S_{\wtd{\fk X}}}

\newcommand{\srm}{\mathrm{s}}
\newcommand{\nrm}{\mathrm{n}}

\newcommand{\MO}{\mathcal{O}}
\newcommand{\MU}{\mathcal{U}}

\newcommand{\MC}{\mathcal{C}}
\newcommand{\MB}{\mathcal{B}}

\newcommand{\fx}{\mathfrak{X}}

\newcommand{\ST}{\mathscr{T}}

\newcommand{\SW}{\mathscr{W}}

\newcommand{\MD}{\mathcal{D}}
\newcommand{\MS}{\mathcal{S}}

\newcommand{\mk}{\mathfrak m}

\newcommand{\bk}[1]{\langle {#1}\rangle}
\newcommand{\bigbk}[1]{\big\langle {#1}\big\rangle}

\newcommand{\bbs}{\boxbackslash}
\newcommand{\fq}{{\mathfrak Q}}
\newcommand{\Mod}{\mathrm{Mod}}
\newcommand{\id}{\mathrm{id}}

\newcommand{\eps}{\varepsilon}

\newcommand{\fp}{\mathfrak{P}}
\newcommand{\fn}{\mathfrak{N}}
\newcommand{\fy}{\mathfrak{Y}}

\newcommand{\MCtd}{{\widetilde{\mathcal C}}}
\newcommand{\ff}{\mathfrak{F}}
\newcommand{\fg}{\mathfrak{G}}
\newcommand{\Imag}{\mathrm{Im}}
\newcommand{\Fbb}{\mathbb{F}}
\newcommand{\fxtd}{{\widetilde{\mathfrak X}}}
\newcommand{\kabla}{\nabla^{\varkappa}}
\newcommand{\ssp}{\mathrm{s}}
\newcommand{\nil}{\mathrm{n}}
\newcommand{\Irr}{\mathrm{Irr}}
\newcommand{\fz}{\mathfrak{Z}}
\newcommand{\fh}{\mathfrak{H}}
\newcommand{\Lex}{{\mathcal Lex}}

\usepackage{tipa} 
\newcommand{\tipar}{\text{\textrtailr}}
\newcommand{\tipaz}{\text{\textctyogh}}

\usepackage{tipx}

\newcommand{\tipxphi}{\text{\textqplig}}

\numberwithin{equation}{section}

\title{Analytic Conformal Blocks of $C_2$-cofinite Vertex Operator Algebras III: The Sewing-Factorization Theorems}
\author{{\sc Bin Gui, Hao Zhang}
}
\date{}
\begin{document}\sloppy 
	\pagenumbering{arabic}
	\setcounter{section}{-1}

	\maketitle

\newcommand\blfootnote[1]{%
	\begingroup
	\renewcommand\thefootnote{}\footnote{#1}%
	\addtocounter{footnote}{-1}%
	\endgroup
}



\begin{abstract}
Let $\Vbb=\bigoplus_{n\in\Nbb}\Vbb(n)$ be a $C_2$-cofinite vertex operator algebra, not necessarily rational or self-dual. In this paper, we establish various versions of the sewing-factorization (SF) theorems for conformal blocks associated to grading-restricted generalized modules of $\Vbb^{\otimes N}$ (where $N\in\Nbb$). In addition to the versions announced in the Introduction of \cite{GZ1}, we prove the following coend version of the SF theorem:

Let $\ff$ be a compact Riemann surface with $N$ incoming and $R$ outgoing marked points, and let $\fg$ be another compact Riemann surface with $K$ incoming and $R$ outgoing marked points. Assign $\Wbb\in\Mod(\Vbb^{\otimes N})$ and $\Xbb\in\Mod(\Vbb^{\otimes K})$ to the incoming marked points of $\ff$ and $\fg$ respectively. For each $\Mbb\in\Mod(\Vbb^{\otimes R})$, assign $\Mbb$ and its contragredient $\Mbb'$ to the outgoing marked points of $\ff$ and $\fg$ respectively. Denote the corresponding spaces of conformal blocks by $\ST^*_\ff(\Mbb\otimes\Wbb)$ and $\ST^*_\fg(\Mbb'\otimes\Xbb)$. Let $\fx$ be the $(N+K)$-pointed surface obtained by sewing $\ff,\fg$ along their outgoing marked points. Then the sewing of conformal blocks---proved to be convergent in \cite{GZ2}---yields an isomorphism of vector spaces
\begin{align*}
\int^{\Mbb\in\Mod(\Vbb^{\otimes R})}\ST^*_\ff\big(\Mbb\otimes\Wbb\big)\otimes_\Cbb\ST^*_\fg\big(\Mbb'\otimes\Xbb\big)\simeq\ST_\fx^*\big(\Wbb\otimes \Xbb\big)
\end{align*}

We also discuss the relationship between conformal blocks and the modular functors defined using Lyubashenko's coend/construction.
\end{abstract}

\tableofcontents





	
	

	

\section{Introduction}

\nocite{HLZ1,HLZ2,HLZ3,HLZ4,HLZ5,HLZ6,HLZ7,HLZ8}

\subsection{Sewing-factorization (SF) theorems: from rational to irrational VOAs}

In this final part of our three-part series, we prove the \textbf{sewing-factorization (SF) theorems} for conformal blocks of a $C_2$-cofinite  vertex operator algebra (VOA) $\Vbb=\bigoplus_{n\in\Nbb}\Vbb(n)$, several versions of which were originally announced in the Introduction of \cite{GZ1}. Roughly speaking, such a theorem asserts that sewing conformal blocks establishes an equivalence between the spaces, sheaves, or functors of conformal blocks associated with compact Riemann surfaces before sewing and those associated with the surfaces after sewing. 

In the literature on rational VOAs, the proof of the sewing-factorization (SF) theorem for rational  VOAs holds a central position \cite{TUY,Zhu-modular-invariance,Hua-tensor-4,Hua-differential-genus-0,NT-P1_conformal_blocks,Hua-differential-genus-1}. A recent breakthrough in this area is the proof of the factorization formulation and a ``formal sewing" theorem for rational $C_2$-cofinite VOAs in \cite{DGT2}. This result was later translated into the analytic setting, leading to a proof of the SF theorem for rational $C_2$-cofinite VOAs; see \cite[Thm. 12.1]{Gui-sewingconvergence}. 

However, even in the case of rational VOAs, there are discrepancies in the formulation of the SF theorem: for instance, whether to involve nodal curves; whether to formulate the theory in terms of coinvariants or conformal blocks; and whether to treat Riemann surfaces of arbitrary genus directly, or to first study lower-genus cases and then reduce higher-genus cases to lower-genus cases via pants decomposition. Unfortunately, the literature does not seem to provide a clear and detailed account of how one version of the rational SF theorem can be translated into another.

The situation becomes even more intricate when considering $C_2$-cofinite VOAs that are not necessarily rational. Here, we focus on the complex-analytic setting. (For an algebro-geometric perspective, see \cite{DGK-presentations, DGK2}.) The theory of vertex tensor categories developed by Huang-Lepowsky-Zhang in \cite{HLZ1}, \cite{HLZ2}-\cite{HLZ8}, particularly their construction of associativity isomorphisms in the category $\Mod(\Vbb)$ of grading-restricted generalized $\Vbb$-modules, can be regarded as a genus-zero version of the SF theorem. (The precise relationship between the associativity isomorphisms and the SF theorems proved in this paper will be given in Subsec. \ref{lb55}.) In particular, the product and iterate of (logarithmic) intertwining operators may be interpreted as instances of sewing conformal blocks—that is, as constructing conformal blocks via contraction.

However, as shown in \cite{Miy-modular-invariance}, the usual sewing of conformal blocks corresponding to the self-sewing of a sphere is insufficient to produce all genus-one conformal blocks. To recover all torus conformal blocks from genus-zero data, one must instead employ \textbf{pseudo-$q$-traces}, cf. \cite{Miy-modular-invariance,AN-pseudo-trace,Hua-modular-C2}. The core issue lies in \textbf{self-sewing}, that is, sewing a compact Riemann surface along a pair of points on the same connected component. As emphasized in the introductions of \cite{GZ1, GZ2}, the distinction between self-sewing and disjoint sewing is crucial: while the usual sewing of conformal blocks (defined via contraction) is sufficient in the case of disjoint sewing, it is generally inadequate for self-sewing. 

Besides pseudo-$q$-traces, \textbf{left exact coends} offer another approach to addressing the challenge of self-sewing conformal blocks. This method was originally introduced by Lyubashenko in \cite{Lyu96-Ribbon} in the context of topological modular functors. (See also \cite{FS-coends-CFT, HR24-MF} for discussions.) To our knowledge, the precise relationship between pseudo-$q$-traces and left exact coends remains unclear. Nevertheless, the coend version of the SF theorem presented in this paper can be more directly connected to pseudo-$q$-traces. 

Before stating our version of the SF theorem, we begin with a brief overview of the left exact coend approach in the VOA setting.

\subsection{Self-sewing and left exact coends}\label{lb56}

Coends naturally generalize the construction of ``direct sums of a complete set of irreducible objects" in the semisimple case. Let $N\in\Nbb$, and let $\scr D$ be a category.  Recall that if $F:\Mod(\Vbb^{\otimes N})\times \Mod(\Vbb^{\otimes N})\rightarrow\scr D$ is a covariant bifunctor and $A\in\scr D$,  then a family of morphisms
\begin{align}\label{eq114}
\varphi_\Wbb:F(\Wbb', \Wbb)\rightarrow A
\end{align}
for all $\Wbb\in\Mod(\Vbb^{\otimes N})$ (with contragredient module $\Wbb'$) is called \textbf{dinatural} if for any $\Mbb\in\Mod(\Vbb^{\otimes N})$ and $T\in\Hom_{\Vbb^{\otimes N}}(\Mbb,\Wbb)$ (with transpose $T^\tr$), the following diagram commutes:
\begin{equation*}
\begin{tikzcd}[row sep=large, column sep=huge]
F(\Wbb',\Mbb) \arrow[r,"{F(T^\tr,\id_\Mbb)}"] \arrow[d,"{F(\id_{\Wbb'},T)}"'] & F(\Mbb',\Mbb) \arrow[d,"\varphi_\Mbb"] \\
F(\Wbb',\Wbb) \arrow[r,"\varphi_\Wbb"]           & A         
\end{tikzcd}
\end{equation*}
A dinatural transformation \eqref{eq114} is called a \textbf{coend} (in $\scr D$) if it satisfies the universal property that for any $B\in\scr D$ and dinatural transformation $\psi_\Wbb:F(\Wbb',\Wbb)\rightarrow B$ (for all $\Wbb\in\Mod(\Vbb^{\otimes N})$) there is a unique $\Phi\in\Hom_{\scr D}(A,B)$ such that $\psi_\Wbb=\Phi\circ\varphi_\Wbb$ holds for all $\Wbb$. In that case, we write
\begin{align*}
A=\int^{\Wbb\in\Mod(\Vbb^{\otimes N})}F(\Wbb',\Wbb)
\end{align*}

In the theory of VOAs, the sewing of conformal blocks provides a fundamental class of dinatural transformations. We begin by recalling some notation. Let
\begin{align}\label{eq116}
\wtd\fx=(\MCtd|\sgm_1,\dots,\sgm_N\Vert \sgm',\sgm'')
\end{align}
be an $(N+2)$-pointed compact Riemann surface, that is, $\MCtd$ is a (not necessarily connected) compact Riemann surface (without boundary), and $\sgm_1,\dots,\sgm_N,\sgm',\sgm''$ are distinct points of $\MCtd$. Assume that each component of $\MCtd$ contains one of $\sgm_1,\dots,\sgm_N$. We associate a local coordinate $\eta_i$ to $\sgm_i$.\footnote{In other words, $\eta_i$ is a univalent (=injective holomorphic) function on a neighborhood of $\sgm_i$ satisfying $\eta_i(\sgm_i)=0$.} Similarly, let $\xi$ and $\varpi$ be local coordinates of $\sgm'$ and $\sgm''$ respectively. Assume that $\xi$ (resp. $\varpi$) is defined on the neighborhood $V'$ (resp. $V''$) of $\sgm'$ (resp. $\sgm''$), such that $V',V'',\sgm_1,\cdots,\sgm_N$ are mutually disjoint. For $p\in\Cbb^\times$ with reasonably large $|p|$, we obtain a new compact Riemann surface
\begin{align*}
\fx_p=(\MC_p|\sgm_1,\dots,\sgm_N)
\end{align*}
(with local coordinates $\eta_1,\dots,\eta_N$) by \textbf{sewing $\wtd\fx$ along the pair of points $\sgm',\sgm''$ with moduli $p$}. More precisely, $\fx_p$ is constructed by removing a closed disk $D'\subset V'$ centered at $\sgm'$, removing another closed disk $D''\subset V''$ centered at $\sgm''$, and then gluing the resulting surface by identifying $x\in V'-D'$ with $y\in V''-D''$ whenever $\xi(x)\cdot\varpi(y)=p$.

Fix $\Wbb\in\Mod(\Vbb^{\otimes N})$, and associate $\Wbb$ to the ordered marked points $\sgm_1,\dots,\sgm_N$ of $\fxtd$ and $\fx_p$. For each $\Mbb\in\Mod(\Vbb)$, associate $\Mbb$  and its contragredient $\Mbb'$ to $\sgm',\sgm''$ respectively. Let $\ST^*_{\fx_p}(\Wbb)$ be the (finite-dimensional) space of conformal blocks associated to $\Wbb$ and $\fx_p$.\footnote{Thus, each element of $\ST^*_{\fx_p}(\Wbb)$ is a linear functional $\Wbb\rightarrow\Cbb$ satisfying a suitable invariance property. See Subsec. \ref{lb57} for details.} Let $\ST^*_\fxtd(\Wbb\otimes\Mbb\otimes\Mbb')$ be the space of conformal blocks associated to $\Wbb\otimes\Mbb\otimes\Mbb'$ and $\fxtd$. We have proved in \cite{GZ2} that for each $\uppsi\in \ST^*_\fxtd(\Wbb\otimes\Mbb\otimes\Mbb')$, the contraction
\begin{gather*}
\MS\uppsi:\Wbb\rightarrow\Cbb\{q\}[\log q]\\
w\mapsto \wick{\uppsi(w\otimes q^{L(0)}\c1-\otimes\c1-)}:=\sum_{\lambda\in\Cbb}\sum_{\alpha\in\fk A_\lambda}\uppsi(w\otimes q^{L(0)}e_\lambda(\alpha)\otimes \wch e_\lambda(\alpha))
\end{gather*}
converges absolutely, where $(e_\lambda(\alpha))_{\alpha\in\fk A_\lambda}$ is a (finite) basis of $\Mbb_{[\lambda]}$, the generalized eigenspace of $L(0)$ on $\Mbb$ with eigenvalue $\lambda$, and $(\wch e_\lambda(\alpha))_{\alpha\in\fk A_\lambda}$ is the dual basis. Moreover, \cite{GZ2} shows that for each $p$, we have $\MS\uppsi|_p\in\ST^*_{\fx_p}(\Wbb)$. Therefore, by letting $\Mbb\in\Mod(\Vbb)$ vary, we obtain a family of linear maps
\begin{align}\label{eq115}
\ST^*_{\wtd\fx}(\Wbb\otimes\Mbb\otimes\Mbb')\rightarrow\ST^*_{\fx_p}(\Wbb)\qquad\uppsi\mapsto\MS\uppsi|_p
\end{align}
which is clearly dinatural.

In the case of self-sewing, the dinatural transformation \eqref{eq115} is generally not a coend. Indeed, by the universal property of coends, any dinatural transformation into 
\begin{align*}
\Vect:=\text{ the category of finite-dimensional $\Cbb$-linear spaces}
\end{align*}
must be surjective in order to be a coend. However, Miyamoto’s work \cite{Miy-modular-invariance} demonstrates that when the surface \( \MCtd \) in \eqref{eq116} is taken to be \( \mathbb{P}^1 \) and $N=1,\Wbb=\Vbb$, the span of the images of \eqref{eq115} over all \( \Mbb \in \Mod(\Vbb) \) fails to be surjective when \( \Vbb \) is not rational.

Lyubashenko's work \cite{Lyu96-Ribbon} suggests a method for obtaining a coend from \eqref{eq115}. (See also \cite[Prop. 9]{FS-coends-CFT} or \cite[Prop. 4.8]{HR24-MF}.) Instead of fixing $\Wbb$, we let $\Wbb\in\Mod(\Vbb^{\otimes N})$ vary. Then, noting that the conformal block functor is a left exact contravariant functor (cf. Thm. \ref{lb11}), the dinatural transformation \eqref{eq115} gives rise to a family of morphisms
\begin{align}\label{eq117}
\ST^*_\fxtd(-\otimes\Mbb\otimes\Mbb')\rightarrow\ST^*_{\fx_p}(-)
\end{align}
in the category $\Lex(\Mod(\Vbb^{\otimes N}),\Vect)$ of left exact contravariant functors from $\Mod(\Vbb^{\otimes N})$ to $\Vect$. It is expected that, at least when $\Vbb$ is \textbf{strongly finite}---that is, $\Vbb$ is $C_2$-cofinite, self-dual, and satisfies $\dim\Vbb(0)=1$---the family of morphisms \eqref{eq117} is a \textbf{left exact coend} (i.e., a coend in $\Lex(\Mod(\Vbb^{\otimes N}),\Vect)$). Following the convention in the literature of denoting left exact coends by $\oint$, our expectation is that \eqref{eq117} induces an equivalence of left exact functors
\begin{align*}
\oint\nolimits^{\Mbb\in\Mod(\Vbb)}\ST^*_\fxtd(-\otimes\Mbb\otimes\Mbb')\simeq\ST^*_{\fx_p}(-)
\end{align*}

At present, it is unclear to us how to prove the above statement about left exact coends. Moreover, since Lyubashenko's approach is based on (not necessarily semisimple) modular categories---which are, in particular, rigid---and since $\Mod(\Vbb)$ is not necessarily rigid but only a Grothendieck-Verdier category when the $C_2$-cofinite VOA $\Vbb$ is not self-dual (cf. \cite{ALSW21}), it remains uncertain whether the above result on left exact coends is expected to hold in the non-self-dual case.

On the other hand, the SF theorem(s) established in this paper will indicate that, in the case of disjoint sewing of compact Riemann surfaces along multiple pairs of points, the sewing of conformal blocks is indeed a coend in $\Vect$. Moreover, in \cite{GZ2}, we showed that pseudo-$q$-traces arise naturally within our framework. A more detailed discussion of the relationship between our SF theorems and pseudo-$q$-traces will be provided in an upcoming work.

\subsection{Disjoint sewing and coends in $\Vect$}

We now introduce our SF theorems. Let
\begin{gather}\label{eq122}
\ff=(x_1',\dots,x_R'|C_1|x_1,\dots,x_N)\qquad \fg=(y_1',\dots,y_R'|C_2|y_1,\dots,y_K)
\end{gather}
be $(R,N)$-pointed and $(R,K)$-pointed compact Riemann surfaces. That is, $\ff$ is an $(R+N)$-pointed compact Riemann surface, where the marked points are divided into $N$ incoming marked points $x_1,\dots,x_N$ and $R$ outgoing marked points. $\fg$ is understood in a similar way. Moreover, we fix local coordinates at these marked points, and we assume that each component of $C_1$ contains one of $x_1,\dots,x_N$, and each component of $C_2$ contains one of $y_1,\dots,y_K,y_1',\dots,y_R'$.

Choose $p_\blt=(p_1,\dots,p_R)\in(\Cbb^\times)^R$ where $|p_1|,\dots,|p_R|$ can be reasonably large. The sewing of $\ff$ and $\fg$ along the $R$ pairs of points $(x_1',y_1'),\dots,(x_R',y_R')$ (using their local coordinates) with moduli $p_\blt$ is denoted by
\begin{align*}
\fx_{p_\blt}=\MS(\ff\sqcup\fg)_{p_\blt}=(\MC_{p_\blt}|x_1,\dots,x_N,y_1,\dots,y_K)
\end{align*}
Fix $\Wbb\in\Mod(\Vbb^{\otimes N})$ and $\Xbb\in\Mod(\Vbb^{\otimes K})$, and associate them to the ordered marked points $x_1,\dots,x_N$ and $y_1,\dots,y_K$ respectively. Similar to \eqref{eq115}, for each $\Mbb\in\Mod(\Vbb^{\otimes R})$ we have a sewing map
\begin{gather}\label{eq118}
\begin{gathered}
\ST^*_\ff(\Mbb\otimes\Wbb)\otimes_\Cbb\ST^*_\fg(\Mbb'\otimes\Xbb)\longrightarrow \ST^*_{\fx_{p_\blt}}(\Wbb\otimes\Xbb)\\
\uppsi\otimes\upchi\mapsto \MS(\uppsi\otimes\upchi)\big|_{p_\blt}
\end{gathered}
\end{gather}
defined by contracting the $\Mbb$-component of $\uppsi:\Mbb\otimes\Wbb\rightarrow\Cbb$ with the $\Mbb'$-component of $\upchi:\Mbb'\otimes\Xbb\rightarrow\Cbb$. (See Def. \ref{lb42} for the rigorous definition.) 

The following is one version of our SF theorems, which may be regarded as the VOA analog of \cite[Cor. 4.9]{HR24-MF}. In the genus-$0$ case, this theorem was obtained in \cite{Moriwaki22-CB}.

\begin{introTheorem}[\ref{lb54}]
Assume that $\Vbb$ is $C_2$-cofinite. Then, as $\Mbb\in\Mod(\Vbb^{\otimes R})$ varies, the family of linear maps \eqref{eq118} is a coend in $\Vect$. In short, the sewing of conformal blocks yields a linear isomorphism
\begin{align*}
\int^{\Mbb\in\Mod(\Vbb^{\otimes R})}\ST^*_\ff\big(\Mbb\otimes\Wbb\big)\otimes_\Cbb\ST^*_\fg\big(\Mbb'\otimes\Xbb\big)\simeq\ST_{\fx_{p_\blt}}^*\big(\Wbb\otimes \Xbb\big)
\end{align*}
\end{introTheorem}
Since $\ST^*_\ff\big(\Mbb\otimes\Wbb\big)\otimes_\Cbb\ST^*_\fg\big(\Mbb'\otimes\Xbb\big)$ is canonically isomorphic to $\ST^*_{\ff\sqcup\fg}(\Mbb\otimes\Wbb\otimes\Mbb'\otimes\Xbb)$ (cf. Thm. \ref{lb16}), Thm. \ref{lb54} can thus be illustrated by Fig. \ref{fig3}. Moreover, by the parameter theorem for (co)ends \cite[Sec. IX.7]{MacLane-Cat}, as $\Wbb$ and $\Xbb$ vary, \eqref{eq118} yields a coend in $\mc Fun(\Mod(\Vbb^{\otimes N})\times\Mod(\Vbb^{\otimes K}),\Vect)$, the category of contravariant functors from $\Mod(\Vbb^{\otimes N})\times\Mod(\Vbb^{\otimes K})$ to $\Vect$. Furthermore, by Thm. \ref{lb11}, this coend is left exact.

\newcommand{\figa}{
&\scalebox{1.2}{$\displaystyle\int$}^{\mathbb M\in\mathrm{Mod}(\mathbb V^{\otimes R})}~\ST^*\Bigg(
\vcenter{\hbox{\scalebox{0.7}{

\tikzset{every picture/.style={line width=0.75pt}} 

\begin{tikzpicture}[x=0.75pt,y=0.75pt,yscale=-1,xscale=1]

\draw   (169.78,99.48) .. controls (190.64,86.63) and (206.07,99.9) .. (223.9,107.6) .. controls (241.72,115.29) and (252.44,113.53) .. (260.54,113.68) .. controls (268.63,113.84) and (273.3,110.7) .. (273.63,119.17) .. controls (273.96,127.65) and (252.96,121.23) .. (252.96,127.9) .. controls (252.96,134.57) and (274.46,124.4) .. (275.71,133.9) .. controls (276.96,143.4) and (254.96,137.9) .. (254.96,142.9) .. controls (254.96,147.9) and (276.46,139.4) .. (276.46,149.4) .. controls (276.46,159.4) and (254.79,154.15) .. (247.57,153.59) .. controls (240.35,153.03) and (241.37,149.53) .. (221.64,152.91) .. controls (201.91,156.29) and (201.35,169.82) .. (174.85,162.38) .. controls (148.36,154.94) and (148.92,112.33) .. (169.78,99.48) -- cycle ;
\draw   (302.66,118.15) .. controls (302.41,107.9) and (333.91,118.4) .. (343.16,116.9) .. controls (352.41,115.4) and (353.94,114.57) .. (360.05,111.61) .. controls (366.16,108.65) and (382,96.49) .. (397.98,108.4) .. controls (413.96,120.31) and (416.98,155.4) .. (395.25,163.95) .. controls (373.53,172.49) and (360.75,157.45) .. (348.75,151.45) .. controls (336.75,145.45) and (308.16,163.65) .. (305.66,151.4) .. controls (303.16,139.15) and (318.16,145.65) .. (317.91,140.15) .. controls (317.66,134.65) and (299.91,144.65) .. (299.16,134.15) .. controls (298.41,123.65) and (317.16,132.9) .. (317.91,127.65) .. controls (318.66,122.4) and (302.91,128.4) .. (302.66,118.15) -- cycle ;
\draw    (169.87,121.96) .. controls (180.25,112.69) and (188.45,115.16) .. (195.55,123.19) ;
\draw    (173.15,120.1) .. controls (179.16,125.05) and (185.72,124.43) .. (191.73,120.1) ;
\draw    (189.37,143.96) .. controls (199.75,134.69) and (207.95,137.16) .. (215.05,145.19) ;
\draw    (192.65,142.1) .. controls (198.66,147.05) and (205.22,146.43) .. (211.23,142.1) ;
\draw    (367.07,135.46) .. controls (377.45,126.19) and (385.65,128.66) .. (392.75,136.69) ;
\draw    (370.35,133.6) .. controls (376.36,138.55) and (382.92,137.93) .. (388.93,133.6) ;
\draw  [fill={rgb, 255:red, 0; green, 0; blue, 0 }  ,fill opacity=1 ] (158.5,125.38) .. controls (158.5,124.48) and (159.23,123.74) .. (160.13,123.74) .. controls (161.04,123.74) and (161.77,124.48) .. (161.77,125.38) .. controls (161.77,126.28) and (161.04,127.02) .. (160.13,127.02) .. controls (159.23,127.02) and (158.5,126.28) .. (158.5,125.38) -- cycle ;
\draw  [fill={rgb, 255:red, 0; green, 0; blue, 0 }  ,fill opacity=1 ] (267.5,117.93) .. controls (267.5,117.03) and (268.23,116.29) .. (269.13,116.29) .. controls (270.04,116.29) and (270.77,117.03) .. (270.77,117.93) .. controls (270.77,118.83) and (270.04,119.57) .. (269.13,119.57) .. controls (268.23,119.57) and (267.5,118.83) .. (267.5,117.93) -- cycle ;
\draw  [fill={rgb, 255:red, 0; green, 0; blue, 0 }  ,fill opacity=1 ] (269.5,134.54) .. controls (269.5,133.63) and (270.23,132.9) .. (271.13,132.9) .. controls (272.04,132.9) and (272.77,133.63) .. (272.77,134.54) .. controls (272.77,135.44) and (272.04,136.17) .. (271.13,136.17) .. controls (270.23,136.17) and (269.5,135.44) .. (269.5,134.54) -- cycle ;
\draw  [fill={rgb, 255:red, 0; green, 0; blue, 0 }  ,fill opacity=1 ] (269.5,149.54) .. controls (269.5,148.63) and (270.23,147.9) .. (271.13,147.9) .. controls (272.04,147.9) and (272.77,148.63) .. (272.77,149.54) .. controls (272.77,150.44) and (272.04,151.17) .. (271.13,151.17) .. controls (270.23,151.17) and (269.5,150.44) .. (269.5,149.54) -- cycle ;
\draw  [fill={rgb, 255:red, 0; green, 0; blue, 0 }  ,fill opacity=1 ] (306.2,119.29) .. controls (306.2,118.38) and (306.93,117.65) .. (307.83,117.65) .. controls (308.74,117.65) and (309.47,118.38) .. (309.47,119.29) .. controls (309.47,120.19) and (308.74,120.92) .. (307.83,120.92) .. controls (306.93,120.92) and (306.2,120.19) .. (306.2,119.29) -- cycle ;
\draw  [fill={rgb, 255:red, 0; green, 0; blue, 0 }  ,fill opacity=1 ] (302.95,133.79) .. controls (302.95,132.88) and (303.68,132.15) .. (304.58,132.15) .. controls (305.49,132.15) and (306.22,132.88) .. (306.22,133.79) .. controls (306.22,134.69) and (305.49,135.42) .. (304.58,135.42) .. controls (303.68,135.42) and (302.95,134.69) .. (302.95,133.79) -- cycle ;
\draw  [fill={rgb, 255:red, 0; green, 0; blue, 0 }  ,fill opacity=1 ] (309.45,149.79) .. controls (309.45,148.88) and (310.18,148.15) .. (311.08,148.15) .. controls (311.99,148.15) and (312.72,148.88) .. (312.72,149.79) .. controls (312.72,150.69) and (311.99,151.42) .. (311.08,151.42) .. controls (310.18,151.42) and (309.45,150.69) .. (309.45,149.79) -- cycle ;
\draw  [fill={rgb, 255:red, 0; green, 0; blue, 0 }  ,fill opacity=1 ] (163.7,144.38) .. controls (163.7,143.48) and (164.43,142.74) .. (165.33,142.74) .. controls (166.24,142.74) and (166.97,143.48) .. (166.97,144.38) .. controls (166.97,145.28) and (166.24,146.02) .. (165.33,146.02) .. controls (164.43,146.02) and (163.7,145.28) .. (163.7,144.38) -- cycle ;
\draw  [fill={rgb, 255:red, 0; green, 0; blue, 0 }  ,fill opacity=1 ] (393.4,115.78) .. controls (393.4,114.88) and (394.13,114.14) .. (395.03,114.14) .. controls (395.94,114.14) and (396.67,114.88) .. (396.67,115.78) .. controls (396.67,116.68) and (395.94,117.42) .. (395.03,117.42) .. controls (394.13,117.42) and (393.4,116.68) .. (393.4,115.78) -- cycle ;
\draw  [fill={rgb, 255:red, 0; green, 0; blue, 0 }  ,fill opacity=1 ] (399.2,148.68) .. controls (399.2,147.78) and (399.93,147.04) .. (400.83,147.04) .. controls (401.74,147.04) and (402.47,147.78) .. (402.47,148.68) .. controls (402.47,149.58) and (401.74,150.32) .. (400.83,150.32) .. controls (399.93,150.32) and (399.2,149.58) .. (399.2,148.68) -- cycle ;
\draw  [fill={rgb, 255:red, 0; green, 0; blue, 0 }  ,fill opacity=1 ] (399.3,131.78) .. controls (399.3,130.88) and (400.03,130.14) .. (400.93,130.14) .. controls (401.84,130.14) and (402.57,130.88) .. (402.57,131.78) .. controls (402.57,132.68) and (401.84,133.42) .. (400.93,133.42) .. controls (400.03,133.42) and (399.3,132.68) .. (399.3,131.78) -- cycle ;

\draw (252.8,90.4) node [anchor=north west][inner sep=0.75pt]    {\scalebox{1.2}{$\mathbb M'$}};
\draw (294.4,90.4) node [anchor=north west][inner sep=0.75pt]    {\scalebox{1.2}{$\mathbb M$}};
\draw (416,122.4) node [anchor=north west][inner sep=0.75pt]    {\scalebox{1.2}{$\mathbb W$}};
\draw (132.8,130.4) node [anchor=north west][inner sep=0.75pt]    {\scalebox{1.2}{$\mathbb X$}};
\end{tikzpicture}
}}}\Bigg)\\
\scalebox{1.2}{$\simeq$}~&\ST^*\Bigg(
\vcenter{\hbox{\scalebox{0.65}{

\tikzset{every picture/.style={line width=0.8pt}} 

\begin{tikzpicture}[x=0.75pt,y=0.75pt,yscale=-1,xscale=1]

\draw   (180.55,223.29) .. controls (206.15,208.89) and (234.8,231.63) .. (271.2,230.83) .. controls (307.6,230.03) and (319.88,209.29) .. (350.55,221.95) .. controls (381.21,234.62) and (381.21,281.95) .. (349.21,287.95) .. controls (317.21,293.95) and (312.38,281.62) .. (271.88,280.62) .. controls (231.38,279.62) and (201.6,302.03) .. (179.2,286.83) .. controls (156.8,271.63) and (154.95,237.69) .. (180.55,223.29) -- cycle ;
\draw    (180.27,249.16) .. controls (190.65,239.89) and (198.85,242.36) .. (205.95,250.39) ;
\draw    (183.55,247.3) .. controls (189.56,252.25) and (196.12,251.63) .. (202.13,247.3) ;
\draw    (205.07,269.16) .. controls (215.45,259.89) and (223.65,262.36) .. (230.75,270.39) ;
\draw    (208.35,267.3) .. controls (214.36,272.25) and (220.92,271.63) .. (226.93,267.3) ;
\draw    (327.47,257.96) .. controls (337.85,248.69) and (346.05,251.16) .. (353.15,259.19) ;
\draw    (330.75,256.1) .. controls (336.76,261.05) and (343.32,260.43) .. (349.33,256.1) ;
\draw    (258.67,246.76) .. controls (269.05,237.49) and (277.25,239.96) .. (284.35,247.99) ;
\draw    (261.95,244.9) .. controls (267.96,249.85) and (274.52,249.23) .. (280.53,244.9) ;
\draw    (257.87,266.76) .. controls (268.25,257.49) and (276.45,259.96) .. (283.55,267.99) ;
\draw    (261.15,264.9) .. controls (267.16,269.85) and (273.72,269.23) .. (279.73,264.9) ;
\draw  [fill={rgb, 255:red, 0; green, 0; blue, 0 }  ,fill opacity=1 ] (167.16,255.38) .. controls (167.16,254.48) and (167.9,253.74) .. (168.8,253.74) .. controls (169.7,253.74) and (170.44,254.48) .. (170.44,255.38) .. controls (170.44,256.28) and (169.7,257.02) .. (168.8,257.02) .. controls (167.9,257.02) and (167.16,256.28) .. (167.16,255.38) -- cycle ;
\draw  [fill={rgb, 255:red, 0; green, 0; blue, 0 }  ,fill opacity=1 ] (175.03,274.38) .. controls (175.03,273.48) and (175.76,272.74) .. (176.67,272.74) .. controls (177.57,272.74) and (178.3,273.48) .. (178.3,274.38) .. controls (178.3,275.28) and (177.57,276.02) .. (176.67,276.02) .. controls (175.76,276.02) and (175.03,275.28) .. (175.03,274.38) -- cycle ;
\draw  [fill={rgb, 255:red, 0; green, 0; blue, 0 }  ,fill opacity=1 ] (354.06,234.45) .. controls (354.06,233.54) and (354.8,232.81) .. (355.7,232.81) .. controls (356.6,232.81) and (357.34,233.54) .. (357.34,234.45) .. controls (357.34,235.35) and (356.6,236.08) .. (355.7,236.08) .. controls (354.8,236.08) and (354.06,235.35) .. (354.06,234.45) -- cycle ;
\draw  [fill={rgb, 255:red, 0; green, 0; blue, 0 }  ,fill opacity=1 ] (359.2,272.68) .. controls (359.2,271.78) and (359.93,271.04) .. (360.83,271.04) .. controls (361.74,271.04) and (362.47,271.78) .. (362.47,272.68) .. controls (362.47,273.58) and (361.74,274.32) .. (360.83,274.32) .. controls (359.93,274.32) and (359.2,273.58) .. (359.2,272.68) -- cycle ;
\draw  [fill={rgb, 255:red, 0; green, 0; blue, 0 }  ,fill opacity=1 ] (361.3,252.45) .. controls (361.3,251.54) and (362.03,250.81) .. (362.93,250.81) .. controls (363.84,250.81) and (364.57,251.54) .. (364.57,252.45) .. controls (364.57,253.35) and (363.84,254.08) .. (362.93,254.08) .. controls (362.03,254.08) and (361.3,253.35) .. (361.3,252.45) -- cycle ;
\draw  [color={rgb, 255:red, 126; green, 211; blue, 33 }  ,draw opacity=1 ] (269,235.5) .. controls (269,232.92) and (270.06,230.83) .. (271.36,230.83) .. controls (272.67,230.83) and (273.73,232.92) .. (273.73,235.5) .. controls (273.73,238.09) and (272.67,240.18) .. (271.36,240.18) .. controls (270.06,240.18) and (269,238.09) .. (269,235.5) -- cycle ;
\draw  [color={rgb, 255:red, 126; green, 211; blue, 33 }  ,draw opacity=1 ] (267.9,254.38) .. controls (267.9,251.07) and (269.18,248.38) .. (270.76,248.38) .. controls (272.34,248.38) and (273.63,251.07) .. (273.63,254.38) .. controls (273.63,257.7) and (272.34,260.38) .. (270.76,260.38) .. controls (269.18,260.38) and (267.9,257.7) .. (267.9,254.38) -- cycle ;
\draw  [color={rgb, 255:red, 126; green, 211; blue, 33 }  ,draw opacity=1 ] (269.01,274.62) .. controls (269.01,271.31) and (270.3,268.62) .. (271.88,268.62) .. controls (273.46,268.62) and (274.74,271.31) .. (274.74,274.62) .. controls (274.74,277.93) and (273.46,280.62) .. (271.88,280.62) .. controls (270.3,280.62) and (269.01,277.93) .. (269.01,274.62) -- cycle ;

\draw (378,246.4) node [anchor=north west][inner sep=0.75pt]    {\scalebox{1.3}{$\mathbb W$}};
\draw (142.13,262.4) node [anchor=north west][inner sep=0.75pt]    {\scalebox{1.3}{$\mathbb X$}};
\end{tikzpicture}
}}}
\Bigg)
}

\begin{figure}[h]

	\centering

\begin{align*}
\figa
\end{align*}

	\caption{. The pictorial illustration of Thm. \ref{lb54}.}
	\label{fig3}
\end{figure}


\subsection{The SF theorem in terms of (dual) fusion products}\label{lb60}

In this paper, we will prove Thm. \ref{lb54} by first proving an equivalent version of the SF theorem. Fix $\Wbb$ and $\Xbb$. Since the conformal block functors are left exact (cf. Thm. \ref{lb11}), and since any left exact linear functor from a finite $\Cbb$-linear category to $\Vect$ is representable \cite[Cor. 1.10]{DSPS19-balanced}, there exists $\bbs_\ff(\Wbb)\in\Mod(\Vbb^{\otimes R})$ such that we have an equivalence of contravariant functors
\begin{align}\label{eq119}
\Mbb\mapsto\Hom_{\Vbb^{\otimes R}}(\Mbb,\bbs_\ff(\Wbb))\qquad\simeq\qquad\Mbb\mapsto\ST^*_\ff(\Mbb\otimes\Wbb)
\end{align}
The element $\gimel\in\ST^*_\ff(\bbs_\ff(\Wbb)\otimes\Wbb)$ corresponding to $\id\in\Hom_{\Vbb^{\otimes R}}(\bbs_\ff(\Wbb),\bbs_\ff(\Wbb))$ via the above equivalence is called the \textbf{canonical conformal block}, and the pair $(\bbs_\ff(\Wbb),\gimel)$ is referred to as a \textbf{dual fusion product}. The contragredient
\begin{align*}
\boxtimes_\ff(\Wbb):=(\bbs_\ff(\Wbb))'
\end{align*}
is called a \textbf{fusion product} of $\Wbb$ along $\ff$.

By \cite[Prop. 4]{FS-coends-CFT}, the family of linear maps
\begin{align*}
\begin{gathered}
\Hom_{\Vbb^{{\otimes R}}}\big(\Mbb,\bbs_\ff(\Wbb)\big)\otimes_\Cbb\ST^*_\fg\big(\Mbb'\otimes\Xbb\big)\longrightarrow \ST^*_\fg\big(\boxtimes_\ff(\Wbb)\otimes\Xbb\big)\\
T\otimes\upchi\mapsto \upchi\circ(T^\tr\otimes\id_\Xbb)
\end{gathered}
\end{align*}
for all $\Mbb\in\Mod(\Vbb^{\otimes R})$ is a coend. (See Lem. \ref{lb53} and its proof for more explanations.) Therefore, the family of linear maps
\begin{align}\label{eq120}
\begin{gathered}
\ST^*_\ff\big(\Mbb\otimes\Wbb\big)\otimes_\Cbb\ST^*_\fg\big(\Mbb'\otimes\Xbb\big)\longrightarrow\ST^*_\fg\big(\boxtimes_\ff(\Wbb)\otimes\Xbb\big)\\ \uppsi\otimes\upchi\mapsto \upchi\circ(T_\uppsi^\tr\otimes\id_\Xbb)
\end{gathered}
\end{align}
(for all $\Mbb$) is a coend, where $T_\uppsi$ is the unique element of $\Hom_{\Vbb^{\otimes R}}(\Mbb,\bbs_\ff(\Wbb))$ corresponding to $\uppsi$ through the equivalence \eqref{eq119}.

Since \eqref{eq118} is dinatural, the universal property of the coend \eqref{eq120} guarantees a unique linear map $\Phi:\ST^*_\fg\big(\boxtimes_\ff(\Wbb)\otimes\Xbb\big)\rightarrow\ST^*_{\fx_{p_\blt}}(\Wbb\otimes\Xbb)$ such that $\Phi\circ\eqref{eq120}=\eqref{eq118}$ holds for all $\Mbb$. One checks easily that the following linear map satisfies this condition (cf. the proof of Thm. \ref{lb54}):
\begin{align}\label{eq121}
\begin{gathered}
\ST_{\fg}^*\big(\boxtimes_\ff(\Wbb)\otimes \Xbb\big)\longrightarrow \ST_{\fx_{p_\blt}}^*\big(\Wbb\otimes \Xbb\big)\\
\upchi \mapsto \MS(\gimel\otimes \upchi)\big|_{p_\blt}
\end{gathered}
\end{align}
Therefore, proving Theorem \ref{lb54} is equivalent to establishing the following version of the SF theorem, which was originally announced in the Introduction of \cite{GZ1}.

\begin{introTheorem}[\ref{lb43}]
Assume that $\Vbb$ is $C_2$-cofinite. Then \eqref{eq121} is a linear isomorphism.
\end{introTheorem}

\subsection{Main idea of the proof}

Theorem \ref{lb43} can be proved in full generality once it is established for the special case where $R=2,K=0$ and $\fg$ is the $(2,0)$-pointed sphere
\begin{align*}
\fn=(\infty,0|\Pbb^1)
\end{align*}
with local coordinates $1/\zeta$ and $\zeta$, where $\zeta$ denotes the standard coordinate of $\Cbb$. We now outline the main idea of the proof of Theorem \ref{lb43} in this key special case. Since $\Xbb\in\Mod(\Vbb^{\otimes 0})=\Vect$, we may assume that $\Xbb$ is the scalar field $\Cbb$.

It is not hard to prove that \eqref{eq121} is injective; the difficulty lies in proving the surjectivity. As we explain below, our proof of surjectivity bears a structural resemblance to the proofs of modular invariance in \cite{Miy-modular-invariance} and \cite{Hua-modular-C2}.

For each $r>0$, we let $\MD_r^\times=\{z\in\Cbb:0<|z|<r\}$. Then, by varying $p_\blt=(p_1,p_2)$, we obtain a family of surfaces over the base $\MD_r^\times\times\MD_{\rho}^\times$ for some $r,\rho>0$, which arises as the pullback of another family $\fx$ (with base manifold $\MD^\times_{r\rho}$) along the map $\MD^\times_r\times\MD^\times_\rho\rightarrow\MD^\times_{r\rho}$ sending $(p_1,p_2)$ to $p=p_1p_2$. Indeed, $\fx$ is the family obtained by sewing $\fg$ along the two outgoing marked points $x_1',x_2'$.

By what we have proved in \cite{GZ2}, the spaces of conformal blocks over the fibers of $\fx$ assemble into a vector bundle $\ST^*_\fx(\Wbb)$ over $\MD_{r\rho}^\times$, equipped with an (automatically flat) connection under which any section defined via sewing (as in \eqref{eq115}) is parallel. Consequently, any element $\uppsi_p\in\ST^*_{\fx_p}(\Wbb)$, for fixed $p$, extends to a multivalued parallel section $\uppsi$ of $\ST^*_\fx(\Wbb)$. It remains to show that $\uppsi|_{q=q_1q_2}=\MS(\gimel\otimes\upchi)$ for some $\upchi\in\ST^*_\fn(\boxtimes_\ff(\Wbb))$. (More precisely, we want to prove that the multivalued section $(q_1,q_2)\mapsto\uppsi|_{q_1q_2}$ coincides with $(q_1,q_2)\mapsto\MS(\gimel\otimes\upchi)|_{q_1,q_2}$.)

By the basic theory of linear differential equations with simple poles, the multivalued section $q\in\MD_{r\rho}^\times\mapsto\uppsi|_q$ admits a formal expansion (cf. Lem. \ref{lb21}) of the form
\begin{align*}
\uppsi=\sum_{l=0}^L\sum_{n\in\Cbb}\uppsi_{n,l}q^n(\log q)^l\qquad \text{where each }\uppsi_{n,l}:\Wbb\rightarrow\Cbb\text{ is a linear functional}
\end{align*}
Here, $L\in\Nbb$, and there exists a finite set $E\subset\Cbb$ such that $\uppsi_{n,l}=0$ whenever $n\notin E+\Nbb$. The remainder of the proof is divided into the following three steps:
\begin{enumerate}[label=(\arabic*)]
\item Show that each $\uppsi_{n,l}$ belongs to $\bbs_\ff(\Wbb)$. (Note that $\bbs_\ff(\Wbb)$ was explicitly constructed in \cite{GZ1} as a linear subspace of the full dual space $\Wbb^*$ of $\Wbb$.) This is established in Lem. \ref{lb25}.
\item Show that the formal series $\uppsi=\sum_l\sum_n\uppsi_{n,l}q^n(\log q)^l\in\bbs_\ff(\Wbb)\{q\}[\log q]$, viewed as a linear map $\boxtimes_\ff(\Wbb)\rightarrow\Cbb\{q\}[\log q]$, is invariant under the left and right actions of $\Vbb$. (Recall that $R=2$, so $\boxtimes_\ff(\Wbb)\in\Mod(\Vbb^{\otimes 2})$.) This is proved in Prop. \ref{lb34}, equivalently, Cor. \ref{lb29}.
\item Show that $q\partial_q\uppsi$ corresponds to the action of the zero-mode Virasoro operators on $\uppsi$. See Prop. \ref{lb27}, or equivalently, Cor. \ref{lb31}.
\end{enumerate}
Our proofs of all these three steps rely on the analysis of global meromorphic sections of the sheaves of VOA, specifically Thm. \ref{geometry1} and Prop. \ref{geometry10}.

Once these three steps have been proved, we define $\upchi$, as a linear functional $\boxtimes_\ff(\Wbb)\rightarrow\Cbb$, to be $\upchi=\sum_{n\in\Cbb}\uppsi_{n,0}$. By Step (3), this sum is finite when evaluated on any element of $\boxtimes_\ff(\Wbb)$. Step (2) then implies that $\upchi\in\ST^*_\fn(\boxtimes_\ff(\Wbb))$. Finally, using (3), one verifies $\uppsi|_{q=q_1q_2}=\MS(\gimel\otimes\upchi)$, finishing the proof of Thm. \ref{lb43}.  (Note that Step (1) is essential, as both (2) and (3) rely on the conclusion that each $\uppsi_{n,l}\in\bbs_\ff(\Wbb)$.)

To see how the above proof strategy parallels those of \cite{Miy-modular-invariance} and \cite{Hua-modular-C2}, consider the case where $\ff$ is the $(2,1)$-pointed sphere
\begin{align}
\fq=(\infty,0|\Pbb^1|1)
\end{align} 
equipped with the local coordinates $1/\zeta,\zeta,\zeta-1$ (where $\zeta$ is the standard coordinate of $\Cbb$). Then $\boxtimes_\fq(\Wbb)$ plays a role analogous to that of the $\wtd A^N(\Vbb)$-bimodules $\wtd A^N(\Wbb)$ (for $N\in\Nbb$) considered in \cite{Hua-associative,Hua22-Ass-IO,Hua-modular-C2}. (When $\Wbb=\Vbb$, $\boxtimes_\fq(\Vbb)$ also serves a similar function to the higher-level Zhu algebras introduced in \cite{DLM-Zhu} and used in the proof of modular invariance in \cite{Miy-modular-invariance}.) In this analogy: 
\begin{itemize}
\item Step (1) corresponds to proving that each \( \psi^N_{S;k,j} \) in \cite[Thm. 4.3]{Hua-modular-C2}, initially just a linear functional on \( U^N(\Wbb) \), descends to a linear functional on \( \wtd A^N(\Wbb) \). 
\item Step (2) resembles the argument that each \( \psi^N_{S;k,j} \) is a symmetric linear functional on $\wtd A^N(\Wbb)$ (see the formula preceding (4.9) in \cite[Thm. 4.3]{Hua-modular-C2}).
\item Step (3) reflects a compatibility condition with the Virasoro action, akin to Eq. (4.9) of \cite[Thm. 4.3]{Hua-modular-C2}.
\end{itemize}

\subsection{Relationship between conformal blocks and topological modular functors}

We close this Introduction with a brief discussion of the relationship between conformal blocks and the (topological) modular functors introduced by Lyubashenko \cite{Lyu96-Ribbon}, illustrating the significance of our SF theorems. We continue to assume that $\Vbb$ is $C_2$-cofinite.

Assume that $\fg$ in \eqref{eq122} is an $(R,K)$-pointed sphere (i.e., $C_2=\Pbb^1$), written as
\begin{align*}
\fg=(y_1',\dots,y_R'|\Pbb^1|y_1,\dots,y_K)
\end{align*}
Let $\fk T$ be a $(1,1)$-pointed torus. Let $\ff$ in \eqref{eq122} be $\fk T^{\sqcup R}$, a disjoint union of $R$ copies of  $\fk T$, which is $(R,R)$-pointed. By sewing $\ff$ and $\fg$ along their outgoing marked points using some moduli $p_\blt=(p_1,\dots,p_R)$, we obtain an $(K+R)$-pointed genus-$R$ surface 
\begin{align*}
\fx_{p_\blt}=(\MC_{p_\blt}|x_1,\dots,x_R,y_1,\dots,y_K)
\end{align*}
See Fig. \ref{fig4}.

\newcommand{\figb}{\scalebox{0.8}
{

\tikzset{every picture/.style={line width=0.75pt}} 

\begin{tikzpicture}[x=0.75pt,y=0.75pt,yscale=-1,xscale=1]

\draw   (130.53,216.78) .. controls (130.53,194.19) and (148.85,175.87) .. (171.44,175.87) .. controls (194.04,175.87) and (212.35,194.19) .. (212.35,216.78) .. controls (212.35,239.37) and (194.04,257.69) .. (171.44,257.69) .. controls (148.85,257.69) and (130.53,239.37) .. (130.53,216.78) -- cycle ;
\draw  [color={rgb, 255:red, 0; green, 0; blue, 0 }  ,draw opacity=1 ][fill={rgb, 255:red, 0; green, 0; blue, 0 }  ,fill opacity=1 ] (143.06,195.31) .. controls (143.06,194.41) and (143.8,193.67) .. (144.7,193.67) .. controls (145.6,193.67) and (146.34,194.41) .. (146.34,195.31) .. controls (146.34,196.21) and (145.6,196.95) .. (144.7,196.95) .. controls (143.8,196.95) and (143.06,196.21) .. (143.06,195.31) -- cycle ;
\draw  [color={rgb, 255:red, 0; green, 0; blue, 0 }  ,draw opacity=1 ][fill={rgb, 255:red, 0; green, 0; blue, 0 }  ,fill opacity=1 ] (137.06,212.81) .. controls (137.06,211.91) and (137.8,211.17) .. (138.7,211.17) .. controls (139.6,211.17) and (140.34,211.91) .. (140.34,212.81) .. controls (140.34,213.71) and (139.6,214.45) .. (138.7,214.45) .. controls (137.8,214.45) and (137.06,213.71) .. (137.06,212.81) -- cycle ;
\draw  [color={rgb, 255:red, 0; green, 0; blue, 0 }  ,draw opacity=1 ][fill={rgb, 255:red, 0; green, 0; blue, 0 }  ,fill opacity=1 ] (147.06,240.81) .. controls (147.06,239.91) and (147.8,239.17) .. (148.7,239.17) .. controls (149.6,239.17) and (150.34,239.91) .. (150.34,240.81) .. controls (150.34,241.71) and (149.6,242.45) .. (148.7,242.45) .. controls (147.8,242.45) and (147.06,241.71) .. (147.06,240.81) -- cycle ;
\draw  [color={rgb, 255:red, 0; green, 0; blue, 0 }  ,draw opacity=1 ][fill={rgb, 255:red, 0; green, 0; blue, 0 }  ,fill opacity=1 ] (193.56,192.81) .. controls (193.56,191.91) and (194.3,191.17) .. (195.2,191.17) .. controls (196.1,191.17) and (196.84,191.91) .. (196.84,192.81) .. controls (196.84,193.71) and (196.1,194.45) .. (195.2,194.45) .. controls (194.3,194.45) and (193.56,193.71) .. (193.56,192.81) -- cycle ;
\draw  [color={rgb, 255:red, 0; green, 0; blue, 0 }  ,draw opacity=1 ][fill={rgb, 255:red, 0; green, 0; blue, 0 }  ,fill opacity=1 ] (201.56,218.81) .. controls (201.56,217.91) and (202.3,217.17) .. (203.2,217.17) .. controls (204.1,217.17) and (204.84,217.91) .. (204.84,218.81) .. controls (204.84,219.71) and (204.1,220.45) .. (203.2,220.45) .. controls (202.3,220.45) and (201.56,219.71) .. (201.56,218.81) -- cycle ;
\draw  [color={rgb, 255:red, 0; green, 0; blue, 0 }  ,draw opacity=1 ][fill={rgb, 255:red, 0; green, 0; blue, 0 }  ,fill opacity=1 ] (192.06,242.81) .. controls (192.06,241.91) and (192.8,241.17) .. (193.7,241.17) .. controls (194.6,241.17) and (195.34,241.91) .. (195.34,242.81) .. controls (195.34,243.71) and (194.6,244.45) .. (193.7,244.45) .. controls (192.8,244.45) and (192.06,243.71) .. (192.06,242.81) -- cycle ;
\draw   (229.04,175.66) .. controls (225.63,168.53) and (233.45,157.69) .. (246.51,151.45) .. controls (259.56,145.21) and (272.91,145.93) .. (276.32,153.06) .. controls (279.73,160.19) and (271.91,171.03) .. (258.85,177.27) .. controls (245.79,183.52) and (232.44,182.79) .. (229.04,175.66) -- cycle ;
\draw    (240.82,171.39) .. controls (248.2,160.29) and (257.14,158.04) .. (266.76,160) ;
\draw    (243.49,168.6) .. controls (251.04,169.03) and (257.37,165.5) .. (261.97,159.77) ;
\draw  [color={rgb, 255:red, 0; green, 0; blue, 0 }  ,draw opacity=1 ][fill={rgb, 255:red, 0; green, 0; blue, 0 }  ,fill opacity=1 ] (231.8,174.28) .. controls (231.41,173.47) and (231.76,172.49) .. (232.57,172.1) .. controls (233.39,171.71) and (234.36,172.06) .. (234.75,172.87) .. controls (235.14,173.69) and (234.8,174.66) .. (233.98,175.05) .. controls (233.17,175.44) and (232.19,175.1) .. (231.8,174.28) -- cycle ;
\draw  [color={rgb, 255:red, 0; green, 0; blue, 0 }  ,draw opacity=1 ][fill={rgb, 255:red, 0; green, 0; blue, 0 }  ,fill opacity=1 ] (269.03,155.93) .. controls (268.64,155.12) and (268.98,154.14) .. (269.8,153.75) .. controls (270.61,153.36) and (271.59,153.71) .. (271.98,154.52) .. controls (272.37,155.34) and (272.02,156.31) .. (271.21,156.7) .. controls (270.39,157.09) and (269.42,156.75) .. (269.03,155.93) -- cycle ;
\draw   (238.66,215.02) .. controls (237.71,207.17) and (248.59,199.4) .. (262.95,197.66) .. controls (277.32,195.91) and (289.74,200.86) .. (290.69,208.7) .. controls (291.64,216.55) and (280.77,224.32) .. (266.4,226.07) .. controls (252.03,227.81) and (239.62,222.87) .. (238.66,215.02) -- cycle ;
\draw    (251.2,214.74) .. controls (261.74,206.57) and (270.93,207.29) .. (279.42,212.23) ;
\draw    (254.62,212.95) .. controls (261.64,215.76) and (268.76,214.44) .. (274.96,210.48) ;
\draw  [color={rgb, 255:red, 0; green, 0; blue, 0 }  ,draw opacity=1 ][fill={rgb, 255:red, 0; green, 0; blue, 0 }  ,fill opacity=1 ] (242.28,215.03) .. controls (242.17,214.13) and (242.81,213.32) .. (243.71,213.21) .. controls (244.61,213.1) and (245.42,213.74) .. (245.53,214.64) .. controls (245.64,215.53) and (245,216.35) .. (244.1,216.46) .. controls (243.21,216.57) and (242.39,215.93) .. (242.28,215.03) -- cycle ;
\draw  [color={rgb, 255:red, 0; green, 0; blue, 0 }  ,draw opacity=1 ][fill={rgb, 255:red, 0; green, 0; blue, 0 }  ,fill opacity=1 ] (282.86,209.1) .. controls (282.75,208.2) and (283.39,207.39) .. (284.29,207.28) .. controls (285.19,207.17) and (286,207.81) .. (286.11,208.7) .. controls (286.22,209.6) and (285.58,210.42) .. (284.68,210.53) .. controls (283.79,210.63) and (282.97,210) .. (282.86,209.1) -- cycle ;
\draw   (224.66,252.92) .. controls (226.45,245.22) and (239.33,241.64) .. (253.42,244.92) .. controls (267.52,248.21) and (277.49,257.11) .. (275.7,264.81) .. controls (273.91,272.5) and (261.03,276.08) .. (246.93,272.8) .. controls (232.84,269.52) and (222.86,260.62) .. (224.66,252.92) -- cycle ;
\draw    (236.53,256.95) .. controls (249.23,252.89) and (257.61,256.71) .. (263.9,264.26) ;
\draw    (240.36,256.43) .. controls (245.99,261.49) and (253.13,262.68) .. (260.31,261.08) ;
\draw  [color={rgb, 255:red, 0; green, 0; blue, 0 }  ,draw opacity=1 ][fill={rgb, 255:red, 0; green, 0; blue, 0 }  ,fill opacity=1 ] (228.05,254.17) .. controls (228.26,253.29) and (229.14,252.74) .. (230.02,252.95) .. controls (230.9,253.15) and (231.44,254.03) .. (231.24,254.91) .. controls (231.03,255.79) and (230.15,256.34) .. (229.27,256.13) .. controls (228.39,255.93) and (227.85,255.05) .. (228.05,254.17) -- cycle ;
\draw  [color={rgb, 255:red, 0; green, 0; blue, 0 }  ,draw opacity=1 ][fill={rgb, 255:red, 0; green, 0; blue, 0 }  ,fill opacity=1 ] (268.21,262.49) .. controls (268.41,261.61) and (269.29,261.07) .. (270.17,261.27) .. controls (271.05,261.48) and (271.6,262.36) .. (271.4,263.24) .. controls (271.19,264.12) and (270.31,264.66) .. (269.43,264.46) .. controls (268.55,264.25) and (268,263.37) .. (268.21,262.49) -- cycle ;
\draw [color={rgb, 255:red, 74; green, 144; blue, 226 }  ,draw opacity=1 ] [dash pattern={on 1.5pt off 1.5pt}]  (206.3,187.32) -- (224.14,179.44) ;
\draw [shift={(225.97,178.63)}, rotate = 156.16] [color={rgb, 255:red, 74; green, 144; blue, 226 }  ,draw opacity=1 ][line width=0.75]    (6.56,-1.97) .. controls (4.17,-0.84) and (1.99,-0.18) .. (0,0) .. controls (1.99,0.18) and (4.17,0.84) .. (6.56,1.97)   ;
\draw [shift={(204.47,188.13)}, rotate = 336.16] [color={rgb, 255:red, 74; green, 144; blue, 226 }  ,draw opacity=1 ][line width=0.75]    (6.56,-1.97) .. controls (4.17,-0.84) and (1.99,-0.18) .. (0,0) .. controls (1.99,0.18) and (4.17,0.84) .. (6.56,1.97)   ;
\draw [color={rgb, 255:red, 74; green, 144; blue, 226 }  ,draw opacity=1 ] [dash pattern={on 1.5pt off 1.5pt}]  (216.33,218.9) -- (234.09,216.45) ;
\draw [shift={(236.07,216.17)}, rotate = 172.12] [color={rgb, 255:red, 74; green, 144; blue, 226 }  ,draw opacity=1 ][line width=0.75]    (6.56,-1.97) .. controls (4.17,-0.84) and (1.99,-0.18) .. (0,0) .. controls (1.99,0.18) and (4.17,0.84) .. (6.56,1.97)   ;
\draw [shift={(214.35,219.18)}, rotate = 352.12] [color={rgb, 255:red, 74; green, 144; blue, 226 }  ,draw opacity=1 ][line width=0.75]    (6.56,-1.97) .. controls (4.17,-0.84) and (1.99,-0.18) .. (0,0) .. controls (1.99,0.18) and (4.17,0.84) .. (6.56,1.97)   ;
\draw [color={rgb, 255:red, 74; green, 144; blue, 226 }  ,draw opacity=1 ] [dash pattern={on 1.5pt off 1.5pt}]  (202.57,247.67) -- (220.15,252.81) ;
\draw [shift={(222.07,253.37)}, rotate = 196.3] [color={rgb, 255:red, 74; green, 144; blue, 226 }  ,draw opacity=1 ][line width=0.75]    (6.56,-1.97) .. controls (4.17,-0.84) and (1.99,-0.18) .. (0,0) .. controls (1.99,0.18) and (4.17,0.84) .. (6.56,1.97)   ;
\draw [shift={(200.65,247.11)}, rotate = 16.3] [color={rgb, 255:red, 74; green, 144; blue, 226 }  ,draw opacity=1 ][line width=0.75]    (6.56,-1.97) .. controls (4.17,-0.84) and (1.99,-0.18) .. (0,0) .. controls (1.99,0.18) and (4.17,0.84) .. (6.56,1.97)   ;
\draw  [color={rgb, 255:red, 0; green, 0; blue, 0 }  ,draw opacity=1 ][fill={rgb, 255:red, 0; green, 0; blue, 0 }  ,fill opacity=1 ] (139.73,229.48) .. controls (139.73,228.57) and (140.46,227.84) .. (141.37,227.84) .. controls (142.27,227.84) and (143,228.57) .. (143,229.48) .. controls (143,230.38) and (142.27,231.11) .. (141.37,231.11) .. controls (140.46,231.11) and (139.73,230.38) .. (139.73,229.48) -- cycle ;

\draw (117.2,184.08) node [anchor=north west][inner sep=0.75pt]    {$y_1$};
\draw (109.2,208.58) node [anchor=north west][inner sep=0.75pt]    {$y_2$};
\draw (120,246.58) node [anchor=north west][inner sep=0.75pt]    {$y_K$};
\draw (154.98,268.08) node [anchor=north west][inner sep=0.75pt]    {\scalebox{1.25}{$\fg$}};
\draw (236.98,280.25) node [anchor=north west][inner sep=0.75pt]    {\scalebox{1.25}{$\ff$}};
\draw (280.58,140.88) node [anchor=north west][inner sep=0.75pt]    {$x_1$};
\draw (294.58,199.48) node [anchor=north west][inner sep=0.75pt]    {$x_2$};
\draw (278.18,255.41) node [anchor=north west][inner sep=0.75pt]    {$x_R$};

\end{tikzpicture}
}
}

\begin{figure}[h]

	\centering

\figb

	\caption{. Sewing $\ff$ with $\fg$.}
	\label{fig4}
\end{figure} 

Let $\Xbb_1,\dots,\Xbb_K\in\Mod(\Vbb)$, and associate $\Xbb_1\otimes\cdots\otimes\Xbb_K\in\Mod(\Vbb^{\otimes K})$ to the ordered marked points $y_1,\dots,y_K$. Associate $\Vbb^{\otimes R}$ to $x_1,\dots,x_R$. Let
\begin{align*}
\fk Y_{p_\blt}=(\MC_{p_\blt}|y_1,\dots,y_K)
\end{align*}
By the propagation of conformal blocks (cf. \cite[Cor. 2.44]{GZ1}), we have a linear isomorphism of spaces of conformal blocks
\begin{align*}
\ST^*_{\fx_{p_\blt}}(\Vbb^{\otimes R}\otimes\Xbb_1\otimes\cdots\otimes\Xbb_K)\simeq\ST^*_{\fy_{p_\blt}}(\Xbb_1\otimes\cdots\otimes\Xbb_K)
\end{align*}
defined by inserting the vacuum vector $\idt$ into each tensor component $\Vbb$. Therefore, by Thm. \ref{lb43}, the sewing map induces a linear isomorphism
\begin{gather*}
\ST^*_{\fy_{p_\blt}}(\Xbb_1\otimes\cdots\otimes\Xbb_K)\simeq\ST^*_\fg\big(\boxtimes_\ff(\Vbb^{\otimes R})\otimes\Xbb_1\otimes\cdots\otimes\Xbb_K\big)
\end{gather*}
Since the fusion product of a disjoint union is the tensor product of the fusion products on each component (cf. Thm. \ref{lb16}), we have a canonical equivalence $\boxtimes_\ff(\Vbb^{\otimes R})\simeq (\boxtimes_{\fk T}\Vbb)^{\otimes R}$. Therefore, the above equivalence becomes
\begin{align}\label{eq123}
\ST^*_{\fy_{p_\blt}}(\Xbb_1\otimes\cdots\otimes\Xbb_K)\simeq\ST^*_\fg\big((\boxtimes_{\fk T}\Vbb)^{\otimes R}\otimes\Xbb_1\otimes\cdots\otimes\Xbb_K\big)
\end{align}

Let us express the RHS of \eqref{eq123} in terms of the monoidal structure of $\Mod(\Vbb)$. Define $(1,2)$-pointed and $(2,1)$-pointed spheres
\begin{align*}
\fp=(1|\Pbb^1|\infty,0)\qquad\fq=(\infty,0|\Pbb^1|1)
\end{align*}
Then $\boxtimes_\fp$ gives the \textbf{tensor bifunctor \pmb{$\boxdot$} of \pmb{$\Mod(\Vbb)$}}, i.e., for each $\Wbb_1,\Wbb_2\in\Mod(\Vbb)$,
\begin{align*}
\Wbb_1\boxdot\Wbb_2=\boxtimes_\fp(\Wbb_1\otimes\Wbb_2)
\end{align*}
Therefore, \eqref{eq123} can be rewritten as
\begin{align}\label{eq125}
\ST^*_{\fy_{p_\blt}}(\Xbb_1\otimes\cdots\otimes\Xbb_K)\simeq\Hom_\Vbb\big((\boxtimes_{\fk T}\Vbb)^{\boxdot R}\boxdot\Xbb_1\boxdot\cdots\boxdot\Xbb_K,\Vbb'\big)
\end{align}

To proceed, we give a categorical description of $\boxtimes_{\fk T}\Vbb$. By \cite{McR-deligne}, the bi-exact functor
\begin{align*}
(\Wbb_1,\Wbb_2)\in\Mod(\Vbb)\times\Mod(\Vbb)\rightarrow \Wbb_1\otimes_\Cbb\Wbb_2\in\Mod(\Vbb^{\otimes 2})
\end{align*} 
is a Deligne product of two copies of $\Mod(\Vbb)$. Therefore, since $\boxtimes_\fp:\Mod(\Vbb^{\otimes2})\rightarrow\Mod(\Vbb)$ is right exact (Thm. \ref{lb11}), it is the unique lift of the (right exact) bifunctor $\boxdot:\Mod(\Vbb)\times\Mod(\Vbb)\rightarrow\Mod(\Vbb)$ to $\Mod(\Vbb^{\otimes2})$. In short, we have
\begin{align*}
\boxdot=\boxtimes_\fp\qquad\text{on }\Mod(\Vbb^{\otimes2})
\end{align*} 
Therefore, by the transitivity of fusion products (Thm. \ref{lb51}), another version of the SF theorem, we have the following canonical isomorphism, also illustrated in Fig. \ref{fig5}.

\begin{align}\label{eq129}
\boxtimes_{\fk T}\Vbb\simeq \boxtimes_\fp(\boxtimes_\fq\Vbb)=\boxdot(\boxtimes_\fq\Vbb)
\end{align}

\newcommand{\figc}{
\vcenter{\hbox{\scalebox{0.8}{

\tikzset{every picture/.style={line width=0.75pt}} 

\begin{tikzpicture}[x=0.75pt,y=0.75pt,yscale=-1,xscale=1]

\draw    (90.97,270.88) .. controls (127.64,269.55) and (116.44,247.55) .. (147.64,247.55) .. controls (178.84,247.55) and (164.97,270.88) .. (198.97,271.55) ;
\draw    (90.3,289.55) .. controls (126.97,288.21) and (115.1,308.88) .. (146.3,308.88) .. controls (177.5,308.88) and (164.3,289.55) .. (198.3,290.21) ;
\draw   (84.11,279.86) .. controls (84.11,274.52) and (86.56,270.18) .. (89.58,270.18) .. controls (92.6,270.18) and (95.05,274.52) .. (95.05,279.86) .. controls (95.05,285.21) and (92.6,289.55) .. (89.58,289.55) .. controls (86.56,289.55) and (84.11,285.21) .. (84.11,279.86) -- cycle ;
\draw   (192.84,280.53) .. controls (192.84,275.18) and (195.28,270.85) .. (198.3,270.85) .. controls (201.32,270.85) and (203.77,275.18) .. (203.77,280.53) .. controls (203.77,285.88) and (201.32,290.21) .. (198.3,290.21) .. controls (195.28,290.21) and (192.84,285.88) .. (192.84,280.53) -- cycle ;
\draw    (118.73,281.9) .. controls (140.56,262.92) and (157.79,267.98) .. (172.73,284.43) ;
\draw    (125.62,278.1) .. controls (139.41,289.49) and (150.9,286.96) .. (164.68,278.1) ;
\draw [color={rgb, 255:red, 126; green, 211; blue, 33 }  ,draw opacity=1 ]   (231.73,280.68) -- (212.73,280.68) ;
\draw [shift={(210.73,280.68)}, rotate = 360] [color={rgb, 255:red, 126; green, 211; blue, 33 }  ,draw opacity=1 ][line width=0.75]    (6.56,-1.97) .. controls (4.17,-0.84) and (1.99,-0.18) .. (0,0) .. controls (1.99,0.18) and (4.17,0.84) .. (6.56,1.97)   ;
\draw [color={rgb, 255:red, 126; green, 211; blue, 33 }  ,draw opacity=1 ]   (78.23,280.18) -- (59.23,280.18) ;
\draw [shift={(57.23,280.18)}, rotate = 360] [color={rgb, 255:red, 126; green, 211; blue, 33 }  ,draw opacity=1 ][line width=0.75]    (6.56,-1.97) .. controls (4.17,-0.84) and (1.99,-0.18) .. (0,0) .. controls (1.99,0.18) and (4.17,0.84) .. (6.56,1.97)   ;

\draw (235.5,272) node [anchor=north west][inner sep=0.75pt]    {\scalebox{1.2}{$\Vbb$}};
\draw (15,272) node [anchor=north west][inner sep=0.75pt]    {\scalebox{1.2}{$\boxtimes_{\fk T}\Vbb$}};

\end{tikzpicture}
}}}
}

\newcommand{\figd}{
\vcenter{\hbox{\scalebox{0.8}{
\tikzset{every picture/.style={line width=0.75pt}} 

\begin{tikzpicture}[x=0.75pt,y=0.75pt,yscale=-1,xscale=1]

\draw   (520.34,278.03) .. controls (520.34,272.68) and (522.78,268.35) .. (525.8,268.35) .. controls (528.82,268.35) and (531.27,272.68) .. (531.27,278.03) .. controls (531.27,283.38) and (528.82,287.71) .. (525.8,287.71) .. controls (522.78,287.71) and (520.34,283.38) .. (520.34,278.03) -- cycle ;
\draw    (470.14,247.68) .. controls (502.14,248.68) and (491.3,268.35) .. (525.8,268.35) ;
\draw    (472.14,309.68) .. controls (505.14,308.68) and (491.3,287.71) .. (525.8,287.71) ;
\draw   (365.34,279.03) .. controls (365.34,273.68) and (367.78,269.35) .. (370.8,269.35) .. controls (373.82,269.35) and (376.27,273.68) .. (376.27,279.03) .. controls (376.27,284.38) and (373.82,288.71) .. (370.8,288.71) .. controls (367.78,288.71) and (365.34,284.38) .. (365.34,279.03) -- cycle ;
\draw    (421.64,248.68) .. controls (389.64,248.68) and (402.14,269.18) .. (368.3,269.85) ;
\draw    (423.14,310.18) .. controls (391.14,310.18) and (404.64,288.05) .. (370.8,288.71) ;
\draw   (416.17,258.36) .. controls (416.17,253.02) and (418.62,248.68) .. (421.64,248.68) .. controls (424.66,248.68) and (427.1,253.02) .. (427.1,258.36) .. controls (427.1,263.71) and (424.66,268.05) .. (421.64,268.05) .. controls (418.62,268.05) and (416.17,263.71) .. (416.17,258.36) -- cycle ;
\draw   (417.67,300.5) .. controls (417.67,295.15) and (420.12,290.82) .. (423.14,290.82) .. controls (426.16,290.82) and (428.6,295.15) .. (428.6,300.5) .. controls (428.6,305.85) and (426.16,310.18) .. (423.14,310.18) .. controls (420.12,310.18) and (417.67,305.85) .. (417.67,300.5) -- cycle ;
\draw   (464.67,257.36) .. controls (464.67,252.02) and (467.12,247.68) .. (470.14,247.68) .. controls (473.16,247.68) and (475.6,252.02) .. (475.6,257.36) .. controls (475.6,262.71) and (473.16,267.05) .. (470.14,267.05) .. controls (467.12,267.05) and (464.67,262.71) .. (464.67,257.36) -- cycle ;
\draw   (466.67,300) .. controls (466.67,294.65) and (469.12,290.32) .. (472.14,290.32) .. controls (475.16,290.32) and (477.6,294.65) .. (477.6,300) .. controls (477.6,305.35) and (475.16,309.68) .. (472.14,309.68) .. controls (469.12,309.68) and (466.67,305.35) .. (466.67,300) -- cycle ;
\draw    (423.14,290.82) .. controls (396.14,290.18) and (394.64,267.18) .. (421.64,268.05) ;
\draw    (470.14,267.05) .. controls (496.64,267.18) and (496.64,289.18) .. (472.14,290.32) ;
\draw [color={rgb, 255:red, 126; green, 211; blue, 33 }  ,draw opacity=1 ]   (558.73,278.68) -- (539.73,278.68) ;
\draw [shift={(537.73,278.68)}, rotate = 360] [color={rgb, 255:red, 126; green, 211; blue, 33 }  ,draw opacity=1 ][line width=0.75]    (6.56,-1.97) .. controls (4.17,-0.84) and (1.99,-0.18) .. (0,0) .. controls (1.99,0.18) and (4.17,0.84) .. (6.56,1.97)   ;
\draw [color={rgb, 255:red, 126; green, 211; blue, 33 }  ,draw opacity=1 ]   (456.73,260.18) -- (437.73,260.18) ;
\draw [shift={(435.73,260.18)}, rotate = 360] [color={rgb, 255:red, 126; green, 211; blue, 33 }  ,draw opacity=1 ][line width=0.75]    (6.56,-1.97) .. controls (4.17,-0.84) and (1.99,-0.18) .. (0,0) .. controls (1.99,0.18) and (4.17,0.84) .. (6.56,1.97)   ;
\draw [color={rgb, 255:red, 126; green, 211; blue, 33 }  ,draw opacity=1 ]   (456.73,304.18) -- (437.73,304.18) ;
\draw [shift={(435.73,304.18)}, rotate = 360] [color={rgb, 255:red, 126; green, 211; blue, 33 }  ,draw opacity=1 ][line width=0.75]    (6.56,-1.97) .. controls (4.17,-0.84) and (1.99,-0.18) .. (0,0) .. controls (1.99,0.18) and (4.17,0.84) .. (6.56,1.97)   ;
\draw [color={rgb, 255:red, 126; green, 211; blue, 33 }  ,draw opacity=1 ]   (359.73,279.68) -- (340.73,279.68) ;
\draw [shift={(338.73,279.68)}, rotate = 360] [color={rgb, 255:red, 126; green, 211; blue, 33 }  ,draw opacity=1 ][line width=0.75]    (6.56,-1.97) .. controls (4.17,-0.84) and (1.99,-0.18) .. (0,0) .. controls (1.99,0.18) and (4.17,0.84) .. (6.56,1.97)   ;

\draw (560,271) node [anchor=north west][inner sep=0.75pt]    {\scalebox{1.2}{$\Vbb$}};
\draw (430,228) node [anchor=north west][inner sep=0.75pt]    {\scalebox{1.2}{$\boxtimes_\fq\Vbb$}};
\draw (255,269.64) node [anchor=north west][inner sep=0.75pt]    {\scalebox{1.2}{$\boxtimes_\fp(\boxtimes_\fq\Vbb)$}};

\end{tikzpicture}
}}}}
\begin{figure}[h]

	\centering
\begin{gather*}
\figc\\
\qquad\rotatebox{-90}{\scalebox{1.3}{$\simeq$}~~}\\
\figd
\end{gather*}

	\caption{. The transitivity of fusion products $\boxtimes_{\fk T}\Vbb\simeq\boxtimes_\fp(\boxtimes_\fq\Vbb)$.}
	\label{fig5}
\end{figure} 

It remains to give a categorical description of $\boxtimes_\fq\Vbb$. By Sec. \ref{lb60}, $\bbs_\fq\Vbb$ is the unique object in $\Mod(\Vbb^{\otimes2})$ implementing the natural equivalence 
\begin{align}
\Hom_{\Vbb^{\otimes 2}}(\Mbb,\bbs_\fq\Vbb)\simeq \ST^*_\fq(\Mbb\otimes\Vbb)
\end{align}
for all $\Mbb\in\Mod(\Vbb^{\otimes2})$. Equivalently,
\begin{align}
\Hom_{\Vbb^{\otimes 2}}(\boxtimes_\fq\Vbb,\Mbb)\simeq \ST^*_\fq(\Mbb'\otimes\Vbb)
\end{align}
If $\Mbb=\Xbb\otimes \Ybb'$ where $\Xbb,\Ybb\in\Mod(\Vbb)$, by the propagation of conformal blocks (cf. \cite[Cor. 2.44]{GZ1}), we have
\begin{align*}
\ST^*_\fq(\Xbb'\otimes\Ybb\otimes\Vbb)\simeq\Hom_\Vbb(\Ybb,\Xbb)
\end{align*}
and hence
\begin{align}\label{eq128}
\Hom_{\Vbb^{\otimes 2}}(\boxtimes_\fq\Vbb,\Xbb\otimes\Ybb')\simeq \Hom_\Vbb(\Ybb,\Xbb)
\end{align}
That is, $\boxtimes_\fq\Vbb$ is the unique object in the Deligne product of two copies of $\Mod(\Vbb)$ representing the left exact profunctor $(\Xbb,\Ybb)\mapsto\Hom_\Vbb(\Ybb,\Xbb)$. Therefore:
\begin{thm}\label{lb61}
For each $\Xbb\in\Mod(\Vbb)$, let $\pi_\Xbb:\boxtimes_\fq\Vbb\rightarrow\Xbb\otimes\Xbb'$ be the morphism representing the identity map $\id_\Xbb\in\Hom_\Vbb(\Xbb,\Xbb)$ in the sense of \eqref{eq128}. Then the dinatural transformation $(\pi_\Xbb)_{\Xbb\in\Mod(\Vbb)}$ is an end in $\Mod(\Vbb^{\otimes2})$. In short, we have
\begin{align*}
\boxtimes_\fq\Vbb\simeq\int_{\Xbb\in\Mod(\Vbb)}\Xbb\otimes_\Cbb\Xbb'
\end{align*}
\end{thm}

An alternative proof of this theorem will be given in \cite{GZ5}.

\begin{proof}
We identify $\Mod(\Vbb)$ with its opposite category $\Mod(\Vbb)^\opp$ by identifying $\Xbb'$ with $\Xbb^\opp$. Therefore, it suffices to prove the claim that the end $\End:=\int_{\Xbb\in\Mod(\Vbb)}\Xbb\otimes^{\mathrm{Del}}\Xbb^\opp$ represents the profunctor $(\Xbb,\Ybb)\mapsto\Hom_\Vbb(\Ybb,\Xbb)$, where $\otimes^{\mathrm{Del}}$ denotes the Deligne product. To prove this, one may assume that $\Mod(\Vbb)=\Mod(A)$ where $A$ is a unital finite dimensional $\Cbb$-algebra. In this case, the claim follows from \cite[Cor. 2.9]{FSS20}.
\end{proof}

It follows from \eqref{eq129} and Thm. \ref{lb61} that
\begin{subequations}\label{eq130}
\begin{align}
\boxtimes_{\fk T}(\Vbb)\simeq \mbb L
\end{align}
where $\mbb L$ is the \textbf{Lyubashenko construction}
\begin{align}
\mbb L:=\boxdot\Big(\int_{\Xbb\in\Mod(\Vbb)}\Xbb\otimes\Xbb'\Big)
\end{align}
\end{subequations}
defined in the context of Grothendieck-Verdier categories in \cite{MW23-cyclic-framed,BW22-MFFH}. In the case of finite tensor categories, this construction is introduced by Lyubashenko \cite{Lyu95-Invariants,Lyu96-Ribbon} in the form of the \textbf{Lyubashenko coend}
\begin{align*}
\mbb L=\int^{\Xbb\in\Mod(\Vbb)}\Xbb'\boxdot\Xbb
\end{align*}
(When $\Vbb$ is strongly finite, and if the conjectured rigidity of $\Mod(\Vbb)$ holds, then $\Mod(\Vbb)$ is modular by \cite{McR21-rational}. In that case, $\int_{\Xbb\in\Mod(\Vbb)}\Xbb\otimes_\Cbb\Xbb'$ is self-dual due to \cite{Shi-unimodular}, and hence $\int_{\Xbb\in\Mod(\Vbb)}\Xbb\otimes\Xbb'\simeq\int^{\Xbb\in\Mod(\Vbb)}\Xbb'\otimes\Xbb$. Since $\boxdot$ is a left adjoint---specifically, it is the left adjoint of $\Xbb\in\Mod(\Vbb)\rightarrow\bbs_\fq(\Xbb')\in\Mod(\Vbb^{\otimes2})$---the functor $\boxdot$ commutes with coends. Therefore, the above two definitions of $\mbb L$ are equivalent.)

Thanks to \eqref{eq130}, the isomorphism \eqref{eq125} becomes
\begin{align}\label{eq126}
\ST^*_{\fy_{p_\blt}}(\Xbb_1\otimes\cdots\otimes\Xbb_K)\simeq\Hom_\Vbb\big(\mbb L^{\boxdot R}\boxdot\Xbb_1\boxdot\cdots\boxdot\Xbb_K,\Vbb'\big)
\end{align}
where the LHS is the space of genus-$R$ conformal blocks (with input modules $\Xbb_1,\dots,\Xbb_K$), and the RHS coincides with Lyubashenko's construction of modular functors (cf. \cite[Sec. 8.2]{Lyu96-Ribbon} or \cite[Cor. 8.1]{BW22-MFFH}).

\subsection*{Acknowledgment}

We would like to thank Chiara Damiolini, Jishen Du, J\"urgen Fuchs, Yi-Zhi Huang, Liang Kong, Robert McRae, Yuto Moriwaki, Ingo Runkel, Christoph Schweigert, Yilong Wang, Lukas Woike, and Hao Zheng for helpful discussions. We are grateful to the referees for their valuable comments and reports on the manuscript.

The first author is grateful to Zhipeng Yang for his warm hospitality and the support of Yunnan Key Laboratory of Modern Analytical Mathematics and Applications. He is supported by NSFC Grant 12401159. The second author would like to thank Wenjun Niu for valuable discussions and for his hospitality at the Perimeter Institute for Theoretical Physics.

\section{Preliminaries}

\subsection{Notation}\label{lb1}

In this paper, we continue to use the notations listed in \cite[Sec. 1.1]{GZ2}. In addition to these notations, we also adopt the following notations and conventions.

If $W$ is a vector space, and $q_1,\dots,q_R$ are (mutually commuting) formal variables, then $\pmb{W\{q_\blt\}}=W\{q_1,\dots,q_R\}$ denotes the set of all
\begin{align*}
\sum_{n_\blt\in E+\Cbb^R}w_{n_\blt}q_1^{n_1}\cdots q_R^{n_R}
\end{align*}
where $E$ is a finite subset of $\Cbb^R$, and each $w_{n_\blt}\in W$.

Throughout this paper, we fix an ($\Nbb$-graded) $C_2$-cofinite vertex operator algebra (VOA) $\Vbb=\bigoplus_{n\in \Nbb}\Vbb(n)$ with conformal vector $\cbf$. For each $N\in\Nbb$,  we let $\pmb{\Mod(\Vbb^{\otimes N})}$ be the category of grading-restricted (generalized) $\Vbb^{\otimes N}$-module, which is an abelian category by \cite{Hua-projectivecover}. 

For each $r\in(0,+\infty]$, we let 
\begin{subequations}\label{eq127}
\begin{gather}
\MD_r=\{z\in\Cbb:|z|<r\}\qquad\MD_r^\times=\MD_r\setminus\{0\}\qquad\wht\MD_r^\times=\text{the universal cover of }\MD_r^\times
\end{gather}
For each $r_1,\dots,r_N\in(0,+\infty]$, we let
\begin{gather}
\MD_{r_\blt}=\MD_{r_1}\times\cdots\times\MD_{r_N} \qquad \MD_{r_\blt}^\times=\MD_{r_1}^\times\times\cdots\times\MD_{r_N}^\times \qquad \wht\MD_{r_\blt}^\times=\wht\MD_{r_1}^\times\times\cdots\times\wht\MD_{r_N}^\times
\end{gather}
\end{subequations}

Recall from \cite[Sec. 1.1]{GZ2} that if $\Wbb\in\Mod(\Vbb^{\otimes N})$ and $v\in\Vbb$, the $i$-th vertex operator is $Y_i(v,z)=\sum_{n\in\Zbb}Y_i(v)_nz^{-n-1}$ is $Y(\idt\otimes\cdots\otimes v\otimes\cdots\otimes\idt,z)$ where $v$ is at the $i$-th component, and $L_i(n)=Y_i(\cbf)_{n-1}$. We also write
\begin{subequations}\label{eq48}
\begin{align}
Y_i'(v,z)=Y_i(\MU(\upgamma_z)v,z^{-1})
\end{align}
where $\MU(\upgamma_z)=e^{zL(1)}(-z^{-2})^{L(0)}$. In particular, $\MU(\upgamma_1)=e^{L(1)}(-1)^{L(0)}$. Clearly $\MU(\upgamma_z)^{-1}=\MU(\upgamma_{1/z})$, and hence
\begin{align}
Y_i(v,z)=Y_i'(\MU(\upgamma_z)v,z^{-1})
\end{align}
\end{subequations}
For each $k\in\Zbb$, let
\begin{gather}\label{eq50}
Y_i(v)_{(k)}=\Res_{z=0}Y_i(z^{k+L(0)-1}v,z)dz\qquad Y_i'(v)_{(k)}=\Res_{z=0}Y'_i(z^{k+L(0)-1}v,z)dz
\end{gather}

If $\Wbb\in\Mod(\Vbb^{\otimes N})$ and $\lambda_1,\dots,\lambda_N\in\Cbb$, then $\Wbb_{[\lambda_\blt]}$ is the subspace of all $w\in\Wbb$ such that for all $1\leq i\leq N$, $w$ is a generalized eigenvector of $L_i(0)$ with eigenvalue $\lambda_i$. The finite-dimensional subspace $\Wbb_{[\leq\lambda_\blt]}$ is defined to be the direct sum of all $\Wbb_{[\mu_\blt]}$ where $\Re(\mu_i)\leq \Re(\lambda_i)$ for all $1\leq i\leq N$. Then the contragredient $\Vbb^{\otimes N}$-module of $\Wbb$, as a vector space, is
\begin{align*}
\Wbb'=\bigoplus_{\lambda_\blt\in\Cbb^N}(\Wbb_{[\lambda_\blt]})^*
\end{align*}
Then for each $w\in\Wbb,w'\in\Wbb$ we clearly have
\begin{align}\label{eq49}
\bk{Y_i(v,z)w,w'}=\bk{w,Y_i'(v,z)w'}
\end{align}
The algebraic completion of $\Wbb$ is 
\begin{align*}
\ovl\Wbb=(\Wbb')^*=\prod_{\lambda_\blt\in\Cbb^N}\Wbb_{[\lambda_\blt]}
\end{align*}
We let $P_{\lambda_\blt}$ and $P_{\leq\lambda_\blt}$ be the projections of $\ovl\Wbb$ onto $\Wbb_{[\lambda_\blt]}$ and $\Wbb_{\leq[\lambda_\blt]}$ respectively.

If $\Wbb$ is a vector space, we let $\bk{\cdot,\cdot}$ be the evaluation pairing between $\Wbb$ and $\Wbb^*$. In other words, for each $w\in\Wbb,\uppsi\in\Wbb^*$, we write
\begin{align*}
\uppsi(w)=\bk{\uppsi,w}=\bk{w,\uppsi}
\end{align*}

If $A,B$ are sets, then both $A-B$ and $A\setminus B$ denote $\{a\in A:a\notin B\}$. Note that if $A,B$ are hypersurfaces in a complex manifold, the notation $A-B$, which denotes a divisor, has a different meaning. (Cf. the proof of Prop. \ref{geometry10}.)

For any complex manifold $X$, recall that $\Theta_X$ is the sheaf of (germs) of holomorphic tangent fields of $X$, i.e., the holomorphic tangent bundle of $X$. We let $\omega_X$ be its dual sheaf $\Theta_X^*$, the holomorphic cotangent bundle of $X$. In particular, if $X$ is a Riemann surface, then $\omega_X$ is the sheaf of (germs of) holomorphic $1$-forms on $X$.

\subsubsection{Series with logarithmic terms}


\begin{df}\label{lb19}
Let $X$ be a complex manifold. Let $R\in\Zbb_+$, and let $\Gamma\subset\Cbb^R$ be locally compact. Let $D$ be a finite subset of $\Cbb^R$. Let $L_1,\dots,L_R\in\Nbb$. We say that  the formal series
\begin{align}\label{eq43}
f(z_\blt):=        \sum_{n_\blt\in D}\sum_{0\leq l_\blt\leq L_\blt}c_{n_\blt,l_\blt} z_1^{n_1}\cdots z_R^{n_R} (\log z_1)^{l_1}\cdots(\log z_R)^{l_R}
    \end{align}
(where each $c_{n_\blt,l_\blt}\in \MO(X)$) \textbf{converges a.l.u.} on $X\times\Gamma$ if, for each $0\leq l_\blt\leq L_\blt$ (i.e. $0\leq l_i\leq L_i$ for all $1\leq i\leq R$) and each compact $K\subset X\times\Gamma$  we have
\begin{align}\label{eq29}
\sup_{(x,z_\blt)\in K}\sum_{n_\blt\in D}\big|c_{n_\blt,l_\blt}(x) z_1^{n_1}\cdots z_R^{n_R}\big|<+\infty
\end{align}
Suppose that $\Wbb$ is a vector space and $\uppsi:\Wbb\rightarrow\MO(X)\{z_\blt\}[\log z_\blt]$ is a linear map such that for each $w\in\Wbb$, the series $\uppsi(w)$ is of the form \eqref{eq43}. We say that $\uppsi$ \textbf{converges a.l.u.} on $X\times\Gamma$ if for each $w\in\Wbb$, the series $\uppsi(w)$ converges a.l.u. on $\Gamma$.
\end{df}

Now, we let $D=E+\Nbb=\{\alpha+\beta:\alpha\in E,\beta\in \Nbb\}$ where $E$ is a finite subset of $\Cbb$. Let $L\in\Nbb$. We recall the following result from \cite[Prop. 2.1]{Huang-applicability} and provide an alternative proof with a slightly different flavor.

\begin{pp}\label{geometry11}
Suppose that there exists $\eps>0$ such that \eqref{eq29} converges a.l.u. (in the sense of Def. \ref{lb19}) on the real interval $(0,\eps)$ to $0$. Then $c_{n,l}=0$ for all $n\in D$ and $0\leq l\leq L$.
\end{pp}

\begin{proof}
Suppose that not all $c_{n,l}$ are zero. Let $r\in\Rbb$ be the smallest number such that $c_{n,l}\neq0$ for some $0\leq l\leq L$ and some $n\in D$ satisfying $\Re(n)=r$. Let $\lambda\in\{0,\dots,L\}$ be the largest number such that $c_{n,\lambda}\neq 0$ for some $n\in D$ satisfying $\Re(n)=r$. Then we can write
\begin{align*}
z^{-r}(\log z)^{-\lambda}f(z)=c_1z^{\im s_1}+\cdots+c_kz^{\im s_k}+g(z)+h(z)
\end{align*}
Here, $c_1,\dots,c_2\in\Cbb$ are non-zero, and $s_1,\dots,s_k\in\Rbb$ are mutually distinct.  Moreover,
\begin{align*}
h(z)=z^{\gamma_1}h_1(z)(\log z)^{\sigma_1}+\cdots+z^{\gamma_m}h_m(z)(\log z)^{\sigma_m}
\end{align*} 
for some $h_1,\dots,h_m\in\Cbb[[z]]$ converging on a neighborhood of $0$, and $\sigma_1,\dots,\sigma_m\in\Zbb$, and the real parts of $\gamma_1\dots,\gamma_m$ are $>0$.  Moreover,
\begin{align*}
g(z)=g_1(z)(\log z)^{-1}+\cdots+g_\lambda(z)(\log z)^{-\lambda}
\end{align*}
where $g_1,\dots,g_\lambda\in\Span_\Cbb\{z^{\im s}:s\in\Rbb\}$. (We let $g(z)=0$ if $\lambda=0$.)

By induction on $j\in\Nbb$, one easily sees that
\begin{align*}
\lim_{t\rightarrow-\infty}\partial^j_tg(e^t)=\lim_{t\rightarrow-\infty}\partial^j_th(e^t)=0
\end{align*}
Therefore, since $f=0$, we have
\begin{align*}
\lim_{t\rightarrow-\infty} (c_1s_1^je^{\im s_1t}+\cdots+c_ks_k^je^{\im s_kt})=0
\end{align*}
for all $j\in\Nbb$. Let $A$ be the $k\times k$ complex matrix $(s_i^{j-1})_{1\leq i,j\leq k}$. Then the above limit implies
\begin{align*}
\lim_{t\rightarrow-\infty}(c_1e^{\im s_1t},\dots,c_ke^{\im s_kt})A=0
\end{align*}
Since $s_1,\dots,s_k$ are mutually distinct, the Vandermonde matrix $A$ is invertible. Therefore, we have $\lim_{t\rightarrow-\infty} c_ie^{\im s_it}=0$ for all $1\leq i\leq k$, and hence $c_i=0$. This is impossible.
\end{proof}

\subsection{Review of basic concepts}\label{lb10}

\subsubsection{The family $\fx=\MS\wtd\fx$ obtained by sewing $\wtd\fx$}\label{lb9}
Let $N\in\Zbb_+$ and $R\in\Nbb$. In this chapter, we assume the setting in \cite[Subsec. 1.2.1]{GZ2}, along with the condition that $\wtd \MB$ is a single point (i.e. a connected $0$-dimensional complex manifold). Namely,
\begin{align}
    \wtd \fx
    =(\wtd \MC\big|\sgm_1,\cdots,\sgm_N;\eta_1,\dots,\eta_N\big\Vert \sgm_1',\cdots,\sgm_R',\sgm_1'',\cdots,\sgm_R'';\xi_1,\dots,\xi_R,\varpi_1,\dots,\varpi_R)
\end{align}
is an $(N+2R)$-pointed compact Riemann surface with local coordinates. More precisely, $\sgm_\blt,\sgm_\blt',\sgm_\blt''$ are distinct points of the compact Riemann surface $\MCtd$. Each $\eta_i$ is a local coordinate at $\sgm_i$, i.e., a univalent function on a neighborhood of $\sgm_i$ sending $\sgm_i$ to $0$. Similarly, each $\xi_j$ is a local coordinate at $\sgm_j'$ and each $\varpi_j$ is a local coordinate at $\sgm_j''$. The $N$-pointed family
\begin{align}
   \fx\equiv\MS\wtd\fx=(\pi:\MC\rightarrow \MB\big|\sgm_1,\cdots,\sgm_N;\eta_1,\dots,\eta_N)
\end{align}
is the obtained by \textbf{sewing $\wtd\fx$ along the pairs of points} $(\sgm_j',\sgm_j'')$ (for all $1\leq j\leq R$) \textbf{via the local coordinates $\xi_j,\varpi_j$}. 

Here, for each $1\leq j\leq R$, we assume that $\xi_j$ is defined on a neighborhood $V_j'$ of $\sgm_j'$, and that $\varpi_j$ is defined on a neighborhood $V_j''$ of $\sgm_j''$. We assume that $V_1',V_1'',\dots,V_R',V_R'',\sgm_1,\dots,\sgm_N$ are mutually disjoint, and 
\begin{align}\label{eq1}
\xi_j(V_j')=\MD_{r_j}\qquad\varpi_j(V_j'')=\MD_{\rho_j}
\end{align}
where $r_j,\rho_j\in(0,+\infty]$ are called the \textbf{sewing radii}. Then
\begin{align}
    \MB=\MD_{r_\blt \rho_\blt}=\MD_{r_1\rho_1}\times \cdots \times \MD_{r_R\rho_R}
\end{align}

For each $1\leq i\leq N$, note that the marked point $\sgm_i\in\wtd\MC$ is extended constantly to a section $\sgm_i:\MB\rightarrow\MC$ of $\fx$. The local coordinate $\eta_i$ of $\wtd\fx$ at $\sgm_i$ is extended constantly to a local coordinate $\eta_i$ of $\fx$ at $\sgm_i(\MB)$ (i.e., a holomorphic function $\eta_i$ on a neighborhood $U_i$ of $\sgm_i(\MB)$ that is univalent on $U_i\cap\pi^{-1}(b)$ for each $b\in\MB$, and which sends $\sgm_i(\MB)$ to $0$). Then $\sgm_1(\MB),\dots,\sgm_N(\MB)$ are mutually disjoint. Let
\begin{align}
    S_\fx=\bigcup_{i=1}^N\sgm_i(\MB)\qquad
    S_{\wtd \fx}=\big\{\sgm_i,\sgm_j',\sgm_j'':1\leq i\leq N,1\leq j\leq R\big\}
\end{align}
which are $1$-codimensional closed submanifolds of $\MC$ and $\wtd\MC$ respectively.

Let $\Delta$ be the set of all $b\in \MB$ such that $\MC_b$ is not smooth. Namely,
\begin{align*}
\MB-\Delta=\mc D_{r_\blt\rho_\blt}^\times:=\MD_{r_1\rho_1}^\times\times \cdots \times \MD_{r_R\rho_R}^\times  
\end{align*}
For each open or closed complex submanifold $T$ of $\MB-\Delta$, we let
\begin{align*}
\fx_T=\big(\pi:\MC_T\rightarrow T\big|\sgm_\blt|_T\big)
\end{align*}
be the restriction of $\fx$ to $T$, where
\begin{align*}
\MC_T=\pi^{-1}(T)
\end{align*}
In particular, for each $b\in \MB-\Delta$, we have the $b$-fiber
\begin{align*}
    \fx_b:=(\MC_b\big| \sgm_\blt(b))=(\MC_b=\pi^{-1}(b)\big|\sgm_1(b),\cdots,\sgm_N(b))
\end{align*}
The restriction of $\eta_1,\dots,\eta_N$ to $\MC_b$ defines local coordinates $\eta_1|_{\MC_b},\dots,\eta_N|_{\MC_b}$ of $\fx_b$. Let
\begin{align}\label{eq8}
\Sigma=\{x\in\MC:\pi\text{ is not a submersion at }x\}
\end{align}

\begin{ass}\label{lb2}
We always assume that each connected component of $\wtd\MC$ intersects  $S_{\wtd\fx}$, and each connected component of each smooth fiber $\fx_b$ of $\fx$ (where $b\in\MB-\Delta$) intersects $\SX$.
\end{ass}

We refer the reader to \cite[Sec. 2.3]{GZ1} or \cite[Sec. 1.2]{GZ2} for more details about the sewing construction. See also the next subsection for a detailed description when $R=1$.

\subsubsection{The case $R=1$}\label{lb3}

Let us describe the construction of $\fx=(\pi:\MC\rightarrow\MB|\sgm_\blt)$ in more detail when $R=1$. We write $\xi_1,\varpi_1,V_1',V_1'',\sgm_1',\sgm_1'',r_1,\rho_1$ as $\xi,\varpi,V',V'',\sgm',\sgm'',r,\rho$ for simplicity. So $\MB=\MD_{r\rho}$. In view of \eqref{eq1}, we make identifications
\begin{gather}\label{eq6}
V'=\MD_r\quad(\text{via }\xi)\qquad\qquad V''=\MD_\rho\quad (\text{via }\varpi)
\end{gather}
so that $\xi:\MD_r\rightarrow\MD_r$ and $\varpi:\MD_\rho\rightarrow\MD_\rho$ become the identity maps. Define open set $W$ and its open subsets $W',W''$ by
\begin{gather}
W=\MD_r\times\MD_\rho\qquad W'=\MD_r^\times\times\MD_\rho\qquad W''=\MD_r\times\MD_\rho^\times
\end{gather}
Extending $\xi,\varpi$ constantly, we can define coordinates
\begin{subequations}\label{eq2}
\begin{gather}
    \xi:W\rightarrow \MD_r  \qquad (z,w,*)\mapsto z\\
    \varpi:W\rightarrow \MD_\rho\qquad (z,w,*)\mapsto w\\
q:W\rightarrow \MD_{r\rho} \qquad (z,w,*)\mapsto zw
\end{gather}
\end{subequations}
so that $q=\xi\varpi$. Then we have open holomorphic embeddings
\begin{subequations}\label{eq11}
\begin{gather}
(\xi,\varpi):W\xrightarrow{=} \MD_r\times \MD_\rho  \label{eq3}\\
(\xi,q):W'\rightarrow \MD_r\times \MD_{r\rho} \label{eq4}\\
(\varpi,q):W''\rightarrow \MD_\rho\times \MD_{r\rho} \label{eq5}
\end{gather}
\end{subequations}
The image of \eqref{eq4} resp. \eqref{eq5} is precisely the subset of all $(z,p)\in \MD_r\times \MD_{r\rho}$ resp. $(w,p)\in \MD_\rho\times \MD_{r\rho}$ satisfying 
\begin{align}\label{eq24}
\frac{\vert p\vert }\rho<\vert z\vert <r \qquad \text{resp.}\qquad \frac{\vert p\vert }r<\vert w\vert <\rho
\end{align}
So closed subsets $F'\subset \MD_r\times \MD_{r\rho}$ and $F''\subset \MD_\rho\times \MD_{r\rho}$ can be chosen such that we have biholomorphisms
\begin{subequations}
\begin{gather}
(\xi,q):W'\xrightarrow{\simeq} \MD_r\times \MD_{r\rho}-F'\\
 (\varpi,q):W''\xrightarrow{\simeq} \MD_\rho\times \MD_{r\rho}-F''
\end{gather}
\end{subequations}
By the identifications \eqref{eq6}, we can write the above maps as
\begin{subequations}\label{eq7}
\begin{gather}
(\xi,q):W'\xrightarrow{\simeq} V'\times \MD_{r\rho}-F'\qquad \subset\wtd \MC\times\mc D_{r\rho}\\
 (\varpi,q):W''\xrightarrow{\simeq} V''\times \MD_{r\rho}-F''\qquad \subset \wtd \MC\times\mc D_{r\rho}
\end{gather}
\end{subequations}
In particular, we view $F'$ and $F''$ as disjoint closed subsets of $\wtd\MC\times \mc D_{r\rho}$.

The complex manifold $\MC$ is defined by 
\begin{gather}\label{eq16}
\MC=W\bigsqcup \big(\wtd\MC\times \MD_{r\rho}-F'-F''\big)\Big/\sim
\end{gather}
Here, the equivalence $\sim$ is defined by identifying each subsets $W',W''$ of $W$ with the corresponding open subsets of $\wtd\MC\times \MD_{r\rho}-F'-F''$ via the biholomorphisms \eqref{eq7}.

The map $\pi:\MC\rightarrow \MB$ is defined as follows. The projection
\begin{align*}
    \wtd\MC\times \MD_{r\rho}\longrightarrow \MD_{r\rho}=\MB
\end{align*}
agrees with
\begin{align*}
q=\xi\varpi:\qquad W=\MD_r\times \MD_\rho\longrightarrow \MD_{r\rho}=\MB
\end{align*}
when restricted to $W'\cup W''$. These two maps give a well-defined surjective holomorphic map $\pi:\MC\rightarrow \MB$. 

\begin{cv}\label{lb4}
We let 
\begin{align*}
q:\MB\rightarrow\Cbb\text{\quad be the standard coordinate of }\MB=\MD_{r\rho}
\end{align*}
Namely, $q(p)=p$ for $p\in\MD_{r\rho}$. If $U\subset\MC$ is open, noting that $\MB=\MD_{r\rho}\subset\Cbb$, we let
\begin{align*}
q:U\rightarrow\Cbb \text{\quad be the extension of }\pi|_U:U\rightarrow\MB\text{ to }U\rightarrow\Cbb
\end{align*}
This convention is compatible with the definition of $q:W\rightarrow\Cbb$ in \eqref{eq2}.
\end{cv}

Therefore, if $\eta\in\MO(U)$ is \textbf{univalent on every fiber} (i.e., for each fiber $U_b=U\cap\pi^{-1}(b)$ of $U$, the restriction $\mu|_{U_b}$ is injective), then $(\eta,q)$ is a set of coordinates of $U$, i.e., it is an open embedding of $U$ into an open subset of $\Cbb^m$ for some $m$.

Clearly $\Sigma=\eqref{eq8}$ is the subset of $W=\MD_r\times\MD_\rho$ described by
\begin{align}\label{eq13}
\Sigma=\{(0,0)\}\qquad\subset\MD_r\times\MD_\rho
\end{align}
and $\Delta=\{0\}$.

For each $1\leq i\leq N$, the marked point $\sgm_i\in\wtd\MC$ is extended constantly $p\in\MD_{r\rho}\mapsto(\sgm_i,p)\wtd\MC\times\MD_{r\rho}$. Since its range it disjoint from $V',V''$ and hence disjoint from $F',F''$, it gives rise to a section $\sgm_i:\MB\rightarrow\MC$.

We also choose a neighborhood $\wtd U_i$ of $\sgm_i$ on which the local coordinate $\eta_i$ is defined. We assume that $\wtd U_1,\dots,\wtd U_N$ are mutually disjoint and are also disjoint from $V',V''$. Then
\begin{align}\label{eq15}
U_i:=\wtd U_i\times\MD_{r\rho}
\end{align}
can be viewed as an open subset of $\MC$ containing $\sgm_i(\MD_{r\rho})$, and $\eta_i$ is extended constantly to a holomorphic map $(x,p)\in \wtd U_i\times\MD_{r\rho}\mapsto \eta_i(x)$ which, by abuse of notation, is also denoted by $\eta_i$. This gives the local coordinate
\begin{align}
\eta_i\in\MO(U_i)
\end{align}
of $\fx$ at $\sgm_i(\MB)$.

\subsubsection{The sheaves $\Theta_\MC(-\log\MC_\Delta)$, $\Theta_\MB(-\log\Delta)$, and $\omega_{\MC/\MB}$}

We recall several sheaves related to the differential geometry of $\pi:\MC\rightarrow\MB$ for $R\leq 1$. See \cite[Subsec. 1.2.2]{GZ2} for details. 

When $R=0$ and hence $\MC=\wtd\MC$ is a compact Riemann surface, then $\Theta_\MC$ and $\omega_\MC$ are defined in Sec. \ref{lb1}.

Now we assume $R=1$. Then $\Theta_\MB(-\log\Delta)$ is the (automatically free) $\MO_\MB$-submodule of $\Theta_\MB$ generated by $q\partial_q$, i.e.
\begin{align}
\Theta_\MB(-\log\Delta)=\MO_\MB\cdot q\partial_q
\end{align}
In particular, $\MO_\MB(-\log\Delta)$ equals $\Theta_\MB$ outside $\Delta=\{0\}$.

The sheaf $\Theta_\MC(-\log\MC_\Delta)$ equals $\Theta_\MC$ outside $\Sigma$. To describe $\Theta_\MC(-\log\MC_\Delta)$ near $\Sigma$, by \eqref{eq13}, it suffices to describe $\Theta_\MC(-\log\MC_\Delta)|_W$. Indeed, the later is the (automatically free) $\MO_W$-submodule of $\Theta|_{W-\Sigma}$ generated 
\begin{align}\label{eq34}
\xi\partial_\xi,\varpi\partial_\varpi
\end{align}
where $\partial_\xi,\partial_\varpi$ are defined under the coordinate $(\xi,\varpi)$ of $W$.

One can define an $\MO_\MC$-module morphism
\begin{align}\label{eq14}
d\pi:\Theta_\MC(-\log\MC_\Delta)\rightarrow\pi^*\Theta_\MB(-\log\Delta)
\end{align}
to be the unique one that restricts to the usual differential map (of tangent vectors) outside $\Sigma$. Therefore, on $W$, we have
\begin{align}\label{eq73}
d\pi(\xi\partial_\xi)=d\pi(\varpi\partial_\varpi)=q\partial_q
\end{align}
where $q\partial_q$ is the abbreviation of $\pi^*(q\partial_q)$. 

Clearly \eqref{eq14} is an epimorphism, and its kernel is denoted by $\Theta_{\MC/\MB}$ and called the \textbf{relative tangent sheaf}. This is a line bundle, i.e., a locally-free $\MO_\MC$-module of rank $1$. Its dual sheaf $\omega_{\MC/\MB}$ is called the \textbf{relative dualizing sheaf}. Therefore, for each open $U\subset \MC-\Sigma$ and each $\eta(U)$ univalent on every fiber of $U$, we have a free generator $d\eta$ of $\omega_{\MC/\MB}$. If $\mu\in\MO(U)$ is also univalent on every fiber of $U$, then
\begin{align}
d\eta=(\partial_\mu\eta)\cdot d\mu
\end{align} 
where the partial derivative is defined with respect to the set of coordinates $(\mu,q\circ\pi)$ of $U$. This gives an explicit description of $\omega_{\MC/\MB}$ outside $\Sigma$. On the other hand,  $\omega_{\MC/\MB}|_W$ can be viewed as the $\MO_\MC$-submodule of $\omega_{\MC/\MB}|_{W-\Sigma}$ generated by the unique element of $H^0(W-\Sigma,\omega_{\MC/\MB}|_{W-\Sigma})$ whose restriction to $W'$ and $W''$ are
\begin{align}\label{eq18}
\xi^{-1}d\xi\qquad\text{resp.}\qquad -\varpi^{-1}d\varpi
\end{align}

\subsubsection{The sheaf $\scr V_\fx$}

We recall the definition of sheaf of VOA
\begin{align*}
\scr V_\fx:=\varinjlim_{n\in\Nbb}\scr V_\fx^{\leq n}=\bigcup_{n\in\Nbb}\scr V_\fx^{\leq n}
\end{align*}
which relies only on $\Vbb$ and on the map $\pi:\MC\rightarrow\MB$ but not on the marked point or the local coordinates of $\fx$. See \cite[Subsec. 1.3.1]{GZ2} and the reference therein for details.

$\scr V^{\leq n}_\fx$ is a locally free $\MO_\MC$-module. The restriction $\scr V^{\leq n}|_{\MC-\Sigma}$ is described as follows. For each open $U\subset\MC-\Sigma$ and each $\eta\in \MO(U)$ univalent on every fiber $U_b$ of $U$, we have a trivialization, i.e., an $\MO_U$-module isomorphism
\begin{align}\label{eq9}
\mc U_\varrho(\eta):\scr V_\fx^{\leq n}|_U\xlongrightarrow{\simeq}\Vbb^{\leq n}\otimes_\Cbb\MO_U
\end{align}
If $\mu\in\MO(U)$ is also univalent on each fiber, the transition function $\mc U_\varrho(\eta)\mc U_\varrho(\mu)^{-1}$, which is an automorphism of $\Vbb^{\leq n}\otimes_\Cbb\MO_U$, is defined using Huang's change of coordinate formula in \cite{Hua97}. The explicit formula of $\mc U_\varrho(\eta)\mc U_\varrho(\mu)^{-1}$, which is not needed in this paper, can be found in Subsec. 1.3.1. But note that these $\mc U_\varrho(\eta)$ are compatible for different $n$ so that $\scr V_\fx^{\leq n}|_{\MC-\Sigma}$ is naturally an $\MO_{\MC-\Sigma}$-submodule of $\scr V_\fx^{\leq n+1}|_{\MC-\Sigma}$.

With abuse of notations, we also denote the tensor product of \eqref{eq9} and the identity map of $\omega_{\MC/\MB}$ by
\begin{align}\label{eq10}
\mc U_\varrho(\eta):\scr V^{\leq n}_\fx\otimes \omega_{\mc C/\mc B}|_U\xlongrightarrow{\simeq}\Vbb^{\leq n}\otimes_\Cbb\mc O_U\otimes_\Cbb d\eta
\end{align}
Namely, it sends $v\otimes d\eta$ to $\mc U_\varrho(\eta)v\otimes_\Cbb d\eta$.

We have finished the definition of $\scr V_\fx$ outside $\Sigma$. In particular, $\scr V_\fx$ is defined when $R=0$ (and hence $\fx=\wtd\fx$). We now describe $\scr V_\fx$ near $\Sigma$ when $R=1$:
\begin{df}\label{lb5}
The restriction $\scr V_\fx^{\leq n}|_W$ is the (automatically free) $\MO_W$-submodule of $\scr V_{\fx-\Sigma}^{\leq n}|_{W-\Sigma}$ generated by the sections whose restrictions to $W'$ and $W''$ are 
\begin{align}\label{eq12}
    \MU_\varrho(\xi)^{-1}\big(\xi^{L(0)}v\big) \qquad \text{resp.}\qquad \MU_\varrho(\varpi)^{-1}\big(\varpi^{L(0)}\MU(\upgamma_1)v\big)
\end{align}
where $\xi\in\MO(W')$ and $\varpi\in\MO(W'')$ are defined by \eqref{eq11} and $v\in \Vbb^{\leq n}$. This is well-defined, i.e., the two expressions in \eqref{eq12} agrees on $W'\cap W''$. 
\end{df}
Again, $\scr V_\fx^{\leq n}$ is naturally an $\MO_\MC$-submodule of $\scr V_\fx^{\leq n+1}$. The description of $\scr V_\fx$ for $R\leq 1$ is complete. The description for general $R$ is similar but not needed and hence is omitted. 

Finally, we note that the description of $\scr V_\fx$ for $R=0$ also applies and defines the sheaf $\scr V_{\fx_b}$ associated to the fiber $\fx_b$.

\subsubsection{Conformal blocks}\label{lb57}

Fix $\Wbb\in\Mod(\Vbb^{\otimes N})$ associated to the ordered marked points $\sgm_1(\wtd\MB),\dots,\sgm_N(\wtd\MB)$ of $\wtd\fx$ and the ordered marked points $\sgm_1(\MB),\dots,\sgm_N(\MB)$ of $\fx$. Note that the orders of the marked points are important. We recall the basic properties about conformal blocks. See \cite[Subsec. 1.3.2]{GZ2} and the reference therein for details. 

Since the local coordinates $\eta_\blt$ of $\fx$ are fixed, we can set
\begin{align*}
\SW_\fx(\Wbb)=\Wbb\otimes_\Cbb\MO_\MB  
\end{align*}
The sheaf of coinvariant $\ST_\fx(\Wbb)$ is a quotient sheaf of $\SW_\fx(\Wbb)$, and its dual sheaf $\ST_\fx^*(\Wbb)$ is called the sheaf of conformal blocks.

Choose any connected open $V\subset\MB$. Recall \eqref{eq15} for the meaning of $U_i$. For each $n\in\Nbb$ and each $\sigma\in H^0(U_i\cap \pi^{-1}(V),\scr V_\fx^{\leq n}\otimes\omega_{\MC/\MB}(\blt\SX))$ and $w\in\Wbb\otimes_\Cbb\MO(V)$, we define the \textbf{$i$-th residue action}
\begin{subequations}\label{eq33}
\begin{align}
\sigma*_i w=\Res_{\eta_i=0}Y_i(\MU_\varrho(\eta_i)\sigma,\eta_i)w\qquad\in\Wbb\otimes_\Cbb\MO(V)
\end{align}
More precisely, note that
the element $\mc U_\varrho(\eta_i)\sigma\in H^0(U_i,\Vbb^{\leq n}\otimes_\Cbb\omega_{\MC/\MB}(\blt\SX))$ is a finite sum $\sum_l v_l\otimes f_ld\eta_i$ where $v_l\in\Vbb^{\leq n}$ and $f_l\in H^0(U_i,\MO_{U_i}(\blt\SX))$. Then for each $b\in V$, noting that $f_l|_{U_i\cap\pi^{-1}(b)}$ can be viewed as an element of $\Cbb((\eta_i))$, we have
\begin{align*}
\sigma*_iw\big|_b=\sum_l\Res_{\eta_i=0} Y_i(v_l,\eta_i)w(b)\cdot f_l|_{U_i\cap\pi^{-1}(b)}\cdot d\eta_i
\end{align*}
where the RHS is the residue of an element of $\Wbb((\eta_i))d\eta_i$.

Now, we define the \textbf{residue action} of each $\sigma\in H^0\big(V,\pi_*\big(\SV_{\fx}\otimes \omega_{\MC/\MB}(\blt S_\fx)\big)\big)=H^0(\MC_V,\SV_{\fx}\otimes \omega_{\MC/\MB}(\blt S_\fx))$ on each $H^0(V,\scr W_\fx(\Wbb))=\Wbb\otimes\MO(V)$ to be
\begin{align}
\sigma\cdot w=\sum_{i=1}^N\sigma*_i w
\end{align}
\end{subequations}

\begin{df}\label{lb48}
The \textbf{sheaf of coinvariants} associated to $\fx$ and $\Wbb$ is defined to be
\begin{align}\label{eq41}
\scr T_\fx(\Wbb)=\frac{\SW_\fx(\Wbb)}{\pi_*\big(\SV_{\fx}\otimes \omega_{\MC/\MB}(\blt S_\fx)\big)\cdot \SW_\fx(\Wbb)}
\end{align}
and is locally free (\cite[Thm. 3.13]{GZ2}), i.e., it is a (finite-rank) holomorphic vector bundle on $\MB$. Its dual bundle is denoted by $\pmb{\scr T^*_\fx(\Wbb)}$ and called the \textbf{sheaf of conformal blocks} (or \textbf{conformal block bundle}) associated to $\fx$ and $\Wbb$.
\end{df}

Note that when $R=0$, then $\fx=\wtd\fx$ is a single surface. Then $\scr W_\fx(\Wbb)=\Wbb$, and $\scr T_\fx(\Wbb)$ is a quotient space of $\Wbb$ that is finite-dimensional, and $\scr T^*_\fx(\Wbb)$ is the dual space of $\scr T_\fx(\Wbb)$.

Therefore, for general $R\in\Nbb$ and $b\in\MB-\Delta$, the vector spaces $\scr W_{\fx_b}(\Wbb),\scr T_{\fx_b}(\Wbb),\scr T^*_{\fx_b}(\Wbb)$ can be defined. The elements of  $\scr T^*_{\fx_b}(\Wbb)$, which are linear functionals $\Wbb\rightarrow\Cbb$ satisfying certain invariant condition, is called a \textbf{conformal block} associated to $\Wbb$ and $\fx_b$ (together with its local coordinates $\eta_\blt|_{\MC_b}$).

By \cite[Remark 3.1]{GZ2}, for each open $V\subset\MB-\Delta$, an element $\uppsi\in H^0\big(V,\ST_\fx^*(\Wbb)\big)$ is equivalently a linear map $\uppsi:\Wbb\rightarrow \MO(V)$ such that for each $b\in V$, the restriction $\uppsi(\cdot)|_b:\Wbb\rightarrow\Cbb$ belongs to $\ST^*_{\fx_b}(\Wbb)$. Such $\uppsi$ is called a \textbf{conformal block} associated to $\Wbb$ and the restricted family $\fx|_V$.

\begin{rem}\label{lb49}
Let $b\in \MB-\Delta$. By \cite[Thm 3.13]{GZ2}, the map
\begin{align*}
\ST^*_\fx(\Wbb)_b\longrightarrow\ST^*_{\fx_b}(\Wbb)\qquad\upphi\mapsto\upphi(\cdot)|_b
\end{align*}
(where $\ST^*_\fx(\Wbb)_b$ is the stalk of $\ST^*_\fx(\Wbb)$ at $b$) descends to a linear isomorphism
\begin{align}\label{eq104}
\frac{\ST^*_\fx(\Wbb)_b}{~\mk_{\MB,b}\ST^*_\fx(\Wbb)_b~}\longrightarrow\ST^*_{\fx_b}(\Wbb)
\end{align}
In other words, the fiber of the vector bundle $\ST_\fx^*(\Wbb)$ at $b$ is canonically isomorphic to the space of conformal blocks $\ST^*_{\fx_b}(\Wbb)$.
\end{rem}

\subsection{Truncated $q$-expansions of global sections of $\SV_\fx\otimes \omega_{\MC/\MB}(\blt S_\fx)$}

In this section, we assume that $R=1$. 

\subsubsection{$q$-expansions and their truncated expansions}\label{lb18}

Let $0<\eps<r\rho$ and $v\in H^0(\pi^{-1}(\MD_\eps),\scr V_\fx\otimes\omega_{\MC/\MB}(\blt\SX))$. We shall expand $v$ into a power series. (See also the proof of \cite[Prop. 4.8]{GZ2}.)

Suppose that $\wtd U$ is a precompact open subset of $\wtd\MC-\SXtd$. Choose small enough $0<\delta<\eps$ such that $\wtd U\times \MD_\delta$ is an open subset of $\wtd \MC\times \MD_{r\rho}-F'-F''$. (Recall Subsec. \ref{lb3} for the notations.) Then $\pi:\MC\rightarrow \MB$, when restricted to $\wtd U\times \MD_\delta$, becomes the projection $\pr:\wtd U\times \MD_\delta\rightarrow\MD_\delta$. 

Thus, if $\eta\in\MO(\wtd U)$ is univalent, and if we denote its constant extension $(x,p)\in\wtd U\times\MD_\delta\mapsto \eta(x)$ also by $\eta$, then noting \eqref{eq10}, we can write
\begin{align*}
\mc U_\varrho(\eta)v\big|_{\wtd U\times\MD_\delta}=\sum_{n\in\Nbb} \sigma_n q^n\cdot d\eta\qquad\text{where }\sigma_n\in \Vbb\otimes_\Cbb\MO(\wtd U)
\end{align*}
(Recall Conv. \ref{lb4} for the meaning of $q$.) Then
\begin{align}\label{eq17}
v_n\big|_{\wtd U}=\mc U_\varrho(\eta)^{-1}\sigma_n\cdot d\eta
\end{align}
This definition of $v_n$ is independent of the choice of $\eta$ and $\delta$. Therefore, we obtain $v_n\in H^0(\wtd\MC-\SXtd,\scr V_{\wtd\fx}\otimes\omega_{\wtd\MC})$ whose local expression is given by \eqref{eq17}.

Each $v_n$ clearly has finite poles at $\sgm_1,\dots,\sgm_N$. By the description of $\SV_\fx|_W$ and $\omega_{\MC/\MB}|_W$ in Def. \ref{lb5} and Eq. \eqref{eq18}, there exist a family $(f^\alpha)_{\alpha\in\fk A}$ in $\MO(W)$ and a linearly independent family $(u^\alpha)_{\alpha\in\fk A}$ of vectors of $\Vbb$, indexed by the same finite set $\fk A$, such that
\begin{subequations}\label{eq19}
\begin{gather}
\MU_\varrho(\xi)v\big|_{W'}=\sum_{\alpha\in\fk A}f^\alpha(\xi,q/\xi)\xi^{L(0)}u^\alpha\cdot\frac{d\xi}\xi\\
\MU_\varrho(\xi)v\big|_{W''}=-\sum_{\alpha\in\fk A} f^\alpha(q/\varpi,\varpi)\varpi^{L(0)}\MU(\upgamma_1)u^\alpha\cdot \frac{d\varpi}\varpi
\end{gather}
\end{subequations}
(Recall from \eqref{eq2} that $(\xi,\varpi)$ is the standard set of coordinates of $W=\MD_r\times\MD_\rho$.) Expanding \eqref{eq19} into power series, we see that the term before $q^n$ has poles of orders at most $n+1$ at $\xi=0$ and at $\varpi=0$. Thus the same can be said about $v_n$. So $v_n$ also has finite poles at $\sgm',\sgm''$. This proves that
\begin{align}
v_n\in H^0\big(\wtd\MC,\scr V_{\wtd\fx}\otimes\omega_{\wtd\MC}(\blt\SXtd)\big)
\end{align}

\begin{df}\label{lb20}
We write
\begin{align}\label{eq20}
v=\sum_{n\in\Nbb}v_nq^n
\end{align}
and call \eqref{eq20} it the \textbf{power series expansion} of $v$ (or simply the \textbf{\pmb{$q$}-expansion} of $v$). For each $\hbar\in\Nbb$, we say that $v_0+v_1q+\cdots+v_\hbar q^\hbar$ is the \textbf{\pmb{$\hbar$}-th truncated \pmb{$q$}-expansion} (or simply the \textbf{\pmb{$\hbar$}-th truncated expansion}) of $v$. 
\end{df}

\subsubsection{Classification of truncated $q$-expansions}

Fix $\hbar\in \Nbb$, and choose $v_0,\dots,v_\hbar\in H^0\big(\wtd \MC,\SV_{\wtd \fx}\otimes \omega_{\wtd\MC}(\blt S_{\wtd \fx})\big)$. Recall \eqref{eq10}. Then for each $0\leq k\leq\hbar$, we have
\begin{subequations}\label{eq23}
\begin{gather}
\MU_\varrho(\xi)v_k\big|_{V'-\{\sgm'\}}\in\Vbb\otimes_\Cbb\MO(V'-\{\sgm'\})d\xi\\ \MU_\varrho(\varpi)v_k\big|_{V''-\{\sgm''\}}\in\Vbb\otimes_\Cbb\MO(V''-\{\sgm''\})d\varpi
\end{gather}
\end{subequations}
Therefore, we can take Laurent series expansions of these two sections at $\sgm'$ and $\sgm''$ respectively. Namely, they can be viewed as elements of $\Vbb((\xi))d\xi$ and $\Vbb((\varpi))d\varpi$ respectively.

Let $\Nbb_{>\hbar}=\{n\in\Nbb:n>\hbar\}$. Then $\Nbb^2\setminus\Nbb^2_{>\hbar}$ is the set of all $(m,n)\in\Nbb\times\Nbb$ such that at least one of $m,n$ is $\leq\hbar$.
\begin{df}\label{lb7}
We say that the element $v_0+v_1q+\dots+v_\hbar q^\hbar$ of $H^0\big(\wtd \MC,\SV_{\wtd \fx}\otimes \omega_{\wtd\MC}(\blt S_{\wtd \fx})\big)[q]$ is \textbf{\pmb{$\hbar$}-compatible} if the following property holds: There exists a linearly independent family $(u^\alpha)_{\alpha\in\fk A}$ of vectors of $\Vbb$ with \textit{finite} index set $\fk A$, together with a family
\begin{align}\label{eq83}
\big(c^\alpha_{m,n}:\alpha\in\fk A\text{ and }(m,n)\in\Nbb^2\setminus\Nbb_{>\hbar}^2\big)
\end{align}
in $\Cbb$ such that for each $0\leq k\leq\hbar$, the Laurent series expansions of the two sections in \eqref{eq23} take the form
\begin{subequations}\label{eq22}
\begin{gather}
\MU_\varrho(\xi)v_k\big|_{V'-\{\sgm'\}}=\sum_{\alpha\in\fk A}\sum_{l= -k}^{+\infty}c_{l+k,k}^\alpha\xi^{l+L(0)-1}u^\alpha d\xi \label{geometry7}\\[0.5ex]
\MU_\varrho(\varpi)v_k\big|_{V''-\{\sgm''\}}=-\sum_{\alpha\in\fk A}\sum_{l=-k}^{+\infty}c_{k,l+k}^\alpha\varpi^{l+L(0)-1}\MU(\upgamma_1)u^\alpha d\varpi\label{geometry8}
\end{gather}
\end{subequations}
\end{df}

\begin{lm}\label{lb8}
In Def. \ref{lb7}, for each $\alpha\in\fk A$ and $0\leq k\leq\hbar$, we have
\begin{gather}\label{eq28}
\sum_{m\in\Nbb}|c^\alpha_{m,k}|\cdot r^m<+\infty\qquad \sum_{n\in\Nbb}|c^\alpha_{k,n}|\cdot\rho^n<+\infty
\end{gather}
\end{lm}

\begin{proof}
By expanding $(u^\alpha)_{\alpha\in\fk A}$, we can assume that $(u^\alpha)_{\alpha\in\fk A}$ is a basis of $\Vbb^{\leq M}$ for some $M\in\Nbb$. The first of \eqref{eq28} is obvious when each $u^\alpha$ is homogeneous. Now, we do not assume that $u^\alpha$ is homogeneous. Let $(w^\beta)_{\beta\in\fk A}$ be a homogeneous basis of $\Vbb^{\leq M}$. Then there is an invertible matrix $\Lambda=(\lambda_{\alpha,\beta})_{\alpha,\beta\in\fk A}$ such that $w^\beta=\sum_\alpha \lambda_{\alpha,\beta}u^\alpha$. Using the inverse matrix of $\Lambda$, one can write \eqref{geometry7} as a formal power series
\begin{align*}
\MU_\varrho(\xi)v_k\big|_{V'-\{\sgm'\}}=\sum_{\beta\in\fk A}\sum_{l= -k}^{+\infty}d_{l+k,k}^\beta\xi^{l+L(0)-1}w^\beta d\xi
\end{align*}
where $c^\alpha_{l+k,k}=\sum_\beta \lambda_{\alpha,\beta}d^\beta_{l+k,k}$. Since each $w^\beta$ is homogeneous, for all $0\leq k\leq\hbar$ we have
\begin{align*}
\sum_{m\in\Nbb}|d^\beta_{m,k}|\cdot r^m<+\infty
\end{align*}
This immediately implies the first of \eqref{eq28}. The second of \eqref{eq28} can be proved in a similar way by reducing it to the special case that each $\mc U(\upgamma_1)u^\alpha$ is homogeneous.
\end{proof}

\begin{thm}\label{geometry1}
The following are equivalent.
\begin{enumerate}[label=(\alph*)]
\item There exist $0<\eps<r\rho$ and an element $v\in H^0(\pi^{-1}(\MD_\eps),\scr V_\fx\otimes\omega_{\MC/\MB}(\blt\SX))$ whose $\hbar$-th truncated expansion is $v_0+v_1q+\cdots+ v_\hbar q^\hbar$.
\item $v_0+v_1q+\cdots +v_\hbar q^\hbar$ is $\hbar$-compatible.
\end{enumerate}
\end{thm}

\begin{proof}[Proof of (a)$\Rightarrow$(b)]
Assume (a). Then we can assume that $v$ satisfies \eqref{eq19}. Expand the function $f^\alpha$ in \eqref{eq19} into power series
\begin{align*}
f^\alpha(\xi,\varpi)=\sum_{m,n\in\Nbb}c^\alpha_{m,n}\xi^m\varpi^n
\end{align*}
where $c^\alpha_{m,n}\in\Cbb$. Then we can expand \eqref{eq19} into power series
\begin{gather*}
\MU_\varrho(\xi)v\big|_{W'}=\sum_{\alpha\in\fk A}\sum_{k\in\Nbb}\sum_{l\geq -k}c_{l+k,k}^\alpha\xi^{l+L(0)-1}u^\alpha d\xi\cdot q^k  \\[0.5ex]
\MU_\varrho(\varpi)v\big|_{W''}=-\sum_{\alpha\in\fk A}\sum_{k\in\Nbb}\sum_{l\geq -k}c_{k,l+k}^\alpha\varpi^{l+L(0)-1}\MU(\upgamma_1)u^\alpha d\varpi \cdot q^k
\end{gather*}
Comparing this with \eqref{eq22}, we see that (b) holds.
\end{proof}

\begin{proof}[Proof of (b)$\Rightarrow$(a)]
Step 1. Assume (b). Choose $M\in\Nbb$ such that $u^\alpha\in\Vbb^{\leq M}$ for all $\alpha\in\fk A$. We shall apply the base change Theorem \ref{lb6} to $\pi:\MC\rightarrow\MB$ and
\begin{align*}
\scr E=\scr V^{\leq M}_\fx\otimes\omega_{\MC/\MB}(t\SX)
\end{align*}
for some $t\in\Nbb$. Recall that $\Delta=\{0\}$, and note that the complex analytic space $\MC_0$ is a nodal curve. By \cite[Thm. 2.3]{Gui-sewingconvergence}, we have $H^1(\MC_0,\scr V^{\leq M}_\fx\otimes\omega_{\MC/\MB}(t\SX)|_{\MC_0})=0$ for sufficiently large $t$. For such $t$, the assumptions in Thm. \ref{lb6} are satisfied. Let us fix such a $t$ that also ensures that $v_0,\dots,v_\hbar$ have poles of order at most $t$ at $\SXtd$.\\[-1ex]

Step 2. The goal of this step is to define sections $\chi$ and $\vartheta$ on open subsets of $\MC$ covering $\pi^{-1}(0)$. First,  Lem. \ref{lb8} implies that for each $\alpha\in\fk A$, the series
\begin{align*}
g^\alpha(\xi,\varpi)=\sum_{(m,n)\in\Nbb^2\setminus\Nbb_{>\hbar}^2} c^\alpha_{m,n}\xi^m\varpi^n
\end{align*}
converges a.l.u. to an element of $\MO(W)=\MO(\MD_r\times\MD_\rho)$, also denoted by $g^\alpha$. By Def. \ref{lb5} and Eq. \eqref{eq18}, there exists an element 
\begin{align*}
\chi\in H^0(W,\scr E)=H^0(W,\scr V_\fx^{\leq M}\otimes\omega_{\MC/\MB})
\end{align*}
such that
\begin{subequations}\label{eq25}
\begin{gather}
\MU_\varrho(\xi)\chi\big|_{W'}=\sum_{\alpha\in\fk A}g^\alpha(\xi,q/\xi)\xi^{L(0)}u^\alpha\cdot\frac{d\xi}\xi\\
\MU_\varrho(\varpi)\chi\big|_{W''}=-\sum_{\alpha\in\fk A} g^\alpha(q/\varpi,\varpi)\varpi^{L(0)}\MU(\upgamma_1)u^\alpha\cdot \frac{d\varpi}\varpi
\end{gather}
\end{subequations}

Next, we choose $r',\rho'$ such that $0<r'<r$ and $0<\rho'<\rho$. Let
\begin{align*}
\wtd\Omega=\wtd\MC-\xi^{-1}(\ovl\MD_{r'})-\varpi^{-1}(\ovl\MD_{\rho'})\qquad \Omega=\wtd\Omega\times\MD_\eps
\end{align*}
In view of \eqref{eq1}, the set $\wtd \Omega$ is an open subset of $\wtd\MC$. (In the construction of $\fx$ from $\wtd\fx$, we have assumed the identification \eqref{eq6}. However, here, we do not omit $\xi^{-1}$ and $\varpi^{-1}$, because we want to stress that $\xi^{-1}(\ovl\MD_{r'})$ is in $V'$ and $\varpi^{-1}(\ovl\MD_{\rho'})$ is in $V''$.) Then, for sufficiently small $\eps>0$, one can view $\Omega$ as an open subset of $\MC$ such that $\pi:\MC\rightarrow\MB$ restricts to the projection $\wtd\Omega\times\MD_\eps\rightarrow\MD_\eps$.   (In fact, any $\eps$ satisfying $0<\eps<r'\rho'$ works.) 
Thus
\begin{align*}
\vartheta:=v_0+v_1q+\cdots+v_\hbar q^\hbar\qquad\text{is an element of }H^0(\Omega,\scr E)
\end{align*}

Step 3. Define open subsets $W_0',W_0''$ of $\wtd\MC\times\MD_\eps\subset\MC$ by
\begin{align*}
W_0'=(V'-\xi^{-1}(\ovl\MD_{r'}))\times\MD_\eps\qquad W_0''=(V''-\varpi^{-1}(\ovl\MD_{\rho'}))\times\MD_\eps
\end{align*}
Due to the gluing maps $(\xi,q)$ and $(\varpi,q)$ in \eqref{eq7} defining the equivalence relation $\sim$ in \eqref{eq16}, one can also view $W_0',W_0''$ as disjoint open subsets of $W',W''$ respectively. Now, we can use \eqref{eq25} and \eqref{eq22} to compute that
\begin{subequations}\label{eq27}
\begin{gather}
\mc U_\varrho(\eta)(\chi|_{W_0'}-\vartheta|_{W_0'})=\sum_{\alpha\in\fk A}\sum_{k=\hbar+1}^{+\infty}\sum_{l= -k}^{\hbar-k}c_{l+k,k}^\alpha\xi^{l+L(0)-1}u^\alpha d\xi\cdot q^k\\[0.5ex]
\mc U_\varrho(\varpi)(\chi|_{W_0''}-\vartheta|_{W_0''})=-\sum_{\alpha\in\fk A}\sum_{k=\hbar+1}^{+\infty}\sum_{l= -k}^{\hbar-k}c_{k,l+k}^\alpha\varpi^{l+L(0)-1}\MU(\upgamma_1)u^\alpha d\varpi \cdot q^k
\end{gather}
\end{subequations}
Since $\mk_{\MB,0}^{\hbar+1}=q^{\hbar+1}\MO_\MB$, and since $\Omega\cap W=W_0'\cup W_0''$ clearly holds, \eqref{eq27} implies that
\begin{align}\label{eq26}
\chi|_{\Omega\cap W}-\vartheta|_{\Omega\cap W}\in H^0\big(\Omega\cap W,\mk_{\MB,0}^{\hbar+1}\scr E\big)
\end{align}

Clearly $\Omega\cup W$ contains $\pi^{-1}(0)$. Thus, since $\pi$ is proper, we can make $\eps$ smaller such that $\Omega\cup W$ covers $\pi^{-1}(\MD_\eps)$. Therefore, by \eqref{eq26}, there is an element
\begin{align*}
\sigma\in H^0\big(\pi^{-1}(\MD_\eps),\scr E/\mk_{\MB,0}^{\hbar+1}\scr E\big)
\end{align*}
whose restriction to $W\cap \pi^{-1}(\MD_\eps)$ is represented by $\chi$, and whose restriction to $\Omega\cap\pi^{-1}(\MD_\eps)$ is represented by $\vartheta$. Hence, by Step 1, we can apply Thm. \ref{lb6}, which says that after further shrinking $\eps$, there exists $v\in H^0(\pi^{-1}(\MD_\eps),\scr E)$ that is sent by the canonical quotient map to $\sigma$. Clearly $v$ has $\hbar$-th truncated expansion $v_0+v_1q+\cdots+v_\hbar q^\hbar$. This proves (a).
\end{proof}

\subsubsection{A base change theorem in complex analytic geometry}

In the proof of Thm. \ref{geometry1} we have used the following base change theorem.

\begin{thm}\label{lb6}
Let $\varphi:X\rightarrow Y$ be a holomorphic map of complex manifolds. Assume that $\varphi$ is proper and open. Let $\scr E$ be a locally-free $\MO_X$-module. 

Let $y\in Y$. Assume that $H^1(X_y,\scr E|_{X_y})=0$. Then for each $\hbar\in\Nbb$, the canonical map of stalks
\begin{align}
(\varphi_*\scr E)_y\longrightarrow \big(\varphi_*\big(\scr E/\mk_{Y,y}^{\hbar+1}\scr E\big)\big)_y
\end{align}
is an epimorphism of $\MO_{Y,y}$-modules. In other words, for each neighborhood $V$ of $y$ and each $\sigma\in H^0\big(\varphi^{-1}(V),\scr E/\mk_{Y,y}^{\hbar+1}\scr E\big)$, there exist a smaller neighborhood $V_0$ of $y$ and some $\wht\sigma\in H^0(\varphi^{-1}(V_0),\scr E)$ that is sent by the quotient map $\scr E\rightarrow\scr E/\mk_{Y,y}^{\hbar+1}\scr E$ to $\sigma|_{\varphi^{-1}(V_0)}$.
\end{thm}

Some of the notations are explained as follows. Recall that $\mk_{Y,y}$ is the maximal ideal of the local ring $\MO_{Y,y}$, understood also as the ideal sheaf of sections of $\MO_Y$ vanishing at $y$. Then $\mk_{Y,y}^{\hbar+1}$ is its $(\hbar+1)$-power, i.e., it is the ideal sheaf generated by $\{g_{i_0}\cdots g_{i_\hbar}:1\leq i_0,\dots,i_\hbar\leq n\}$ if $\mk_{Y,y}$ is generated by $g_1,\dots,g_n\in \MO(Y)$. So $\mk_{Y,y}^{\hbar+1}\scr E$ is the $\MO_X$-submodule of $\scr E$ generated by
\begin{align*}
\{\beta\cdot s\equiv (\varphi^*\beta)\cdot s\text{ where }\beta\in \mk_{Y,y}^{\hbar+1}\text{ and }s\in\scr E\}
\end{align*}
The complex analytic space $X_y$ is defined to be the fiber of $X$ at $y$. It equals $\varphi^{-1}(y)$ as a (Hausdorff) topological space, and its structure sheaf is $(\MO_X/\mk_{Y,y}\MO_X)\scalebox{1.2}{$\upharpoonright$}_{\varphi^{-1}(y)}$, the inverse image of $\MO_X/\mk_{Y,y}\MO_X$ under the inclusion map $\varphi^{-1}(y)\hookrightarrow X$. (See \cite{Fis76} or \cite{GR84} for the basic notions of complex analytic spaces.)

\begin{proof}
This theorem remains valid under a more general assumption, where the first paragraph is replaced by the weaker condition that $\varphi:X\rightarrow Y$ is a proper holomorphic map of complex analytic spaces, that $\scr E$ is a coherent $\MO_X$-module, and that $\scr E$ is $\varphi$-flat (i.e., for each $x\in X$, viewing the stalk $\scr E_x$ as an $\MO_{Y,\varphi(x)}$-module, the functor $\scr E_x\otimes_{\MO_{Y,\varphi(x)}}-$ on the abelian category of $\MO_{Y,\varphi(x)}$-modules is exact). Under this assumption, the theorem follows from Cor. 3.5 in Ch. III of \cite{BaSt}.

We now explain why this assumption is indeed weaker. If the original conditions stated in the theorem's first paragraph hold, then $\varphi$ is clearly proper. Since $\varphi$ is open and since $X,Y$ are complex manifolds, by the last Corollary in Sec. 3.20 of \cite{Fis76}, the structure sheaf $\MO_X$ is $\varphi$-flat. Since $\scr E$ is locally-free, it is also $\varphi$-flat.
\end{proof}

\subsubsection{The geometric meaning of $\hbar$-compatibility}

From the proof of Thm. \ref{geometry1}, readers familiar with complex analytic spaces (as treated in \cite{Fis76,GR84,GPR94}, for example) can easily recognize the geometric meaning of $\hbar$-compatibility. To illustrate this, let $\MB^{\hbar}$ be the closed analytic subspace of $\MB$ associated to $\MO_\MB/q^{\hbar+1}\MO_\MB$. Namely, $\MB^\hbar$ is the topological space $\{0\}$, equipped with the structure sheaf $\MO_\MB/q^{\hbar+1}\MO_B\scalebox{1.2}{$\upharpoonright$}_{\{0\}}$. Pulling back $\fx$ along the inclusion $\MB^\hbar\hookrightarrow\MB$ gives a subfamily $\fx^\hbar$ with map $\pi:\MC^\hbar\rightarrow\MB^\hbar$. Here, $\MC^\hbar$ is the (Hausdorff) topological space $\pi^{-1}(0)$ together with the structure sheaf $\MO_\MC/q^{\hbar+1}\MO_\MC\scalebox{1.2}{$\upharpoonright$}_{\pi^{-1}(0)}$. (In other words, $\pi:\MC^\hbar\rightarrow\MB^\hbar$ is the $\hbar$-th order infinitesimal deformation of the nodal curve $\MC_0$.)

Restricting $\scr V_\fx$ and $\omega_{\MC/\MB}$ to $\MC^\hbar$ yields the sheaves $\scr V_{\fx^\hbar}$ and $\omega_{\MC^\hbar/\MB^\hbar}$, where $\scr V_{\fx^\hbar}$ is the \textbf{sheaf of VOA associated to the \pmb{$\hbar$}-th order infinitesimal deformation} $\fx^\hbar$. In the special case $\hbar=0$, then $\scr V_{\fx^\hbar}$ is the sheaf of VOA associated to the nodal curve $\MC_0$.

The restriction of $\pi:\MC^\hbar\rightarrow\MB^\hbar$ to $\MC^\hbar-\Sigma$ can be regarded as the projection $(\wtd \MC-\sgm'-\sgm'')\times \MB^\hbar\rightarrow\MB^\hbar$. Therefore, each element of $H^0(\MC^\hbar,\scr V_{\fx^\hbar}\otimes\omega_{\MC^\hbar/\MB^\hbar}(\blt\SX))$  can be faithfully presented as $v_0+v_1q+\cdots+v_\hbar q^\hbar$ where $v_0,\dots,v_\hbar\in H^0(\wtd\MC,\scr V_{\wtd\fx}\otimes\omega_{\wtd\MC}(\blt\SXtd))$.

By adapting Steps 2 and 3 of the proof of (b)$\Rightarrow$(a) in Thm. \ref{geometry1}, one can verify the following remark, which provides the geometric meaning of $\hbar$-compatibility.

\begin{rem}
Let $v_0,\dots,v_\hbar\in H^0(\wtd\MC,\scr V_{\wtd\fx}\otimes\omega_{\wtd\MC}(\blt\SXtd))$. Then $v_0+v_1q+\cdots+v_\hbar q^\hbar$ corresponds to some element of $H^0(\MC^\hbar,\scr V_{\fx^\hbar}\otimes\omega_{\MC^\hbar/\MB^\hbar}(\blt\SX))$ if and only if $v_0+v_1q+\cdots+v_\hbar q^\hbar$ is $\hbar$-compatible.
\end{rem}

\subsection{Truncated $q$-expansions with prescribed Laurent coefficients at $\sgm',\sgm''$}

We continue to assume $R=1$. We shall show that any finitely many elements of the family $(c^\alpha_{m,n})$ in Def. \ref{lb7} can be assigned prescribed numbers. Note that if $0\leq\hbar\leq Q$ are integers, then $\Nbb_{\leq Q}^2\setminus\Nbb_{>\hbar}^2$ is the set of all $(m,n)\in\Nbb$ such that $m,n\leq Q$, and that at least one of $m,n$ is $\leq\hbar$.

\begin{pp}\label{geometry10}
Assume that each connected component of $\wtd\MC$ contains one of $\sgm_1,\dots,\sgm_N$. Choose $M\in\Nbb$ and $u\in\Vbb^{\leq M}$. Choose integers $0\leq\hbar\leq Q$.  Choose a family
\begin{align*}
\big(c_{m,n}:(m,n)\in\Nbb_{\leq Q}^2\setminus\Nbb_{>\hbar}^2\big)
\end{align*}
in $\Cbb$. Then for each $0\leq k\leq\hbar$, there exists $v_k\in H^0\big(\wtd \MC,\SV_{\wtd\fx}^{\leq M}\otimes \omega_{\wtd \MC}(\blt S_{\wtd \fx})\big)$ whose Laurent series expansions at $\sgm'$ and $\sgm''$ are respectively
\begin{subequations}\label{eq45}
\begin{gather}
\MU_\varrho(\xi)v_k\big|_{V'-\{\sgm'\}}=\sum_{t=-k}^{Q-k}c_{t+k,k}\xi^{t+L(0)-1}ud\xi \qquad\mathrm{mod~} \Vbb^{\leq M}\otimes_\Cbb\MO(V')\xi^{Q+M}d\xi\\
\MU_\varrho(\varpi)v_k\big|_{V''-\{\sgm''\}}=-\sum_{t=-k}^{Q-k}c_{k,t+k}\varpi^{t+L(0)-1}\MU(\upgamma_1)u d\varpi\qquad\mathrm{mod~} \Vbb^{\leq M}\otimes_\Cbb\MO(V'')\varpi^{Q+M}d\varpi
\end{gather}
\end{subequations}
\end{pp}

\begin{proof}
    Write $Z_{\wtd\fx}=\{\sgm_1,\dots,\sgm_N\}$. Let $A\in\Nbb$, and consider the short exact sequence 
    \begin{gather*}
        0\rightarrow \SV_{\wtd\fx}^{\leq M}\otimes \omega_{\wtd\MC}\big(AZ_{\wtd\fx}-(Q+M)\sgm'-(Q+M)\sgm''\big)\\
        \rightarrow \SV_{\wtd\fx}^{\leq M}\otimes \omega_{\wtd\MC}\big(AZ_{\wtd\fx}+(k+1)\sgm'+(k+1)\sgm''\big)\rightarrow \scr P\rightarrow 0.
    \end{gather*}
    where $\scr P$ is the quotient of the previous two sheaves. Abbreviate this exact sequence to $0\rightarrow\scr P''\rightarrow\scr P'\rightarrow\scr P\rightarrow0$. Since each component of $\wtd\MC$ intersects $Z_{\wtd\fx}$, by Serre's vanishing theorem (or by the proof of \cite[Thm. 2.3]{Gui-sewingconvergence}), we have $H^1(\wtd\MC,\scr P'')=0$ for sufficiently large $A$. Fix such an $A$. Then, we obtain a long exact sequence
    \begin{align*}
        &0\rightarrow H^0(\wtd\MC,\scr P'')\rightarrow H^0(\wtd\MC,\scr P')\rightarrow H^0(\wtd\MC,\scr P)\rightarrow 0
\end{align*}
Define $\sigma\in H^0(\wtd \MC,\scr P)$ as follows. Let $\mc U_\varrho(\xi)\sigma|_{V'}$ and $\MU_\varrho(\varpi)\sigma|_{V''}$ be represented by
\begin{gather*}
\sum_{t=-k}^{Q-k}c_{t+k,k}\xi^{t+L(0)-1}ud\xi\qquad -\sum_{t=-k}^{Q-k}c_{k,t+k}\varpi^{t+L(0)-1}\MU(\upgamma_1)u d\varpi
\end{gather*}
respectively. On $\wtd \MC-\{\sgm',\sgm''\}$, we set $\sigma=0$. Then, by the above exact sequence, $\sigma$ has a lift $v_k\in H^0(\wtd\MC,\scr P')$. Clearly $v_k$ satisfied the desired property.
\end{proof}

\begin{pp}\label{lb33}
Under the assumptions of Prop. \ref{geometry10}, assume that $v_0,\dots,v_\hbar\in H^0\big(\wtd \MC,\SV_{\wtd\fx}^{\leq M}\otimes \omega_{\wtd \MC}(\blt S_{\wtd \fx})\big)$ satisfy \eqref{eq45}. Then $v_0+v_1q+\cdots+v_\hbar q^\hbar$ is $\hbar$-compatible. 
\end{pp}

\begin{proof}
Choose a basis $(u^\alpha)_{\alpha\in\fk A}$ of $\Vbb^{\leq M}$ such that the element $u$ in Prop. \ref{geometry10} equals $\epsilon\cdot u^\beta$ for some $\beta\in\fk A$ and $\epsilon\in\Cbb$. By \eqref{eq45}, 
\begin{gather*}\label{eq82}
\MU_\varrho(\xi)v_k\big|_{V'-\{\sgm'\}}-\sum_{t=-k}^{Q-k}c_{t+k,k}\xi^{t+L(0)-1}u d\xi\tag{$\star$}
\end{gather*}
is a (finite) $\Cbb[[\xi]]$-linear combination of $(\xi^{Q+M}u^\alpha)_{\alpha\in\fk A}$. Since
\begin{align*}
\xi^{Q+M}u^\alpha=\xi^{Q+L(0)}\cdot \xi^{M-L(0)}u^\alpha
\end{align*}
and since $\xi^{M-L(0)}u^\alpha$ is a $\Cbb[\xi]$-linear combination of $(u^\alpha)_{\alpha\in\fk A}$, we conclude that \eqref{eq82} is a $\Cbb[[\xi]]$-linear combination of $(\xi^{Q+L(0)}u^\alpha)_{\alpha\in\fk A}$. Therefore, for each $\alpha\in\fk A,0\leq k\leq\hbar$, there exists a family $(f^\alpha_{t+k,k})_{t\in\Nbb_{\geq Q+1}}$ in $\Cbb$ such that $\MU_\varrho(\xi)v_k\big|_{V'-\{\sgm'\}}$ has power series expansion
\begin{subequations}\label{eq84}
\begin{gather}
\sum_{t=-k}^{Q-k}c_{t+k,k}\xi^{t+L(0)-1}u d\xi+\sum_{\alpha\in\fk A}\sum_{t=Q+1}^{+\infty}f^\alpha_{t+k,k}\xi^{t+L(0)-1}u^\alpha d\xi
\end{gather}
Similarly, we can expand $\MU_\varrho(\varpi)v_k\big|_{V''-\{\sgm''\}}$ into the power series
\begin{align}
-\sum_{t=-k}^{Q-k}c_{k,t+k}\varpi^{t+L(0)-1}\MU(\upgamma_1)u d\varpi-\sum_{\alpha\in\fk A}\sum_{t=Q+1}^{+\infty} g^\alpha_{k,t+k}\varpi^{t+L(0)-1}\MU(\upgamma_1)u^\alpha d\varpi
\end{align}
\end{subequations}
where $g^\alpha_{k,t+k}\in\Cbb$. Define the family $(c^\alpha_{m,n})$ (where $\alpha\in\fk A$ and $(m,n)\in\Nbb^2\setminus\Nbb_{>\hbar}^2$) by
\begin{align*}
c^\alpha_{m,n}=\left\{\begin{array}{ll}
\delta_{\alpha,\beta}\cdot\epsilon\cdot c_{m,n}&\text{ if }(m,n)\in\Nbb_{\leq Q}^2\setminus\Nbb_{>\hbar}^2\\[0.5ex]
f^\alpha_{m,n}&\text{ if }n\leq\hbar\text{ and }m>Q+n\\[0.5ex]
g^\alpha_{m,n}&\text{ if }m\leq\hbar\text{ and }n>Q+m\\[0.5ex]
0&\text{ otherwise}
\end{array}\right.
\end{align*}
Then \eqref{eq22} is satisfied.
\end{proof}

\subsection{Fusion product $\boxtimes_\ff(\Wbb)$ and dual fusion product $\bbs_\ff(\Wbb)$}

In this section, we let $N,R\in\Nbb$.

\begin{df}\label{lb12}
An \textbf{\pmb{$(R,N)$}-pointed compact Riemann surface with local coordinates} denotes a data of the form
\begin{align*}
\fk F=\big(x_1',\dots,x_R'; \theta_1',\dots,\theta_R'\big|C\big|x_1,\dots,x_N;\theta_1,\dots,\theta_N\big)
\end{align*}
Here $C$ is a compact Riemann surface, $x_1,\dots,x_N,x_1',\dots,x_R'$ are distinct marked points of $C$. Each $x_i$ has local coordinate $\theta_i$ (i.e., $\theta_i$ is a univalent holomorphic function on a neighborhood of $x_i$ sending $x_i$ to $0$), and each $x_j'$ has a local coordinate $\theta_j'$. We call $x_1,\dots,x_N$ the \textbf{incoming marked points} (or simply \textbf{inputs}) of $\fk F$, and we call $x_1',\dots,x_R'$ the \textbf{outgoing marked points} (or simply \textbf{outputs}) of $\fk F$.
\end{df}

\subsubsection{Definition and explicit construction}

Let $\ff$ be as in Def. \ref{lb12}. Associate $\Wbb\in\Mod(\Vbb^{\otimes N})$ to the ordered marked points $x_1,\dots,x_N$. 

\begin{df}\label{lb13}
Assume that
\begin{align}\label{eq30}
\text{each connected component of $C$ contains one of $x_1,\dots,x_N$}
\end{align}
A \textbf{dual fusion product} of $\Wbb$ along $\fk F$ denotes a pair $(\bbs_\fk F(\Wbb),\gimel)$ where $\bbs_\fk F(\Wbb)\in\Mod(\Vbb^{\otimes R})$ is associated to $x_1',\dots,x_R'$ and  $\gimel\in\ST_\fk F^*(\bbs_\fk F(\Wbb)\otimes\Wbb)$ satisfies the universal property:
\begin{itemize}
\item For each $\Mbb\in\Mod(\Vbb^{\otimes R})$ associated to $x_1',\dots,x_R'$, the map
\begin{align}\label{eq31}
\Hom_{\Vbb^{\otimes R}}(\Mbb,\bbs_\fk F(\Wbb))\rightarrow \scr T_\fk F^*(\Mbb\otimes\Wbb)\qquad T\mapsto \gimel\circ(T\otimes\id_\Wbb)
\end{align}
is a linear isomorphism.
\end{itemize}
We abbreviate $(\bbs_\fk F(\Wbb),\gimel)$ to $\bbs_\fk F(\Wbb)$ when no confusion arises. The contragredient $\Vbb^{\otimes R}$-module of $\bbs_\ff(\Wbb)$ is denoted by $\boxtimes_\ff(\Wbb)$ and called the \textbf{fusion product} of $\Wbb$ along $\ff$. We call $\gimel$ the \textbf{canonical conformal block}.
\end{df}

Note that $\Mbb\otimes\Wbb$ and $\bbs_\fk F(\Wbb)\otimes\Wbb$ are grading-restricted $\Vbb^{\otimes{(N+R)}}$-modules, both associated to the ordered marked points $x_1',\dots,x_N',x_1,\dots,x_R$.

Dual fusion products are clearly unique up to unique isomorphisms. The existence of dual fusion products was proved in \cite[Thm. 3.31]{GZ1}. Note that when $R=0$, then  $\bbs_\ff(\Wbb)$ is the space of conformal blocks $\scr T_\ff^*(\Wbb)$, and $\gimel:\bbs_\ff(\Wbb)\otimes\Wbb\rightarrow\Cbb$ is given by
\begin{align*}
\gimel:\ST^*_\ff(\Wbb)\otimes\Wbb\rightarrow\Cbb\qquad \upphi\otimes w\mapsto\upphi(w)
\end{align*}
Hence $\boxtimes_{\fk F}(\Wbb)$ is the space of coinvariant $\scr T_{\fk F}(\Wbb)$. See \cite[Sec. 3.4]{GZ2} for more explanations.

\begin{rem}\label{lb15}
The map \eqref{eq31} is clearly injective if $\gimel$ is \textbf{partially injective} in the sense that
\begin{align*}
\{\xi\in\bbs_\ff(\Wbb):\gimel(\xi\otimes w)=0)\text{ for all }w\in\Wbb\}
\end{align*}
is zero.  Conversely, if $(\bbs_\ff(\Wbb),\gimel)$ is a dual fusion product, then its explicit construction in Thm. \ref{lb14} clearly indicates that $\gimel$ is partially injective.
\end{rem}

We need to recall the explicit construction of (dual) fusion products. Let
\begin{align}\label{eq36}
Z_\ff=\{x_1,\dots,x_N\}\qquad Z'_\ff=\{x_1',\dots,x_R'\}
\end{align}
be divisors. In the remaining part of this section, we always assume \eqref{eq30}.

\begin{df}\label{lb22}
For each $a_1,\cdots,a_R\in \Nbb$ and $n\in\Nbb$, define \index{VX@$\SV_{\ff,a_1,\cdots,a_R}^{\leq n}=\SV_{\ff,a_\star}^{\leq n}$}
\begin{gather*}
\SV_{\ff,a_1,\cdots,a_R}^{\leq n}\equiv\SV_{\ff,a_\star}^{\leq n}:=\SV_{\ff}^{\leq n}\big(-(L(0)Z'_\ff+a_1x_1'+\cdots+a_Rx_R')\big)\\
\SV_{\ff,a_1,\cdots,a_R}\equiv\SV_{\ff,a_\star}:=\varinjlim_{n\in \Nbb}\SV_{\ff,a_1,\cdots,a_R}^{\leq n}
\end{gather*}
using the data of $\ff,\Vbb$. More precisely, $\SV_{\ff,a_\star}^{\leq n}$ is a locally free $\MO_{C}$-submodule of $\SV_\ff^{\leq n}$ described as follows. Outside $x_1',\cdots,x_R'$, it is exactly $\SV_\ff^{\leq n}$; for each $1\leq j\leq R$, if $\Omega_j$ is a neighborhood of $x_j'$ on which $\theta_j'$ is defined (and univalent), and if $\Omega_j\cap\{x_1',\dots,x_R'\}=\{x_j'\}$, then $\SV_{\ff,a_\star}^{\leq n}\vert_{\Omega_j}$ is generated by 
\begin{gather}
\MU_\varrho(\theta_j')^{-1}(\theta_j')^{a_j+L(0)}v   
\end{gather}
for homogeneous vectors $v\in\Vbb^{\leq n}$.
\end{df}

\begin{df}
Define a vector space
\begin{align}
\ST_{\ff,a_1,\cdots,a_R}(\Wbb)\equiv\ST_{\ff,a_\star}(\Wbb)=\frac\Wbb{H^0\big(C,\SV_{\ff,a_1,\cdots,a_R}\otimes \omega_{C}(\blt Z_\ff)\big)\cdot\Wbb}
\end{align}
where the linear action of $H^0\big(C,\SV_{\ff,a_\star}\otimes\omega_C(\blt Z_\ff))$ on $\Wbb$ is defined to be the restriction of the residue action of $H^0(C,\scr V_\ff\otimes\omega_{C}(\blt Z_\ff))$. Its dual space is denoted by
\begin{align*}
\pmb{\scr T^*_{\ff,a_1,\dots,a_R}(\Wbb)}\equiv\scr T_{\ff,a_\star}^*(\Wbb)
\end{align*}
\end{df}

\begin{thm}\label{lb14}
Let
\begin{gather*}
\bbs_{\ff}(\Wbb):=\varinjlim_{a_1,\cdots,a_R\in \Nbb}\ST_{\ff,a_1,\cdots,a_R}^*(\Wbb)\equiv \bigcup_{a_1,\dots,a_R\in\Nbb}\ST_{\ff,a_1,\cdots,a_R}^*(\Wbb)
\end{gather*}
which is naturally a linear subspace of $\Wbb^*$. Moreover, $\bbs_\ff(\Wbb)$ has a canonical grading-restricted $\Vbb^{\otimes R}$-module structure. If we restrict the evaluation pairing $\Wbb^*\otimes\Wbb\rightarrow\Cbb$ to the linear functional $\gimel:\bbs_\ff(\Wbb)\otimes\Wbb\rightarrow\Cbb$, then $(\bbs_\ff(\Wbb),\gimel)$ is a dual fusion product of $\Wbb$ along $\ff$.
\end{thm}

\begin{proof}
See \cite[Ch. 3]{GZ1}. In particular, the $\Vbb^{\otimes R}$-module structure on $\bbs_\ff(\Wbb)$ can be defined in terms of the propagation of partial conformal blocks (cf. \cite[Sec. 3.1]{GZ1}). It can also be described by the residue action of elements of $H^0(C,\scr V_\ff\otimes\omega_{C}(\blt Z_\ff+\blt Z'_\ff))$ (cf. \cite[Prop. 3.19]{GZ1}).
\end{proof}

\subsubsection{$\boxtimes$ is right exact and $\bbs$ is left exact}

The following theorem will be used in the proof of Thm. \ref{lb16}.

\begin{thm}\label{lb11}
The covariant functor $\Wbb\in\Mod(\Vbb^{\otimes N})\mapsto\boxtimes_\ff(\Wbb)\in\Mod(\Vbb^{\otimes R})$ is right exact. Equivalently, the contravariant functor $\Wbb\in\Mod(\Vbb^{\otimes N})\mapsto\bbs_\ff(\Wbb)\in\Mod(\Vbb^{\otimes R})$ is left exact.
\end{thm}

\begin{proof}
We first fix $a_1,\dots,a_R$ and prove that $\Wbb\mapsto \ST_{\ff,a_\star}(\Wbb)$ is right exact. Write $J=H^0\big(C,\SV_{\ff,a_\star}\otimes \omega_{C}(\blt Z_\ff)\big)$ so that $\ST_{\ff,a_\star}(\Wbb)=\Wbb/J\Wbb$.  Let $\Wbb_3\rightarrow\Wbb_2\rightarrow\Wbb_1\rightarrow 0$ be an exact sequence in $\Mod(\Vbb^{\otimes N})$. Then we have an exact sequence of chain complexes $0\rightarrow J\Wbb_\blt\rightarrow \Wbb_\blt\rightarrow \ST_{\ff,a_\star}(\Wbb_\blt)\rightarrow0$. Namely, we have the commutative diagram Fig. \ref{fig1}, where the rows are exact, and the composition of any two consecutive vertical arrows is zero.
\begin{figure}[h]
	\centering
\begin{equation*}
\begin{tikzcd}
0 \arrow[r] & J\Wbb_3 \arrow[r] \arrow[d] & \Wbb_3 \arrow[r] \arrow[d,"\varphi"] & \ST_{\ff,a_\star}(\Wbb_3) \arrow[r] \arrow[d,"\alpha"] & 0 \\
0 \arrow[r] & J\Wbb_2 \arrow[r] \arrow[d,"\gamma"] & \Wbb_2 \arrow[r] \arrow[d,"\psi"] & \ST_{\ff,a_\star}(\Wbb_2) \arrow[r] \arrow[d,"\beta"] & 0 \\
0 \arrow[r] & J\Wbb_1 \arrow[r] \arrow[d,"\delta"] & \Wbb_1 \arrow[r] \arrow[d] & \ST_{\ff,a_\star}(\Wbb_1) \arrow[r] \arrow[d] & 0 \\
            & 0                     & 0                     & 0                     &  
\end{tikzcd}
\end{equation*}
	\caption{. The exact sequence of chain complexes in Thm. \ref{lb11}.}
	\label{fig1}
\end{figure} 

We use the notations in Fig. \ref{fig1}. Then, the zig-zag lemma yields an exact sequence
\begin{align*}
\Ker\psi/\Imag\varphi\rightarrow \Ker\beta/\Imag\alpha\rightarrow\Ker\delta/\Imag\gamma
\end{align*}
By assumption, we have $\Ker\psi/\Imag\varphi=0$. Note that if we view $\Wbb_1$ as the quotient module $\Wbb_2/\Wbb_3$, then the action of $J$ on $\Wbb_2$ descends to that of $\Wbb_1$. Therefore $\gamma$ is surjective, and hence $\Ker\delta/\Imag\gamma=0$. Thus $\Ker\beta=\Imag\alpha$. A similar argument shows that $\beta$ is surjective. Therefore, we have an exact sequence $\ST_{\ff,a_\star}(\Wbb_3)\xlongrightarrow{\alpha}\ST_{\ff,a_\star}(\Wbb_2)\xlongrightarrow{\beta}\ST_{\ff,a_\star}(\Wbb_1)\rightarrow0$. This finishes the proof that $\ST_{\ff,a_\star}(-)$ is right exact. Equivalently, $\ST^*_{\ff,a_\star}(-)$ is left exact.

Now, consider the sequence
\begin{align*}
0\rightarrow\bigcup_{a_\star}\ST^*_{\ff,a_\star}(\Wbb_1)\xlongrightarrow{\beta^\tr}\bigcup_{a_\star}\ST^*_{\ff,a_\star}(\Wbb_2)\xlongrightarrow{\alpha^\tr} \bigcup_{a_\star}\ST^*_{\ff,a_\star}(\Wbb_3)
\end{align*}
By the left exactness of $\ST^*_{\ff,a_\star}(-)$, the restriction of $\beta^\tr$ to each $\ST^*_{\ff,a_\star}(\Wbb_1)$ is injective, and the map $\beta^\tr$ sends each $\ST^*_{\ff,a_\star}(\Wbb_1)$ onto $\Ker(\alpha^\tr)\big|_{\ST^*_{\ff,a_\star}(\Wbb_2)}$. Therefore, the above sequence must be exact. This proves that $\bbs_\ff(-)$ is left exact, and hence $\boxtimes_\ff(-)$ is right exact.
\end{proof}

\subsubsection{Fusion product along $\fk F\sqcup\fk G$}\label{lb41}

Let $N,K\in\Zbb_+$ and $R,Q\in\Nbb$. In this subsection, we let 
\begin{gather*}
    \fk F=\big(x_1',\dots,x_R'; \theta_1',\dots,\theta_R'\big|C_1\big|x_1,\dots,x_N;\theta_1,\dots,\theta_N\big)\\
    \fk G=\big(y_1',\dots,y_Q';\mu_1',\dots,\mu_Q'\big|C_2\big| y_1,\dots,y_K;\mu_1,\dots,\mu_K \big)
    \end{gather*}
be respectively $(R,N)$-pointed and $(Q,K)$-pointed compact Riemann surface with local coordinates. We assume that each component of $C_1$ contains one of $x_1,\dots,x_N$, and each component of $C_2$ contains one of $y_1,\dots,y_K$.

Define the \textbf{disjoint union} $\ff\sqcup\fg$ to be
\begin{align}\label{eq86}
\begin{aligned}
    \fk F\sqcup \fk G=\big(&x_1',\dots,x_R',y_1',\dots,y_Q'; \theta_1',\dots,\theta_R',\mu_1',\dots,\mu_Q'\big|C_1\sqcup C_2
    \big|\\
    &x_1,\dots,x_N, y_1,\dots,y_K;\theta_1,\dots,\theta_N,\mu_1,\dots,\mu_K\big)
\end{aligned}
\end{align}
which is an $(R+Q,N+K)$-pointed compact Riemann surface with local coordinates. Associate $\Wbb\in\Mod(\Vbb^{\otimes N})$ to the ordered marked points $x_1,\dots,x_N$. Associate  $\Xbb\in\Mod(\Vbb^{\otimes K})$ to the ordered marked points $y_1,\dots,y_K$. Thus we have dual fusion products
\begin{align*}
(\bbs_\ff(\Wbb),\gimel)\qquad (\bbs_\fg(\Xbb),\daleth)
\end{align*}
Associate the tensor product module $\Wbb\otimes \Xbb\in\Mod(\Vbb^{\otimes(N+K)})$ to the ordered marked points $x_1,\dots,x_N,y_1,\dots,y_K$ of $\ff\sqcup\fg$.

The following lemma is a special case of Thm. \ref{lb16}.

\begin{lm}\label{lb17}
Assume that $R=Q=0$. Then each $\upomega\in\ST^*_{\ff\sqcup\fg}(\Wbb\otimes\Xbb)$ is can be written as a finite sum $\upomega=\sum_i \upphi_i\otimes\uppsi_i$ where $\upphi_i\in\ST^*_\ff(\Wbb)$ and $\uppsi_i\in\ST^*_\fg(\Xbb)$.
\end{lm}

\begin{proof}
Since $\upomega$ is a conformal block, each $\sigma\in H^0(C_1,\SV_\ff\otimes\omega_{C_1}(\blt Z_\ff))$ can be extended by zero to an element of $H^0(C_1\sqcup C_2,\SV_{\ff\sqcup\fg}\otimes \omega_{C_1\sqcup C_2}(\blt Z_\ff+\blt Z_\fg))$. Therefore $\upomega((\sigma\cdot w)\otimes u)=0$ for all $w\in\Wbb,u\in\Xbb$. 

Now, let $(\upphi_i)_{i\in I}$ be a (finite) basis of $\ST^*_\ff(\Wbb)$. By the first paragraph, for each $u\in\Xbb$, the linear functional $w\in\Wbb\mapsto \upomega(w\otimes u)$ is a conformal block associated to $\ff$ and $\Wbb$, and hence can be written uniquely as a linear combination of $(\upphi_i)_{i\in I}$. For each $u\in\Xbb$,  We write it as $\sum_i \uppsi_i(u)\upphi_i$ where $\uppsi_i(u)\in\Cbb$. Similar to the first paragraph, for each $\sigma\in H^0(C_2,\SV_\fg\otimes\omega_{C_2}(\blt Z_\fg))$ we have $\upomega(w\otimes (\sigma\cdot u))=0$, and hence $\uppsi_i(\sigma\cdot u)=0$. Therefore each $\uppsi_i:\Xbb\rightarrow\Cbb$ belongs to $\ST^*_\fg(\Xbb)$. Clearly $\upomega=\sum_i \upphi_i\otimes\uppsi_i$.
\end{proof}

We now arrive at the main result of this section, which is needed in the proof of the sewing-factorization Thm. \ref{lb40}. Roughly speaking, this result says that
\begin{gather*}
\bbs_{\ff\sqcup\fg}(\Wbb\otimes\Xbb)\simeq\bbs_\ff(\Wbb)\otimes\bbs_\fg(\Xbb)\qquad\boxtimes_{\ff\sqcup\fg}(\Wbb\otimes\Xbb)\simeq\boxtimes_\ff(\Wbb)\otimes\boxtimes_\fg(\Xbb)
\end{gather*}
In the special case that $R=Q=0$, the above isomorphisms become
\begin{gather*}
\ST^*_{\ff\sqcup\fg}(\Wbb\otimes\Xbb)\simeq\ST^*_\ff(\Wbb)\otimes\ST^*_\fg(\Xbb)\qquad\ST_{\ff\sqcup\fg}(\Wbb\otimes\Xbb)\simeq\ST_\ff(\Wbb)\otimes\ST_\fg(\Xbb)
\end{gather*}

\begin{thm}\label{lb16}
Define the linear functional
\begin{gather*}
\gimel\otimes\daleth: \bbs_{\fk F}(\Wbb)\otimes \bbs_{\fk G}(\Xbb)\otimes\Wbb\otimes \Xbb\rightarrow \Cbb\\
\upphi\otimes\uppsi\otimes w\otimes u\mapsto \gimel(\upphi\otimes w)\cdot\daleth(\uppsi\otimes u)
\end{gather*}
Then $(\bbs_{\fk F}(\Wbb)\otimes \bbs_{\fk G}(\Xbb),\gimel\otimes\daleth)$ is a dual fusion product of $\Wbb\otimes\Xbb$ along $\ff\sqcup\fg$.
\end{thm}

Here, $\bbs_{\fk F}(\Wbb)\otimes \bbs_{\fk G}(\Xbb)\otimes\Wbb\otimes \Xbb$ is associated to the marked points of $\ff\sqcup\fg$ in the order $x_1',\dots,x_R',y_1',\dots,y_Q',x_1,\dots,x_N,y_1,\dots,y_K$.

\begin{proof}
Clearly $\gimel\otimes\daleth$ is a conformal block associated to $\ff\sqcup\fg$. We need to check the universal property in Def. \ref{lb13}: for each $\Mbb\in\Mod(\Vbb^{\otimes(R+Q)})$, the linear map
\begin{align*}
\Psi_\Mbb:\Hom_{\Vbb^{\otimes (R+Q)}}(\Mbb,\bbs_{\fk F}(\Wbb)\otimes \bbs_{\fk G}(\Xbb))\rightarrow\ST^*_{\ff\sqcup\fg}(\Mbb\otimes \Wbb\otimes\Xbb)\qquad T\mapsto \Psi_\Mbb(T)
\end{align*}
is bijective, where
\begin{align*}
\Psi_\Mbb(T):\Mbb\otimes\Wbb\otimes\Xbb\rightarrow\Cbb\qquad m\otimes w\otimes u\mapsto (\gimel\otimes\daleth)(T(m)\otimes w\otimes u)
\end{align*}

Step 1. Let us prove that $\Psi_\Mbb$ is injective. Suppose that $\Psi_\Mbb(T)=0$. Choose any $m\in\Mbb$, and write $T(m)=\sum_{i\in I} \upphi_i\otimes\uppsi_i$ where $\upphi_i\in\bbs_\ff(\Wbb)$ and $\uppsi_i\in\bbs_\fg(\Xbb)$, and $(\upphi_i)_{i\in I}$ is linearly independent. Then for all $w\in\Wbb,u\in\Xbb$ we have $(\gimel\otimes\daleth)(T(m)\otimes w\otimes u)=0$, which means
\begin{align*}
\sum_i \gimel (\upphi_i\otimes w)\cdot \daleth(\uppsi_i\otimes u)=0
\end{align*}
Since $\bbs_\ff(\Wbb)$ is grading-restricted, we can find $a_1,\dots,a_R\in\Rbb$ such that $\bbs_\ff(\Wbb)_{[\leq a_\blt]}$ contains all $\upphi_i$.  The linear map
\begin{align*}
w\in\Wbb\quad\mapsto \quad \gimel(-\otimes w)\in(\bbs_\ff(\Wbb)_{[\leq a_\blt]})^*
\end{align*}
is surjective; otherwise, we can find a nonzero $\upphi\in \bbs_\ff(\Wbb)_{[\leq a_\blt]}$ such that $\gimel(\upphi\otimes w)=0$ for all $w\in\Wbb$, contradicting the partial injectivity of $\gimel$ (cf. Rem. \ref{lb15}). Therefore, since $(\upphi_i)_{i\in I}$ is linearly independent, we can find a collection $(w_i)_{i\in I}$ in $\Wbb$ such that $\gimel(\upphi_i\otimes w_j)=\delta_{i,j}$ for all $i,j\in I$. Thus $\daleth(\uppsi_i\otimes u)=0$ for all $i\in I$ and all $u\in \Xbb$. Since $\daleth$ is partially injective, we have $\uppsi_i=0$. This proves $T(m)=0$. Thus $T=0$.\\[-1ex]

Step 2. Let us prove that $\Psi_\Mbb$ is bijective. We first consider the special case that $\Mbb=\Ebb\otimes\Fbb$ where $\Ebb\in\Mod(\Vbb^{\otimes R})$ and $\Fbb\in\Mod(\Vbb^{\otimes Q})$. By Lem. \ref{lb17}, it suffices to prove that $\Upomega\otimes\Uptheta$ belongs to the range of $\Psi_\Mbb$ if $\Upomega\in\ST^*_\ff(\Ebb\otimes\Wbb)$ and $\Uptheta\in\ST^*_\fg(\Fbb\otimes\Xbb)$. By the universal property for $(\bbs_\ff(\Wbb),\gimel)$ and $(\bbs_\fg(\Xbb),\daleth)$, there exist $F\in\Hom_{\Vbb^{\otimes R}}(\Ebb,\bbs_\ff(\Wbb))$ and $G\in\Hom_{\Vbb^{\otimes Q}}(\Fbb,\bbs_\fg(\Xbb))$ such that $\Upomega=\gimel\circ (F\otimes\id_\Wbb)$ and $\Uptheta=\daleth\circ (G\otimes \id_\Xbb)$. Then $\Psi_{\Ebb\otimes\Fbb}(F\otimes G)=\Upomega\otimes\Uptheta$. Thus $\Psi_{\Ebb\otimes\Fbb}$ is surjective, and hence is bijective.

Next, we consider the general case. By \cite[Prop. 3.2]{McR-deligne}, there is an epimorphism $\beta:\Ebb\otimes\Fbb\rightarrow\Mbb$ where $\Ebb\in\Mod(\Vbb^{\otimes R})$ and $\Fbb\in\Mod(\Vbb^{\otimes Q})$. Applying the same result to $\Ker\beta$, we obtain an exact sequence
\begin{align}\label{eq32}
\wtd\Ebb\otimes\wtd\Fbb\xlongrightarrow{\alpha}\Ebb\otimes \Fbb\xlongrightarrow{\beta} \Mbb\rightarrow0
\end{align}
Then \eqref{eq32} induces the commutative diagram Fig. \ref{fig2}, where the vertical arrows above the last two rows are defined by composition with $\alpha$ and $\beta$.
\begin{figure}[h]
	\centering
\begin{equation*}
\begin{tikzcd}
0 \arrow[d]           & 0 \arrow[d] \\
\Hom_{\Vbb^{\otimes (R+Q)}}(\Mbb,\bbs_{\fk F}(\Wbb)\otimes \bbs_{\fk G}(\Xbb)) \arrow[d] \arrow[r,"\Psi_\Mbb"] & \ST^*_{\ff\sqcup\fg}(\Mbb\otimes \Wbb\otimes\Xbb) \arrow[d] \\
\Hom_{\Vbb^{\otimes (R+Q)}}(\Ebb\otimes\Fbb,\bbs_{\fk F}(\Wbb)\otimes \bbs_{\fk G}(\Xbb)) \arrow[d] \arrow[r,"\Psi_{\Ebb\otimes\Fbb}"] & \ST^*_{\ff\sqcup\fg}(\Ebb\otimes\Fbb\otimes \Wbb\otimes\Xbb) \arrow[d] \\
\Hom_{\Vbb^{\otimes (R+Q)}}(\wtd\Ebb\otimes\wtd\Fbb,\bbs_{\fk F}(\Wbb)\otimes \bbs_{\fk G}(\Xbb)) \arrow[r,"\Psi_{\wtd\Ebb\otimes\wtd\Fbb}"]           & \ST^*_{\ff\sqcup\fg}(\wtd\Ebb\otimes\wtd\Fbb\otimes \Wbb\otimes\Xbb)          
\end{tikzcd}
\end{equation*}
	\caption{. The commutative diagram induced by \eqref{eq32}.}
	\label{fig2}
\end{figure} 

The left column in Fig. \ref{fig2} is exact because $\Hom_{\Vbb^{\otimes(R+Q)}}(-,\bbs_\ff(\Wbb)\otimes\bbs_\fg(\Xbb))$ is left exact. The right column is exact due to Thm. \ref{lb11}. By the previously proved special case, $\Psi_{\Ebb\otimes\Fbb}$ and $\Psi_{\wtd\Ebb\otimes\wtd\Fbb}$ are isomorphisms. Therefore, by the five lemma, the linear map $\Psi_\Mbb$ is an isomorphism.
\end{proof}

\section{Sewing-factorization theorem: A key special case}

In this chapter, we assume the setting in Subsec. \ref{lb3} and use freely the notations in that subsection. In particular, we let $R=1$. Moreover, we assume that each connected component of $\wtd\MC$ contains one of $\sgm_1,\dots,\sgm_N$. Then Asmp. \ref{lb2} is satisfied. Compatible with \eqref{eq36}, we let
\begin{align}
Z_\fxtd=\{\sgm_1,\dots,\sgm_N\}\qquad Z'_\fxtd=\{\sgm',\sgm''\}
\end{align}
Hence $\SXtd=Z_\fxtd\cup Z'_\fxtd$.

\subsection{The SF theorem}

\subsubsection{The setting}

Define a $(2,N)$-pointed compact Riemann surface with local coordinates
\begin{gather}
\ff=\big(\sgm',\sgm'';\xi,\varpi\big|C\big|\sgm_1,\dots,\sgm_N;\eta_1,\dots,\eta_N \big)
\end{gather}
In other words, $\ff$ is almost the same as $\fxtd$, except that the marked points $\sgm',\sgm'$ of $\fxtd$ are viewed as outgoing points of $\ff$. Let $(\bbs_\ff(\Wbb),\gimel)$ be the dual fusion product associated to $\ff$ and $\Wbb$. Thus $\gimel:\bbs_\ff(\Wbb)\otimes\Wbb\rightarrow\Cbb$ is the canonical conformal block, and $\bbs_\ff(\Wbb)\in\Mod(\Vbb^{\otimes 2})$ is associated to $\sgm',\sgm''$. In particular, for each $v\in\Vbb$, the vertex operator $Y_1(v,z)=Y(v\otimes 1,z)$ is for $\sgm'$, and $Y_2(v,z)=Y(1\otimes v,z)$ is for $\sgm''$. We write
\begin{align*}
Y_+=Y_1\qquad Y_-=Y_2
\end{align*}

Define a $(2,0)$-pointed sphere with local coordinates
\begin{align}\label{eq87}
\fn=\big(\infty,0;1/\zeta,\zeta\big|\Pbb^1\big)
\end{align}
where $\zeta$ denotes the standard coordinate of $\Cbb$. 

\begin{rem}\label{lb39}
Associate $\Mbb\in\Mod(\Vbb^{\otimes 2})$ to the ordered pair of marked points $(\infty,0)$. So $Y_+=Y_1$ is associated to $\infty$, and $Y_-=Y_2$ is associated to $0$. Note that $\ST^*_\fn(\Mbb)\subset\Mbb^*$. Let $\upchi\in\Mbb^*$. It is easy to see that $H^0(\Pbb^1,\SV_\fn\otimes\omega_{\Pbb^1}(\blt\infty+\blt 0))$ is spanned by all $\MU_\varrho(z)^{-1}(uz^ndz)$ where $u\in\Vbb$ and $n\in\Zbb$. Combining this observation with the fact that $(\MU_\varrho(1/\zeta)\MU_\varrho(\zeta)^{-1})_z=\MU(\upgamma_z)$ for all $z\in\Cbb^\times$ (cf. \cite[Exp. 1.15]{GZ1}), we see that $\upchi$ belongs to $\ST^*_\fn(\Mbb)$ if and only if for all $u\in\Vbb,m\in\Mbb$, the relation
\begin{subequations}\label{eq67}
\begin{align}\label{eq67a}
\bigbk{\upchi,Y_+'(u,z)m}=\bigbk{\upchi,Y_-(u,z)m}
\end{align}
holds in $\Cbb[[z^{\pm1}]]$. Due to \eqref{eq48} and \eqref{eq49}, condition \eqref{eq67a} is equivalent to 
\begin{align}
\bigbk{\upchi,Y_+(u,z)m}=\bigbk{\upchi,Y'_-(u,z)m}
\end{align}
\end{subequations}
A similar description holds for conformal blocks associated to a disjoint union of several pieces of $\fn$.
\end{rem}

Associate $\boxtimes_\ff(\Wbb)$ to $(\infty,0)$. In particular, $Y_+=Y_1$ is for $\infty$ and $Y_-=Y_2$ is for $0$. Let
\begin{gather}
\begin{gathered}
\MS(\ff\sqcup\fn)\text{ be the sewing of $\ff\sqcup\fn$ along the pair}\\
\text{$(\sgm',\infty)$ with sewing radii $r,1$, and along $(\sgm'',0)$ with sewing radii $\rho,1$}
\end{gathered}
\end{gather}
(See Subsec. \ref{lb9} for more details.) Then $\MS(\ff\sqcup\fn)$ has base manifold $\MD_r\times\MD_\rho$. Moreover, the pullback of $\fx|_{\MD_{r\rho}^\times}$ along the map $(q_1,q_2)\in\MD_r^\times\times\MD_\rho^\times\mapsto q_1q_2\in\MD_{r\rho}^\times$ is canonically equivalent to $\MS(\ff\sqcup\fn)|_{\MD_r^\times\times\MD_\rho^\times}$. In particular, for each $q_1\in\MD_r^\times,q_2\in\MD_\rho^\times$, we have a canonical equivalence of fibers
\begin{align}\label{eq61}
\fx_{q_1q_2}\simeq \MS(\ff\sqcup\fn)_{q_1,q_2}
\end{align}

For each $\upchi\in\ST_\fn^*(\boxtimes_\ff(\Wbb))$, noting that $\gimel:\bbs_\ff(\Wbb)\otimes\Wbb\rightarrow\Cbb$ and $\upchi:\boxtimes_\ff(\Wbb)\rightarrow\Cbb$ are linear functional, define
\begin{subequations}\label{eq63}
\begin{gather}\label{eq63a}
\begin{gathered}
\MS(\gimel\otimes\upchi):\Wbb\rightarrow \Cbb\{q_1,q_2\}[\log q_1,\log q_2]\\[0.5ex]
\bigbk{\MS(\gimel\otimes\upchi),w}= \wick{\sum_{\lambda_1,\lambda_2\in\Cbb}\gimel(P_{\lambda_\blt}\c1 -\otimes w )\cdot \upchi\big(q_1^{L_+(0)}q_2^{L_-(0)}P_{\lambda_\blt}\c1-\big)}
\end{gathered}
\end{gather}
where, for each $\lambda_\blt$, the contraction is taken using a (finite) basis  $(e_{\lambda_\blt}(\alpha))_{\alpha\in\fk A_{\lambda_\blt}}$ of $\bbs_\ff(\Wbb)_{[\lambda_\blt]}$ and its dual basis $(\wch e_{\lambda_\blt}(\alpha))_{\alpha\in\fk A_{\lambda_\blt}}$ in $\boxtimes_\ff(\Wbb)_{[\lambda_\blt]}$. More precisely, 
\begin{align}\label{eq64a}
\bk{\MS(\gimel\otimes\upchi),w}=\sum_{\lambda_1,\lambda_2\in\Cbb}\sum_{\alpha\in\fk A_{\lambda_\blt}}\gimel(\wch e_{\lambda_\blt}(\alpha)\otimes w)\cdot\bigbk{\upchi,q_1^{L_+(0)}q_2^{L_-(0)}e_{\lambda_\blt}(\alpha)}
\end{align}
\end{subequations}
See \cite[Sec. 4.1]{GZ2} for more discussions. 

By \cite[Thm. 4.9]{GZ2}, $\MS(\gimel\otimes\upchi)$ converges a.l.u. on $\MD^\times_r\times\MD^\times_\rho$ in the sense of Def. \ref{lb19}, and for each $p_1\in\MD_r^\times,p_2\in\MD_\rho^\times$ with chosen arguments, the linear functional $\MS(\gimel\otimes\upchi)_{p_1,p_2}:\Wbb\rightarrow\Cbb$ is a conformal block associated to $\fx_{p_1p_2}=\MS(\ff\sqcup\fn)_{p_1,p_2}$ and $\Wbb$. Therefore
\begin{align}\label{eq65}
\MS(\gimel\otimes\upchi)\in H^0\big(\wht\MD_r^\times\times\wht\MD_\rho^\times,\ST^*_{\MS(\ff\sqcup\fn)}(\Wbb)\big)
\end{align}
(Recall \eqref{eq127} for the meanings of $\wht\MD_r^\times$ and $\wht\MD_\rho^\times$.)

\subsubsection{The sewing-factorization theorem}

The goal of this chapter, which will be achieved in Sec. \ref{lb36}, is to prove the following version of the sewing-factorization theorem.

\begin{thm}\label{lb23}
Let $p_1\in\MD_r^\times,p_2\in\MD_\rho^\times$ with fixed $\arg p_1,\arg p_2$. Then
\begin{gather}\label{eq60}
\ST^*_\fn(\boxtimes_\ff(\Wbb))\rightarrow \ST^*_{\MS(\ff\sqcup\fn)_{p_1,p_2}}(\Wbb)\qquad \chi\mapsto \MS(\gimel\otimes\upchi)\big|_{p_1,p_2}
\end{gather}
is a linear isomorphism.
\end{thm}

\begin{rem}
Thm. \ref{lb23} was originally suggested by Kong and Zheng \cite{KZ-conformal-block} in a slightly different form. In their formulation, they consider self-sewing $\wtd\fx$ rather than sewing $\ff$ with $\fn$, and---due to the absence of $\fn$---they interpret elements of $\ST^*_\fn(\boxtimes_\ff(\Wbb))$ as invariant linear functionals on $\boxtimes_\ff(\Wbb)$ rather than as conformal blocks associated to $\boxtimes_\ff(\Wbb)$ and $\fn$. Our formulation, which is based on the disjoint sewing of conformal blocks, allows us to view Theorem \ref{lb23} as a special (yet significant) case of the general sewing-factorization Thm. \ref{lb43}, the latter being more useful in applications.
\end{rem}

In the next chapter, we will need a variant of Thm. \ref{lb23} to prove the general versions of the sewing-factorization theorems. Let
\begin{align}\label{eq89}
\fq=\big(\infty,0;1/\zeta,\zeta\big|\Pbb^1\big|1;\zeta-1\big)
\end{align}
Again, $\zeta$ denotes the standard coordinate of $\Cbb$. Associate $\Vbb$ to the incoming marked point $1$, which gives a dual fusion product $(\bbs_\fq(\Vbb),\aleph)$ where $\aleph:\bbs_\fq(\Vbb)\otimes\Vbb\rightarrow\Cbb$ is the canonical conformal block. Recall that $\idt\in\Vbb$ is the vacuum vector. It is easy to check that
\begin{gather}\label{eq88}
\upomega=\aleph(-\otimes\idt):\bbs_\fq(\Vbb)\rightarrow\Cbb
\end{gather}  
is an element of $\ST^*_\fn(\bbs_\fq(\Vbb))$.

\begin{co}\label{lb32}
Let $p_1\in\MD_r^\times,p_2\in\MD_\rho^\times$. Then the linear map
\begin{gather}\label{eq44}
\ST^*_{\ff}(\boxtimes_\fq(\Vbb)\otimes\Wbb)\rightarrow \ST^*_{\MS(\ff\sqcup\fn)_{p_1,p_2}}(\Wbb)\qquad\upphi\mapsto \MS(\upphi\otimes\upomega)\big|_{p_1,p_2}
\end{gather}
is an isomorphism.
\end{co}

Here, $\MS(\upphi\otimes\upomega)$ is defined in a similar way to \eqref{eq63a}. (See also Def. \ref{lb42}.)

\begin{proof}
By the universal property for $\gimel$ (cf. Def. \ref{lb13}), the linear map
\begin{align*}
\Hom_{\Vbb^{\otimes 2}}(\boxtimes_\fq(\Vbb),\bbs_\ff(\Wbb))\rightarrow \ST^*_{\ff}(\boxtimes_\fq(\Vbb)\otimes\Wbb)\qquad T\mapsto \gimel\circ(T\otimes\id_{\Wbb})
\end{align*}
is an isomorpism. Therefore, to prove that \eqref{eq44} is an isomorphism, it suffices to prove that the linear map
\begin{gather*}
\Hom_{\Vbb^{\otimes 2}}(\boxtimes_\fq(\Vbb),\bbs_\ff(\Wbb))\rightarrow \ST^*_{\MS(\ff\sqcup\fn)_{p_1,p_2}}(\Wbb)\\
T\mapsto \MS((\gimel\circ (T\otimes\id_\Wbb))\otimes\upomega)\big|_{p_1,p_2}=\MS(\gimel\otimes (\upomega\circ T^\tr))
\end{gather*}
is an isomorphism. By Thm. \ref{lb23}, it suffices to prove that
\begin{gather*}
\Hom_{\Vbb^{\otimes 2}}(\boxtimes_\fq(\Vbb),\bbs_\ff(\Wbb))\rightarrow \ST^*_\fn(\boxtimes_\ff(\Wbb))\\
T\mapsto \upomega\circ T^\tr=\aleph \circ (T^\tr\otimes\idt)
\end{gather*}
But this follows from the universal property for $(\bbs_\fq(\Vbb),\aleph)$ (i.e., $T\mapsto \aleph\circ(T^\tr\otimes\id_\Vbb)$ is an isomorphism) and the propagation of conformal blocks, which says that
\begin{gather*}
\ST^*_\fq(\boxtimes_\ff(\Wbb)\otimes\Vbb)\rightarrow\ST^*_\fn(\boxtimes_\ff(\Wbb))\qquad \uplambda\mapsto \uplambda (-\otimes\idt)
\end{gather*}
is a linear isomorphism, cf. \cite[Cor. 2.44]{GZ1}.
\end{proof}

\begin{rem}
Assume that $\Vbb$ is rational, and let $\Irr$ be a set of representatives of isomorphism classes of simple objects of $\Mod(\Vbb)$. Let
\begin{gather*}
\bbs_\fq(\Vbb)=\bigoplus_{\Mbb\in\Irr}\Mbb'\otimes\Mbb\\
\gimel:\bbs_\fq(\Vbb)\otimes\Vbb\rightarrow\Cbb\qquad m'\otimes m\otimes v\mapsto\bk{m',Y(v,1)m}
\end{gather*}
It is easy to show that $(\bbs_\fq(\Vbb),\gimel)$ is a dual fusion product of $\Vbb$ along $\fq$, and $\upomega|_{\Mbb'\otimes\Mbb}$ is the evaluation pairing. Given $p_0\in\MD_{r\rho}^\times$, we write $p_0=p_1p_2$ where $p_1\in\MD_r^\times,p_2\in\MD_\rho^\times$, and recall the isomorphism \eqref{eq61}. Then  Cor. \ref{lb32} asserts that
\begin{gather}\label{eq71}
\bigoplus_{\Mbb\in\Irr}\ST_{\wtd\fx}^*(\Mbb'\otimes\Mbb\otimes\Wbb)\longrightarrow\ST_{\fx_{p_0}}^*(\Wbb)\qquad \oplus_\Mbb\upphi_\Mbb\mapsto \sum_\Mbb \MS\upphi_\Mbb\big|_{p_0}
\end{gather}
is a linear isomorphism, where
\begin{gather*}
\MS\upphi_\Mbb\big|_{p_0}:\Wbb\rightarrow\Cbb\qquad w\mapsto \sum_{\lambda\in\Cbb}\wick{\upphi\big(q^{L(0)}P_\lambda\c1-\otimes P_\lambda\c1-\otimes w\big)}
\end{gather*}
In particular, we obtain the \textbf{factorization formula}
\begin{align}\label{eq72}
\dim\ST_{\fx}^*(\Wbb)\big|_{p_0}=\sum_{\Mbb\in\Irr}\dim\ST_{\wtd\fx}^*(\Mbb'\otimes\Mbb\otimes\Wbb)
\end{align}
originally proved in \cite{DGT2} using algebro-geometric methods. Therefore, the present paper may also be regarded as providing an alternative, complex-analytic proof of the factorization formula for rational $C_2$-cofinite VOAs.
\end{rem}

\subsection{The lift $\widetilde{\mathfrak{y}}$ of $q\partial_q$ and its associated connection $\kabla$}

The main challenge in Thm. \ref{lb23} is proving that the sewing map \eqref{eq60} is surjective. Our strategy is as follows. Let $p_0=p_1p_2$, and choose $\uppsi_{p_0}\in\ST^*_{\MS(\ff\sqcup\fn)_{p_1,p_2}}(\Wbb)$. To show that \eqref{eq60} maps some $\upchi\in\ST^*_\fn(\boxtimes_\ff(\Wbb))$ to $\uppsi_{p_0}$, we first extend $\uppsi_{p_0}$ to a multivalued parallel section $\uppsi$ of the conformal block bundle $\ST^*_\fx(\Wbb)|_{\MD_{r\rho}^\times}$. We then show that $\uppsi$ has a logarithmic $q$-expansion and use the coefficients in this expansion to define $\upchi$.

To implement this strategy, we begin by reviewing the definition of (logarithmic) connections on sheaves of conformal blocks.

Since $\MB=\MD_{r\rho}$ is a Stein manifold, by \cite[Rem. 2.19]{GZ2}, the differential map $d\pi$ in \eqref{eq14} gives rise to a surjective map 
\begin{align}
   d\pi: H^0\big(\MC,\Theta_\MC(-\log \MC_\Delta+\blt S_\fx)\big)\twoheadrightarrow H^0\big(\MC,\pi^*\Theta_\MB(-\log \Delta)(\blt S_\fx)\big)
\end{align}
Therefore, for the element $\yk=q\partial_q$ of $H^0\big(\MB,\Theta_\MB(-\log \Delta)\big)$, we can find a \textbf{lift \pmb{$\wtd\yk$} of \pmb{$q\partial_q$}}, i.e., and element $\wtd \yk\in H^0\big(\MC,\Theta_\MC(-\log \MC_\Delta+\blt S_\fx)\big)$ satisfying $d\pi(\wtd \yk)=\pi^*(q\partial_q)$. (In the following, we abbreviate $\pi^*(q\partial_q)$ to $q\partial_q$ as usual.)

Throughout this chapter, we fix a lift $\wtd \yk$ of $q\partial_q$. A detailed description of $\wtd\yk$ is given in the next subsection.

\subsubsection{The $q$-expansion of the vertical part of $\wtd\yk$}

Let us expand the ``vertical part" of $\wtd\yk$ into a $q$-power series $\sum_{n\in\Nbb}\wtd\yk_n^\perp q^n$ as in Subsec. \ref{lb18}. See the proof of \cite[Thm. 11.4]{Gui-sewingconvergence} or \cite[Thm. 4.9]{GZ2} for more discussions.

Choose any precompact open subset $\wtd U\subset \wtd \MC-Z'_\fxtd$ equipped with a univalent $\eta\in \MO(\wtd U)$. Choose $0<\delta<r\rho$ small enough such that $\wtd U\times \MD_\delta$ can be viewed as an open subset of $\MC$. After extending $\eta$ constantly to a fiberwise univalent function on $U:=\wtd U\times\MD_\delta$, we have
\begin{subequations}\label{eq35}
\begin{align}\label{eq35a}
    \wtd \yk |_U=h(\eta,q)\partial_\eta +q\partial_{q}
\end{align}
where $h\in\mc O\big((\eta,q)(U-Z_\fx)\big)$ has finite poles at $(\eta,q)(Z_\fx)$. Write
\begin{align}\label{eq35b}
h=\sum_{n \in \Nbb} h_{n}(\eta)q^{n}
\end{align}
where each $h_{n}\in \mc O\big(\eta(\wtd U-Z_\fxtd)\big)$ has finite poles at $\eta(Z_\fxtd)$, and set
\begin{align}\label{eq35c}
    \wtd \yk_{n}^\perp =h_{n}(\eta)\partial_\eta\qquad\in H^0(\wtd U-\{\sgm',\sgm''\},\Theta_{\wtd \MC}(\blt Z_\fxtd))
\end{align}
\end{subequations}
It is easy to see that $\yk_{n}^\perp$ is independent of the choice of $\eta$. Hence one obtains $\wtd\yk_n^\perp\in H^0(\wtd\MC-Z'_\fxtd,\Theta_{\wtd \MC}(\blt Z_\fxtd))$ whose local expression is given by \eqref{eq35}.

Let us show that $\wtd\yk_n^\perp$ has finite poles at $Z'_\fxtd$. Recall Subsec. \ref{lb3} for the meanings of $W,W',W''$. Due to \eqref{eq34}, we can write 
\begin{gather*}
      \wtd \yk|_W=a(\xi,\varpi)\xi\partial_\xi+b(\xi,\varpi)\varpi\partial_\varpi
\end{gather*}
where $a,b\in \MO(\MD_r\times \MD_\rho)$. Since  $d\pi(\wtd\yk)=q\partial_q$, by \eqref{eq73}, we have
\begin{align}\label{eq53}
a+b=1
\end{align}
We take power series expansions
\begin{gather*}
a(\xi,\varpi)=\sum_{m,n\in \Nbb}a_{m,n}\xi^m\varpi^n\qquad
b(\xi,\varpi)=\sum_{m,n\in \Nbb}b_{m,n}\xi^m\varpi^n
\end{gather*}
where $a_{m,n},b_{m,n}\in\Cbb$. Under the sets of coordinates $(\xi,q)$ and $(q,\varpi)$ respectively, we can write\begin{subequations}
\begin{gather}
\wtd\yk\big|_{W'}=a(\xi,q/\xi)\xi\partial_\xi+q\partial_q=\sum_{n\geq0,l\geq-n}a_{l+n,n}\xi^{l+1}q^n\partial_\xi+q\partial_q\\
\wtd\yk\big|_{W''}=b(q/\varpi,\varpi)\varpi\partial_\varpi+q\partial_q=\sum_{m\geq0,l\geq-m}b_{m,l+m}\varpi^{l+1}q^m\partial_\varpi+q\partial_q
\end{gather}
\end{subequations}
in the spirit of \eqref{eq35}. Hence
\begin{subequations}\label{eq39}
\begin{gather}
\wtd\yk^\perp_n\big|_{V'-\{\sgm'\}}=\sum_{l\geq-n}a_{l+n,n}\xi^{l+1}\partial_\xi\\
\wtd\yk_n^\perp\big|_{V''-\{\sgm''\}}=\sum_{l\geq-n} b_{n,l+n}\varpi^{l+1}\partial_{\varpi}
\end{gather}
\end{subequations}
This proves that $\wtd\yk^\perp_n\in H^0(\wtd\MC,\Theta_\MCtd(\blt\SXtd))$.

\subsubsection{The connection $\nabla$ and its shifted connection $\nabla^\varkappa$}

For each $1\leq i\leq N$, recall that the local coordinate $\eta_i\in\MO(\wtd U_i)$ at $\sgm_i$ is extended constantly to $\eta_i\in\MO(U_i)$ at $\sgm_i(\MB)$. Therefore, $U_i$ has a set of coordinates $(\eta_i,q)$. By \eqref{eq35a}, we can find $h_i\in \MO\big((\eta_i,q)(U_i-S_\fx)\big)$ such that 
\begin{align}\label{geometry2}
    \wtd \yk|_{U_i}=h^i(\eta_i,q)\partial_{\eta_i}+q\partial_q
\end{align}
where $\eta_i^k h^i(\eta_i,q)$ is holomorphic on $U_i$ for some $k\in \Nbb$. Define 
\begin{gather}\label{eq37}
    \begin{gathered}
    \upnu(\wtd \yk)\in H^0\big(U_1\cup \cdots \cup U_N,\SV_\fx\otimes \omega_{\MC/\MB}(\blt S_\fx)\big)\\
     \MU_\varrho(\eta_i)\upnu(\wtd \yk)\vert_{U_i}=h^i(\eta_i,q)\cbf d\eta_i
    \end{gathered}
    \end{gather}
(Recall that $\cbf$ is the conformal vector of $\Vbb$.) The sheaf map $\nabla_{q\partial_q}:\Wbb\otimes_\Cbb \MO_\MB\rightarrow \Wbb\otimes_\Cbb \MO_\MB$ is defined as follows. For each open $V\subset \MB$ and $w\in \Wbb\otimes_\Cbb \MO(V)$, we set 
\begin{align}\label{eq38}
    \nabla_{q\partial_q} w=q\partial_q w-\upnu(\wtd \yk)\cdot w\qquad\in\Wbb\otimes_\Cbb\MO(V)
\end{align}
where $\upnu(\wtd \yk)\cdot w$ is the residue action of $\upnu(\wtd \yk)$ on $w$, cf. \eqref{eq33}. By \cite[Thm. 2.23]{GZ2},
\begin{align*}
\nabla_{\partial_q}=q^{-1}\nabla_{q\partial_q}
\end{align*}
descends to a sheaf map on $\ST_\fx(\Wbb)|_{\MD_{r\rho}^\times}$. Moreover, if we set $\nabla_{g\partial_q}=g\nabla_{\partial_q}$ for any $g\in\MO_{\MD^\times_{r\rho}}$, then $\nabla$ is a connection on the vector bundle $\ST_\fx(\Wbb)|_{\MD_{r\rho}^\times}$. See \cite[Thm. 2.23]{GZ2} for details.

To describe $\nabla_{q\partial_q}$ in more detail, we define
\begin{gather}
    \upnu(\wtd \yk_{n}^\perp)\in H^0\big(\wtd U_1\cup \cdots \cup \wtd U_N\cup V'\cup V'',\SV_{\wtd\fx}\otimes \omega_{\wtd\MC}(\blt S_{\wtd \fx})\big)
\end{gather}
in a similar way to \eqref{eq37}. Namely, in view of \eqref{eq35c} and \eqref{eq39}, let
\begin{subequations}\label{eq74}
\begin{gather}
    \MU_\varrho(\eta_i)\upnu(\wtd \yk_{n}^\perp)|_{\wtd U_i}=h_n(\eta_i)\cbf d\eta_i\label{eq74a}\\
    \MU_\varrho(\xi)\upnu(\wtd \yk_{n}^\perp)|_{V'}=\sum_{l\geq -n} a_{l+n,n}\xi^{l+1}\cbf d\xi\label{eq74b}\\
    \MU_\varrho(\varpi)\upnu(\wtd \yk_{n}^\perp)|_{V''}=\sum_{l\geq -n} b_{n,l+n}\varpi^{l+1}\cbf d\varpi\label{eq74c}
\end{gather}
\end{subequations}
Thus \eqref{eq38} becomes
\begin{align}\label{geometry3}
    \nabla_{q\partial_q} w=q\partial_qw-\sum_{n\in\Nbb} \upnu(\wtd\yk_n^\perp)w\cdot q^n=q\partial_q w-\sum_{i=1}^N\sum_{n\in\Nbb}\Res_{\eta_i=0}~h_n(\eta_i)Y_i(\cbf,\eta_i)wd\eta_i\cdot q^n
\end{align}

\begin{pp}\label{geometry6}
For each $n\in \Nbb$, there exists $\#(\wtd \yk_{n}^\perp)\in \Cbb$  independent of the choice of $\Wbb$ such that for each $\Mbb\in\Mod(\Vbb^{\otimes 2})$ associated to $\sgm',\sgm''$, the residue action of $\upnu(\wtd \yk_{n}^\perp)$ on $\ST_{\wtd \fx}(\Mbb\otimes \Wbb)$ equals the multiplication by $\#(\wtd \yk_{n}^\perp)$, i.e., for each $w\in\Wbb,m\in\Mbb$ we have
\begin{align}
\upnu(\wtd \yk_{n}^\perp) m\otimes w+m\otimes \upnu(\wtd \yk_{n}^\perp) w=\#(\wtd\yk^\perp_n)\cdot m\otimes w\qquad\text{in }\ST_\fxtd(\Mbb\otimes\Wbb)
\end{align}
Moreover, the following power series converges a.l.u. to an element in $\MO(\MB)$.
    \begin{align}\label{eq51}
\varkappa:=\sum_{n\in \Nbb}\#(\wtd \yk_{n}^\perp) q^n
    \end{align}
\end{pp}

We are mainly interested in the case that $\Mbb=\bbs_\fx(\Wbb)$.

\begin{proof}
The first part (about $\#(\wtd\yk_n^\perp)$) holds more generally when $\wtd\yk_n^\perp$ is replaced by any element of $H^0(\wtd\MC,\Theta_\MCtd(\blt\SXtd))$. Cf. \cite[Thm. 2.29]{GZ2} or  \cite[Prop. 9.2]{Gui-sewingconvergence}. Since $\#(\wtd\yk_n^\perp)$ is independent of the choice of $\Mbb$, to prove that $\varkappa$ converges, one can choose $\Mbb$ to be the tensor product of two objects of $\Mod(\Vbb)$, e.g., $\Mbb=\Vbb\otimes\Vbb$. Similarly, one can choose $\Wbb$ to be $\Vbb^{\otimes N}$. Then the convergence of $\varkappa$ follows from \cite[Prop. 11.12]{Gui-sewingconvergence}. (See also Step 5 of the proof of \cite[Thm. 4.9]{GZ2}.)
\end{proof}

Unfortunately, the sewing of a conformal block is not parallel under $\nabla$. To fix this issue, we consider the \textbf{shifted (logarithmic) connection} \pmb{$\nabla^\varkappa$} on the $\MO_\MB$-module $\ST_\fx(\Wbb)$ defined by
\begin{align}\label{eq52}
\kabla_{q\partial q}w=\nabla_{q\partial q}w+\varkappa\cdot w
\end{align}
for all $w\in \Wbb\otimes_\Cbb\MO_\MB$. When restricted to $\MD_{r\rho}^\times$, the \textbf{dual connection} $\kabla$ on the vector bundle $\ST^*_\fx(\Wbb)|_{\MD^\times_{r\rho}}$ is given by
\begin{gather}\label{eq75}
\bk{\kabla_{q\partial_q}\uppsi,w}=q\partial_q\bk{\uppsi,w}-\bk{\uppsi,\kabla_{q\partial_q}w}
\end{gather}
for all $\uppsi\in \ST^*_\fx(\Wbb)|_{\MD^\times_{r\rho}}$.

\subsection{The parallel section $\uppsi$ and its logarithmic $q$-expansion}

Fix $p_0\in \MB-\Delta=\MD_{r\rho}^\times$ and a conformal block $\uppsi_{p_0}\in \ST_{\fx_q}^*(\Wbb)$. Fix an argument $\arg p_0$. Recall \eqref{eq127} for the meaning of $\wht\MD_{r\rho}$. Then, pulling back the vector bundle $\ST^*_\fx(\Wbb)|_{\MD^\times_{r\rho}}$ and its connection $\kabla$ to $\wht\MD_{r\rho}^\times$, and using the isomorphism \eqref{eq104} in Rem. \ref{lb49}, we obtain a $\kabla$-parallel section $\uppsi\in H^0(\wht\MD_{r\rho}^\times,\ST^*_\fx(\Wbb))$ whose initial value at $p_0$ (with argument $\arg p_0$) is $\uppsi_{p_0}$. Thus, by \eqref{eq75}, for any $w\in\Wbb\otimes_\Cbb\MO_\MB$ we have
\begin{align}\label{eq40}
q\partial_q\bk{\uppsi,w}=\bk{\uppsi,\kabla_{q\partial_q}w} \qquad \uppsi|_{p_0}=\uppsi_{p_0}
\end{align}
where the first relation holds in $\MO_{\wht\MD_{r\rho}^\times}$.

We view $\Wbb$ as a linear subspace of $\Wbb\otimes_\Cbb\MO(\MB)$ by viewing each $w\in\Wbb$ as $w\otimes 1$. Then $\uppsi$ gives a linear map
\begin{align*}
\uppsi:\Wbb\rightarrow\MO(\wht\MD_{r\rho}^\times)\qquad w\mapsto \uppsi(w)\equiv\bk{\upphi,w}
\end{align*}

\begin{lm}\label{lb21}
There exist $L\in\Nbb$, a finite subset $E\subset\Cbb$, and a unique collection of linear functionals
\begin{align*}
\uppsi_{n,l}:\Wbb\rightarrow\Cbb\qquad\text{where }l\in\{0,1,\dots,L\}\text{, }n\in \Cbb\text{, and }\uppsi_{n,l}=0\text{ if }n\notin E+\Nbb
\end{align*}
such that for each $w\in\Wbb$, the element $\uppsi(w)\in\MO(\wht\MD_{r\rho}^\times)$ equals
\begin{align}\label{eq78}
    \uppsi(w)=\sum_{l=0}^L\sum_{n\in \Cbb} \uppsi_{n,l}(w)\cdot q^n (\log q)^l
\end{align}
where the RHS of \eqref{eq78} converges a.l.u. on $\MD^\times_{r\rho}$ in the sense of Def. \ref{lb19}.

\end{lm}

\begin{proof}
The uniqueness follows from Prop. \ref{geometry11}. Then for each $\tipar_1,\dots,\tipar_N\in\Rbb$, $\uppsi$ restricts to a linear map $\uppsi^{\leq \tipar_\blt}:\Wbb_{[\leq \tipar_\blt]}\rightarrow\MO(\wht\MD^\times_{r\rho})$, i.e., a $\Wbb_{[\leq \tipar_\blt]}$-valued multivalued holomorphic function on $\MD^\times_{r\rho}$. Let
\begin{align*}
\mc J=H^0(\MC,\SV_\fx\otimes\omega_{\MC/\MB}(\blt\SX))\cdot (\Wbb\otimes_\Cbb\MO(\MB))
\end{align*}
where $\Span_\Cbb$ has been suppressed. Then $\mc J$ is an $\MO(\MB)$-submodule of $\SW_\fx(\Wbb):=\Wbb\otimes_\Cbb\MO(\MB)$. By \cite[Thm. 1.16]{GZ2}, the quotient $\MO(\MB)$-module $\SW_\fx(\Wbb)/\mc J$ is finitely generated. Therefore, there exist $\tipar_1,\dots,\tipar_N\in\Rbb$ such that $\SW_\fx(\Wbb)/\mc J$ is $\MO(\MB)$-generated by the elements of $\Wbb_{[\leq \tipar_\blt]}$.  In the following, we fix these $\tipar_\blt$.

Now let $(e_i)_{i\in I}$ be a (finite) basis of $\Wbb_{[\leq \tipar_\blt]}$. Since $\kabla_{q\partial_q}e_i\in\SW_\fx(\Wbb)$ (cf. \eqref{eq38}), there exists a family $(\Omega_{i,j})_{i,j\in I}$ in $\MO(\MB)$ such that
\begin{align*}
\kabla_{q\partial_q}e_i=\sum_{i,j\in I}\Omega_{i,j}e_j\quad\mod\quad\mc J
\end{align*}
for all $i\in I$. Note that $\uppsi$ is a conformal block and hence is vanishing on $\mc J$, cf. \eqref{eq41}. Therefore, by \eqref{eq40}, we have
\begin{align*}
q\partial_q \uppsi^{\leq \tipar_\blt}(e_i)=\sum_{j\in I}\Omega_{i,j}\uppsi^{\leq \tipar_\blt}(e_j)
\end{align*}
for all $i\in I$. By the basic theory of linear differential equations with simple poles (cf. \cite[Thm. 4.7]{Tes-ODE}, for example), we can find $L\in\Nbb$ and a finite set $E\subset\Cbb$ such that
\begin{align}\label{eq42}
\uppsi^{\leq \tipar_\blt}(w)=\sum_{l=0}^L\sum_{n\in E+\Nbb} \uppsi_{n,l}^{\leq \tipar_\blt}(w)\cdot q^n (\log q)^l\qquad\text{for all }w\in\Wbb_{[\leq \tipar_\blt]}
\end{align}
where each $\uppsi_{n,l}^{\leq \tipar_\blt}:\Wbb_{[\leq \tipar_\blt]}\rightarrow\Cbb$ is a linear functional.

Let $(m_j)_{j\in J}$ be a basis of $\Wbb$. For each $j\in J$, choose $w_j\in\Wbb_{[\leq \tipar_\blt]}\otimes_\Cbb\MO(\MB)$ such that $m_j-w_j\in\mc J$. So $\uppsi(m_j)$ equals $\uppsi(w_j)$ as an element of $\MO(\wht\MD_{r\rho}^\times)$. By \eqref{eq42}, we can write
\begin{align*}
\uppsi(w_j)=\sum_{l=0}^L\sum_{n\in E+\Nbb}\lambda_{n,l,j}\cdot q^n(\log q)^l
\end{align*}
where $\lambda_{n,l,j}\in\Cbb$. For each $0\leq l\leq L$ and $n\in E+\Nbb$, define $\uppsi_{n,l}:\Wbb\rightarrow\Cbb$ to be the unique linear functional satisfying $\uppsi_{n,l}(m_j)=\lambda_{n,l,j}$ for all $j\in J$. Then \eqref{eq78} is satisfied.
\end{proof}

\subsection{Important properties of the logarithmic $q$-expansion of $\uppsi$}

Let $(\uppsi_{n,l})$ be as in Lem. \ref{lb21}. From now until the end of this chapter, we view $\bbs_\ff(\Wbb)$ as a linear subspace of $\Wbb^*$ and view $\gimel:\bbs_\ff(\Wbb)\otimes\Wbb\rightarrow\Cbb$ as the restriction of the evaluation pairing $\Wbb^*\otimes\Wbb\rightarrow\Cbb$, cf. Thm. \ref{lb14}. Moreover, we assume, without loss of generality, that the set $E$ in Lem. \ref{lb21} is chosen in such a way that $E+\Nbb$ equals the disjoint union $\bigsqcup_{e\in E}(e+\Nbb)$. 

\begin{lm}\label{lb24}
Let $0<\eps<r\rho$ and $v\in H^0\big(\pi^{-1}(\MD_\eps),\SV_\fx\otimes \omega_{\MC/\MB}(\blt S_\fx)\big)$ with $q$-expansion $v=\sum_{k\in\Nbb}v_kq^k$ (cf. Def. \ref{lb20}). Then for each $w\in\Wbb$, $0\leq l\leq L$, and $m\in\Cbb$, we have
\begin{align}\label{eq77}
    \sum_{k\in\Nbb}\uppsi_{m-k,l}(v_k\cdot w)=0
\end{align}
\end{lm}

Recall that $v_k\in H^0\big(\wtd\MC,\scr V_{\wtd\fx}\otimes\omega_{\wtd\MC}(\blt\SXtd)\big)$, and $v_k\cdot w=\sum_{i=1}^N v_k*_iw\in\Wbb$ is the residue action of $v_k$ on $w$ (cf. \eqref{eq33}). Note also that the LHS of \eqref{eq77} is a finite sum, because $\uppsi_{n,l}=0$ when $n\in\Cbb$ and $\Re(n)\ll0$.

\begin{proof}
Step 1. Consider the residue action $v\cdot w\in\Wbb\otimes_\Cbb\MO(\MB)$ defined by \eqref{eq33}. Since $\uppsi$ is a conformal block associated to $\fx|_{\MD_{r\rho}^\times}$, it vanishes at
\begin{align}
v\cdot w=\sum_{k\in\Nbb}(v_k\cdot w)\cdot q^k  \label{eq76}
\end{align}
Here, if we choose $\tipar_1,\dots,\tipar_N\in\Rbb$ such that $w\in\Wbb_{[\leq \tipar_\blt]}$, then the series on the RHS above converges a.l.u. to a $\Wbb_{[\leq \tipar_\blt]}$-valued holomorphic function on $\MD_\eps$.

Roughly speaking, we can obtain \eqref{eq77} by substituting the series \eqref{eq78} and \eqref{eq76} into the equation $\uppsi(v\cdot w)=0$. However, this formal manipulation requires justification, as the evaluation of $\uppsi$ on a section of $\Wbb\otimes_\Cbb\MO_\MB$ is defined in a sheaf-theoretic manner rather than using formal series. (That is, we first define $\uppsi$ on $\Wbb$ by solving a differential equation, and then extend it to $\Wbb\otimes_\Cbb\MO_\MB$ by $\MO_\MB$-linearity.) In the following, we provide a justification for this.\\[-1ex]

Step 2. Note that if $X$ is a complex manifold and $\upphi:\Wbb\otimes_\Cbb\MO_X\rightarrow\MO_X$ is an $\MO_X$-module morphism, then for any sequence  $f_n$ in $\Wbb_{[\leq\tipar_\blt]}\otimes_\Cbb\MO(X)$ (viewed as holomorphic functions $X\rightarrow\Wbb_{[\leq\tipar_\blt]}$) converging locally uniformly to $f\in\Wbb_{[\leq\tipar_\blt]}\otimes_\Cbb\MO(X)$, the sequence $\upphi(f_n)$ (in $\MO(X)$) clearly converges locally uniformly on $X$ to $\upphi(f)$.

Now, for each $0\leq l\leq L$, $e\in E$, let $\uplambda_{e,l}:\Wbb\otimes_\Cbb\MO_\MB\rightarrow\MO_\MB$ be the $\MO_\MB$-module morphism determined by the fact that for each $\fk w\in\Wbb$,
\begin{align*}
\uplambda_{e,l}(\fk w)=\sum_{n\in\Nbb}\uppsi_{e+n,l}(\fk w)q^n
\end{align*}
where the RHS converges a.l.u. on $\MD_{r\rho}$. By the previous paragraph, for each $m\in\Zbb$ we have
\begin{align*}
\Res_{q=0}~\uplambda_{e,l}(v\cdot w)\cdot q^{-m-1}dq=\Res_{q=0}~\sum_{k\in\Nbb}\uplambda_{e,l}(v_k\cdot w)\cdot q^{k-m-1}dq
\end{align*}
By the a.l.u. convergence, the residue on the RHS (viewed as a contour integral) commutes with $\sum_k$. Therefore
\begin{align*}
&\Res_{q=0}~\uplambda_{e,l}(v\cdot w)\cdot q^{-m-1}dq=\sum_{k\in\Nbb}\Res_{q=0}~\uplambda_{e,l}(v_k\cdot w)\cdot q^{k-m-1}dq\\
=&\sum_{k\in\Nbb}\Res_{q=0}~\sum_{n\in\Nbb}\uppsi_{e+n,l}(v_k\cdot w)q^{n+k-m-1}dq=\sum_{k\in\Nbb} \uppsi_{e+m-k,l}(v_k\cdot w)
\end{align*}
where the last term is a finite sum. Therefore
\begin{align*}
\uplambda_{e,l}(v\cdot w)=\sum_{m\in\Nbb} \sum_{k\in\Nbb} \uppsi_{e+m-k,l}(v_k\cdot w)q^m
\end{align*}
Here, the outer sum $\sum_m$ converges a.l.u. on $\MD_\eps$, and the inner sum $\sum_k$ is finite. Since
\begin{align*}
\uppsi=\sum_{0\leq l\leq L}\sum_{e\in E}\uplambda_{e,l}\cdot q^e(\log q)^l
\end{align*}
holds on $\Wbb$ and hence (by $\MO_\MB$-linearity) on $\Wbb\otimes_\Cbb\MO(\MD_\eps^\times)$, we obtain in $\MO(\wht\MD_\eps^\times)$ that
\begin{align*}
\uppsi(v\cdot w)=\sum_{0\leq l\leq L}\sum_{m\in\Cbb}\sum_{k\in\Nbb}\uppsi_{m-k,l}(v_k\cdot w)\cdot q^m(\log q)^l
\end{align*}
Since the $\uppsi(v\cdot w)$ is zero for all $q\in\wht\MD_\eps^\times$, by Prop. \ref{geometry11}, we obtain \eqref{eq77}.
\end{proof}

\begin{lm}\label{lb25}
    For all $n\in \Cbb,0\leq l\leq L$, we have $\uppsi_{n,l}\in \bbs_\ff(\Wbb)$.
\end{lm}
\begin{proof}
It suffices to assume that $n\in E+\Nbb$. Recall Def. \ref{lb22}. Then, for each $\hbar\in \Nbb$ and $\sigma\in H^0\big(\wtd \MC,\SV_{\ff,\hbar,\hbar}\otimes \omega_{\wtd\MC}(\blt Z_{\wtd \fx})\big)$, we have that $\sigma+0\cdot q+\cdots +0\cdot q^\hbar$ is $\hbar$-compatible, cf. Def. \ref{lb7}. (Indeed, by setting $v_0=\sigma,v_1=\cdots=v_\hbar=0$, we obtain \eqref{eq22} with the family $(c^\alpha_{m,n})$, where the only potentially nonzero terms are $c_{m,0}^\alpha$ for $m>\hbar$ and $c_{0,n}^\alpha$ for $n>\hbar$. In fact, $c^\alpha_{m,0},c^\alpha_{0,n}$ are determined respectively by the Laurent series expansions of $\MU_\varrho(\xi)\sigma,\MU_\varrho(\varpi)\sigma$ near $\sgm',\sgm''$.)

Thus, by Thm. \ref{geometry1}, there exist $0<\eps<r\rho$ and $v\in H^0\big(\pi^{-1}(\MD_\eps),\SV_\fx\otimes \omega_{\MC/\MB}(\blt S_\fx)\big)$ with $q$-expansion $v=\sigma+\star q^{\hbar+1}+\star q^{\hbar+2}+\cdots$. Let us assume, at the beginning of the proof, that $\hbar\in\Nbb$ is large enough such that $n-k\notin E+\Nbb$ for any integer $k>\hbar$. Then Lem. \ref{lb24} implies $\uppsi_{n,l}(\sigma\cdot w)=0$ for each $w\in \Wbb$. So $\uppsi_{n,l}\in \ST_{\ff,\hbar,\hbar}^*(\Wbb)\subset \bbs_\ff(\Wbb)$.
\end{proof}

\begin{pp}\label{lb34}
    Assume that $f(\xi,\varpi)\in \Cbb[[\xi,\varpi]]$ and $0\leq \ell\leq L$. Then, for each $u\in \Vbb$, the following equation holds in $\bbs_\ff(\Wbb)\{q\}$.
    \begin{equation}\label{eq80}
    \begin{aligned}
        &\Res_{\xi=0}\sum_{n\in \Cbb} f(\xi,q/\xi)Y_+(\xi^{L(0)-1}u,\xi)\uppsi_{n,\ell}q^n d\xi\\
        =&\Res_{\varpi=0}\sum_{n\in \Cbb}f(q/\varpi,\varpi)Y_-(\varpi^{L(0)-1}\MU(\upgamma_1)u,\varpi)\uppsi_{n,\ell}q^n d\varpi
    \end{aligned}
\end{equation}
\end{pp}

This proposition can be viewed as an inverse of \cite[Prop. 4.6]{GZ2}. Thus, \cite[Rem. 4.7]{GZ2} can also help the reader understand the two residues in \eqref{eq80}.

\begin{proof}

It suffices to prove that for each $\hbar\in\Nbb$, Eq. \eqref{eq80} holds true mod $\mbb J_\hbar$, where
\begin{align*}
\mbb J_\hbar=\sum_{e\in E}q^{e+\hbar+1}\bbs_\ff(\Wbb)[[q]]
\end{align*}
Fix $\hbar\in\Nbb$. Write $D=E+\Nbb$ and $D_\hbar=E+\Nbb_{\leq\hbar}$.  Write 
    \begin{align*}
        f(\xi,\varpi)=\sum_{m,n=0}^\infty c_{m,n}\xi^m\varpi^n
    \end{align*}    
where each $c_{m,n}\in\Cbb$. Then, the LHS of \eqref{eq80},  which equals
\begin{align*}
  &\Res_{\xi=0}\sum_{n\in D} \sum_{k=0}^{+\infty} \sum_{t=-k}^{+\infty}c_{t+k,k}Y_+(\xi^{t+L(0)-1}u,\xi)\uppsi_{n,\ell}q^{k+n}d\xi
\end{align*}
is mod $\mbb J_\hbar$ equal to
\begin{align}\label{eq46}
  &\Res_{\xi=0}\sum_{n\in D_\hbar} \sum_{k=0}^\hbar \sum_{t=-k}^{+\infty}c_{t+k,k} Y_+(\xi^{t+L(0)-1}u,\xi)\uppsi_{n,\ell}q^{k+n}d\xi
\end{align}

Choose $Q\in\Nbb_{\geq\hbar}$ such that
\begin{align}\label{eq85}
\Res_{z=0}~Y_\pm(z^{Q-\hbar+L(0)+j}\fk v,z)\uppsi_{n,\ell}dz=0\qquad \forall n\in D_\hbar,0\leq \ell\leq L,\fk v\in\Vbb^{\leq M},j\in\Nbb
\end{align}
By Prop. \ref{geometry10}, we can find $v_0,\dots,v_\hbar\in H^0\big(\wtd \MC,\SV_{\wtd\fx}\otimes \omega_{\wtd \MC}(\blt S_{\wtd \fx})\big)$ satisfying \eqref{eq45}. Combining \eqref{eq45} (or \eqref{eq84}) with \eqref{eq85}, we see that
\begin{align*}
\eqref{eq46}=\sum_{n\in D_\hbar} \sum_{k=0}^\hbar\Res_{\xi=0}~ Y_+(\MU_\varrho(\xi)v_k,\xi)\uppsi_{n,\ell}q^{k+n}
\end{align*}
So the LHS of \eqref{eq80} is mod $\mbb J_\hbar$ equal to the RHS of the above equation. By a similar argument, the RHS of \eqref{eq80} is mod $\mbb J_\hbar$ equal to
\begin{align*}
-\sum_{n\in D_\hbar}\sum_{k=0}^\hbar \Res_{\varpi=0}~Y_-(\MU_\varrho(\varpi)v_k,\varpi)\uppsi_{n,\ell}q^{k+n}
\end{align*}

Recall that the evaluation pairing $\bk{\cdot,\cdot}:\bbs_\ff(\Wbb)\otimes \Wbb\rightarrow\Cbb$ is just the canonical conformal block $\gimel\in\ST^*_\ff(\bbs_\ff(\Wbb)\otimes\Wbb)$. So it vanishes on the residue action $v_k\cdot(\uppsi_{n,\ell}\otimes w)$ for each $w\in \Wbb$. In other words,
   \begin{align*}
    \Res_{\xi=0}\bk{Y_+(\MU_\varrho(\xi)v_k,\xi)\uppsi_{n,\ell},w}+\Res_{\varpi=0}\bk{Y_-(\MU_\varrho(\varpi)v_k,\varpi)\uppsi_{n,\ell},w}=-\uppsi_{n,\ell}(v_k\cdot w)
   \end{align*}
Therefore, when evaluated with $w$, the LHS minus the RHS of \eqref{eq80} is mod $\mbb J_\hbar$ equal to
\begin{align}\label{eq47}
-\sum_{n\in D_\hbar}\sum_{k=0}^\hbar \uppsi_{n,\ell}(v_k\cdot w)q^{k+n}
\end{align}
By Thm. \ref{geometry1} and Prop. \ref{lb33}, there exist $0<\eps<r\rho$ and $v\in H^0\big(\pi^{-1}(\MD_\eps),\SV_\fx\otimes \omega_{\MC/\MB}(\blt S_\fx)\big)$ whose $\hbar$-truncated $q$-expansion is $v_0+v_1q+\cdots+v_\hbar q^\hbar$. Write $v=\sum_{k=0}^\infty v_kq^k$. Then, by Lem. \ref{lb24}, we have
\begin{align*}
\sum_{n\in D}\sum_{k=0}^{+\infty}\uppsi_{n,\ell}(v_k\cdot w)q^{k+n}=0
\end{align*}
Therefore \eqref{eq47} equals zero mod $\mbb J_\hbar$. This proves that \eqref{eq80} holds mod $\mbb J_\hbar$.
\end{proof}

Recall \eqref{eq50} for the meaning of $Y_\pm(u)_{(k)}$ and $Y'_\pm(u)_{(k)}$.

\begin{co}\label{lb29}
    For each $u\in \Vbb,k\in \Zbb,0\leq l\leq L$, the following equation holds in $\bbs_\ff(\Wbb)\{q\}$.
    \begin{align}\label{eq79}
        \sum_{n\in \Cbb} Y_+(u)_{(k)}\uppsi_{n,l}q^n =\sum_{n\in \Cbb}Y_-'(u)_{(k)}\uppsi_{n,l}q^{n+k}
    \end{align}
    In particular, we have $Y_+(u)_{(k)}\uppsi_{n,l}=Y_-'(u)_{(k)}\uppsi_{n-k,l}$ and $L_+(k)\uppsi_{n,l}=L_-(-k)\uppsi_{n-k,l}$.
\end{co}
\begin{proof}
    Recall that $\MU(\upgamma_z)z^{L(0)}=z^{-L(0)}\MU(\upgamma_1)$. Setting $z=\varpi^{-1}$ and using the identity $\Res_{\varpi=0}g(\varpi)d\varpi=-\Res_{z=0}g(z^{-1})d(z^{-1})$ for any formal Laurent series $g(\varpi)$, we find that \eqref{eq80} becomes 
    \begin{equation}\label{eq81}
        \begin{aligned}
            &\Res_{\xi=0}\sum_{n\in \Cbb} f(\xi,q/\xi)Y_+(\xi^{L(0)-1}u,\xi)\uppsi_{n,l}q^n d\xi\\
            =&\Res_{z=0}\sum_{n\in \Cbb}f(qz,z^{-1})Y_-'(z^{L(0)-1}u,z)\uppsi_{n,l}q^n dz
        \end{aligned}
    \end{equation}
    When $k\geq 0$ (resp. $k<0$), we choose $f(\xi,\varpi)=\xi^k$ (resp. $\varpi^{-k}$). Then \eqref{eq81} becomes 
    \begin{align*}
        &\sum_{n\in \Cbb} Y_+(u)_{(k)}\uppsi_{n,l}q^n =\sum_{n\in \Cbb}Y_-'(u)_{(k)}\uppsi_{n,l}q^{n+k}\\
        \text{resp. }&\sum_{n\in \Cbb} Y_+(u)_{(k)}\uppsi_{n,l}q^{n-k} =\sum_{n\in \Cbb}Y_-'(u)_{(k)}\uppsi_{n,l}q^{n}
    \end{align*}
    So \eqref{eq79} is true. The rest of our corollary is obvious.
\end{proof}

\begin{pp}\label{lb27}
For each $0\leq l\leq L$, the following relations hold in $\bbs_\ff(\Wbb)\{q\}[\log q]$.
    \begin{align}\label{eq57}
        q\partial_q \uppsi=\sum_{l=0}^L\sum_{n\in \Cbb}L_+(0)\uppsi_{n,l}q^n (\log q)^l=\sum_{l=0}^L\sum_{n\in \Cbb}L_-(0)\uppsi_{n,l}q^n (\log q)^l
    \end{align}
In other words, we have $q\partial_q\uppsi=L_+(0)\uppsi=L_-(0)\uppsi$.
\end{pp}

\begin{proof}
    The second equality of \eqref{eq57} follows from Cor. \ref{lb29}. Choose any $w\in \Wbb$. By Prop. \ref{geometry6}, for each $n\in \Cbb,0\leq l\leq L,m\in\Nbb$, we have in $\ST_{\wtd\fx}(\bbs_\fxtd(\Wbb)\otimes\Wbb)$ that
    \begin{align*}
\upnu(\wtd\yk_m^\perp)\uppsi_{n,l}\otimes w+\uppsi_{n,l}\otimes \upnu(\wtd \yk_m^\perp)w=\#(\wtd\yk_m^\perp)\cdot \uppsi_{n,l}\otimes w
    \end{align*}
Applying the canonical conformal block $\gimel:\Wbb\otimes \bbs_\ff(\Wbb)\rightarrow \Cbb$ (i.e., the evaluation pairing) to the above equation yields
    \begin{align*}
\bk{\upnu(\wtd\yk_m^\perp)\uppsi_{n,l},w}+\bk{\uppsi_{n,l},\upnu(\wtd \yk_m^\perp)w}=\#(\wtd\yk_m^\perp)\cdot\uppsi_{n,l}(w)
    \end{align*}
Hence, by \eqref{eq51}, we have in $\Cbb[[q]]$ that
\begin{align*}
\sum_{m\in\Nbb}\bk{\upnu(\wtd\yk_m^\perp)\uppsi_{n,l},w}q^m+\sum_{m\in\Nbb}\bk{\uppsi_{n,l},\upnu(\wtd \yk_m^\perp)w}q^m=\varkappa\bk{\uppsi_{n,l},w}
\end{align*}
It follows that in $\Cbb\{q\}[\log q]$ we have
\begin{align}\label{eq62}
\sum_{m\in\Nbb}\bk{\upnu(\wtd\yk_m^\perp)\uppsi,w}q^m+\sum_{m\in\Nbb}\bk{\uppsi,\upnu(\wtd \yk_m^\perp)w}q^m=\varkappa\bk{\uppsi,w}
\end{align}

On the other hand, by \eqref{eq40} and \eqref{eq52}, we have in $\MO(\wht\MD_{r\rho}^\times)$ that
\begin{align*}
q\partial_q\bk{\uppsi,w}=\bk{\uppsi,\nabla_{q\partial_q}w+\varkappa w}
\end{align*}
Applying \eqref{geometry3} and using an argument similar to that in the proof of Lem. \ref{lb24} to transition from the sheaf-theoretic equation to the formal equation, we obtain the following equation in $\Cbb\{q\}[\log q]$.
\begin{align*}
q\partial_q\bk{\uppsi,w}=-\sum_{m\in\Nbb}\bk{\uppsi,\upnu(\wtd\yk_m^\perp)w}q^m+\varkappa\bk{\uppsi,w}
\end{align*}
Combining this equation with \eqref{eq62}, we get $q\partial_q\bk{\uppsi,w}=\sum_{m\in\Nbb}\bk{\upnu(\wtd\yk_m^\perp)\uppsi,w}q^m$. Therefore
\begin{align}\label{eq54}
q\partial_q\uppsi=\sum_{m\in\Nbb}\upnu(\wtd\yk_m^\perp)\uppsi\cdot q^m=\sum_{m\in\Nbb}\upnu(\wtd\yk_m^\perp)_+\uppsi\cdot q^m+\sum_{m\in\Nbb}\upnu(\wtd\yk_m^\perp)_-\uppsi\cdot q^m
\end{align}
holds in $\bbs_\ff(\Wbb)\{q\}[\log q]$, where $\upnu(\wtd\yk_m^\perp)_+\uppsi=\upnu(\wtd\yk_m^\perp)*_1\uppsi$ and $\upnu(\wtd\yk_m^\perp)_-\uppsi=\upnu(\wtd\yk_m^\perp)*_2\uppsi$, cf. \eqref{eq33} for the meaning of $i$-th residue action $*_i$.

Let us compute the RHS of \eqref{eq54} using an argument similar to that in the proof of \cite[Lem. 11.8]{Gui-sewingconvergence}. Due to \eqref{eq74}, we have
\begin{gather*}
\upnu(\wtd\yk_m^\perp)_+\uppsi=\sum_{t\geq -m}\Res_{\xi=0}~ a_{t+m,m}\xi^{t+1}Y_+(\cbf,\xi)\uppsi d\xi
\end{gather*}
where the RHS is a finite sum. Hence, in $\bbs_\ff(\Wbb)\{q\}[\log q]$ we have
\begin{align}\label{eq55}
\begin{aligned}
&\sum_{m\in\Nbb}\upnu(\wtd\yk_m^\perp)_+\uppsi\cdot q^m=\sum_{m\in\Nbb}\sum_{t\geq -m}\Res_{\xi=0}~ a_{t+m,m}\xi^{t+1}q^mY_+(\cbf,\xi)\uppsi d\xi\\
=&\Res_{\xi=0} ~a(\xi,q/\xi)Y_+(\xi^{L_0}\cbf,\xi)\uppsi\frac{d\xi}\xi
\end{aligned}
\end{align}
Similarly, noting $\MU(\upgamma_1)\cbf=\cbf$, we have
\begin{align*}
\sum_{m\in\Nbb}\upnu(\wtd\yk_m^\perp)_-\uppsi\cdot q^m=\Res_{\varpi=0}~b(q/\varpi,\varpi)Y_-(\varpi^{L(0)}\MU(\upgamma_1)\cbf,\varpi)\uppsi\frac{d\varpi}\varpi
\end{align*}
By Prop. \ref{lb34}, we have
\begin{align}\label{eq56}
\sum_{m\in\Nbb}\upnu(\wtd\yk_m^\perp)_-\uppsi\cdot q^m=\Res_{\xi=0} ~b(\xi,q/\xi)Y_+(\xi^{L_0}\cbf,\xi)\uppsi\frac{d\xi}\xi
\end{align}
Combining \eqref{eq54}, \eqref{eq55}, \eqref{eq56} with the fact that $a+b=1$ (cf. \eqref{eq53}), we obtain
\begin{align*}
q\partial_q\uppsi=\Res_{\xi=0} ~Y_+(\xi^{L_0}\cbf,\xi)\uppsi\frac{d\xi}\xi=L_+(0)\uppsi
\end{align*}
This finishes the proof.
\end{proof}

\begin{co}\label{lb31}
Let $L_\pm(0)=L_\pm(0)_\ssp+L_\pm(0)_\nil$ be the Jordan-Chevalley decomposition of $L_\pm(0)$ where $L_\pm(0)_\ssp$ is the semisimple part and $L_\pm(0)_\nil$ is the nilpotent part. Then for each $n\in\Cbb$ and $0\leq l\leq L$ we have
\begin{subequations}\label{eq58}
\begin{gather}
L_+(0)_{\srm}\uppsi_{n,l}=L_-(0)_{\srm}\uppsi_{n,l}=n\uppsi_{n,l}\\
L_+(0)_{\nrm}\uppsi_{n,l}=L_-(0)_{\nrm}\uppsi_{n,l}=(l+1)\uppsi_{n,l+1}
\end{gather}
\end{subequations}
(where $\uppsi_{n,l}=0$ if $l>L$.) In particular, we have $\uppsi_{n,l}\in\bbs_\ff(\Wbb)_{[n,n]}$ for all $n\in\Cbb,0\leq l\leq L$.
\end{co}

\begin{proof}
For each $n\in\Cbb$, let $G_n\subset\bbs_\ff(\Wbb)$ be the generalized eigenspace of $L_+(0)$ with eigenvalue $n$. Then $L_+(0)_\ssp|_{G_n}=n$. By \eqref{eq57}, for each $n,l$ we have
\begin{align}\label{eq59}
(L_\pm(0)-n)\uppsi_{n,l}=(l+1)\uppsi_{n,l+1}
\end{align}
Therefore, $H_n:=\Span\{\uppsi_{n,0},\dots,\uppsi_{n,L}\}$ is invariant under $L_+(0)$, and $(L_+(0)-n)|_{H_n}$ is nilpotent. Thus $H_n\subset G_n$. Therefore, we have $L_+(0)_\ssp|_{H_n}=n$, and hence $L_+(0)_\nil|_{H_n}=(L_+(0)-n)|_{H_n}$. Similarly, we have  $L_-(0)_\ssp|_{H_n}=n$ and $L_-(0)_\nil|_{H_n}=(L_-(0)-n)|_{H_n}$. Combining these results with \eqref{eq59}, we get \eqref{eq58}.
\end{proof}

\subsection{Proof of the SF Theorem \ref{lb23}}\label{lb36}

In this section, we prove Thm. \ref{lb23}, which means proving that the linear map $\Gamma:=\eqref{eq60}$ is bijective.

\subsubsection{Proof that $\Gamma$ is surjective}

Let $p_0=p_1p_2$ with $\arg p_0=\arg p_1+\arg p_2$. Recall the equivalence $\fx_{p_0}\simeq \MS(\ff\sqcup\fn)_{p_1,p_2}$ (cf. \eqref{eq61}). Let $\uppsi_{p_0}$ be an element of $\ST^*_{\MS(\ff\sqcup\fn)_{p_1,p_2}}(\Wbb)$, equivalently, an element of $\ST^*_{\fx_{p_0}}(\Wbb)$. Let $\uppsi\in H^0(\wht\MD_{r\rho}^\times,\ST^*_\fx(\Wbb))$ satisfy \eqref{eq40}. Recall from Lem. \ref{lb25} that each $\uppsi_{n,l}$ is an element of $\bbs_\ff(\Wbb)$. In particular, it is a linear functional $\uppsi_{n,l}:\boxtimes_\ff(\Wbb)\rightarrow\Cbb$. Thus, we can define a linear functional
\begin{gather}\label{eq66}
\upchi:\boxtimes_\ff(\Wbb)\rightarrow\Cbb\qquad \bk{\upchi,\fk w}= \sum_{n\in\Cbb}\bk{\uppsi_{n,0},\fk w}
\end{gather}
Note that the above sum is finite because $L_+(0)_\ssp\uppsi_{n,0}=n\uppsi_{n,0}$ (cf. Cor. \ref{lb31}) and $\fk w$ is a finite sum of eigenvectors of $L_+(0)_\ssp$. We write $\upchi=\sum_{n\in\Cbb}\uppsi_{n,0}$.

By Cor. \ref{lb29} and \eqref{eq49}, for each $\fk w\in\boxtimes_\ff(\Wbb),u\in\Vbb,k\in\Zbb$, we have
\begin{align*}
&\bk{\upchi,Y'_+(u)_{(k)}\fk w}=\sum_{n\in\Cbb}\bk{\uppsi_{n,0},Y'_+(u)_{(k)}\fk w}=\sum_{n\in\Cbb}\bk{Y_+(u)_{(k)}\uppsi_{n,0},\fk w}\\
=&\sum_{n\in\Cbb}\bk{Y'_-(u)_{(k)}\uppsi_{n,0},\fk w}=\sum_{n\in\Cbb}\bk{\uppsi_{n,0},Y_-(u)_{(k)}\fk w}=\bk{\upchi,Y_-(u)_{(k)}\fk w}
\end{align*}
Therefore, by Rem. \ref{lb39}, we have  $\upchi\in\ST^*_\fn(\boxtimes_\ff(\Wbb))$. 

Let us prove $\Gamma(\upchi)=\uppsi_{p_0}$. By Cor. \ref{lb31}, we have
\begin{align*}
&q_1^{L_+(0)}q_2^{L_-(0)}\uppsi_{n,0}=(q_1q_2)^{L_+(0)}\uppsi_{n,0}=(q_1q_2)^{L_+(0)_\ssp+L_+(0)_\nil}\uppsi_{n,0}\\
=&(q_1q_2)^n\cdot \sum_{l\in\Nbb} \frac1{l!}(L_+(0)_\nil\log(q_1q_2))^l\uppsi_{n,0}=\sum_{l=0}^L(q_1q_2)^n(\log(q_1q_2))^l\uppsi_{n,l}
\end{align*}
in $\bbs_\ff(\Wbb)_{[n,n]}\{q_\blt\}[\log q_\blt]$. Combining this result with \eqref{eq64a}, we get
\begin{align*}
&\bk{\MS(\gimel\otimes\upchi),w}\xlongequal{\eqref{eq66}}\sum_{\lambda_1,\lambda_2\in\Cbb}\sum_{\alpha\in\fk A_{\lambda_\blt}}\sum_{n\in\Cbb}\gimel(\wch e_{\lambda_\blt}(\alpha)\otimes w)\cdot\bigbk{q_1^{L_+(0)}q_2^{L_-(0)}\uppsi_{n,0},e_{\lambda_\blt}(\alpha)}\\
=&\sum_{n\in\Cbb}\sum_{l=0}^L \sum_{\alpha\in\fk A_{n,n}} (q_1q_2)^n (\log(q_1q_2))^l  \cdot \gimel(\wch e_{n,n}(\alpha)\otimes w)\cdot\bk{\uppsi_{n,l},e_{n,n}(\alpha)}\\
=&\sum_{n\in\Cbb}\sum_{l=0}^L (q_1q_2)^n (\log(q_1q_2))^l  \cdot\gimel(\uppsi_{n,l}\otimes w)
\end{align*}
where the last equation is due to the easy fact that  $\uptheta=\sum_{\alpha\in\fk A_{n,n}}\bk{\uptheta,e_{n,n}(\alpha)}\cdot \wch e_{n,n}(\alpha)$ for any $\uptheta\in\bbs_\ff(\Wbb)_{n,n}$. Recall again that $\gimel:\bbs_\ff(\Wbb)\otimes\Wbb\rightarrow\Cbb$ is the restriction of the evaluation pairing $\Wbb^*\otimes\Wbb\rightarrow\Cbb$. Therefore
\begin{align*}
\bk{\MS(\gimel\otimes\upchi),w}=\sum_{n\in\Cbb}\sum_{l=0}^L  (q_1q_2)^n (\log(q_1q_2))^l  \bk{\uppsi_{n,l},w}=\bk{\uppsi,w}\big|_{q=q_1q_2}
\end{align*}
We conclude that $\MS(\gimel\otimes\upchi)=\uppsi|_{q=q_1q_2}$, and hence $\MS(\gimel\otimes\upchi)|_{p_1,p_2}=\uppsi|_{p_0}=\uppsi_{p_0}$. This proves that $\Gamma(\upchi)=\uppsi_{p_0}$.

\subsubsection{Proof that $\Gamma$ is injective}

For each $\upchi\in\ST^*_\fn(\boxtimes_\ff(\Wbb))$, recall that $\MS(\gimel\otimes\upchi)$ can be viewed as a section of the vector bundle $\ST^*_{\MS(\ff\sqcup\fn)}(\Wbb)$ on $\MD^\times_r\times\MD^\times_\rho$, cf. \eqref{eq65}. By \cite[Thm. 4.11]{GZ2}, there exists a connection on $\ST^*_{\MS(\ff\sqcup\fn)}(\Wbb)|_{\MD^\times_r\times\MD^\times_\rho}$ such that for any $\upchi\in\ST^*_\fn(\boxtimes_\ff(\Wbb))$, the section $\MS(\gimel\otimes\upchi)$ is parallel under this connection. 

Recall that parallel sections are determined by their values at a given point. Therefore, if we assume that $\Gamma(\chi)=0$ (which means $\MS(\gimel\otimes\upchi)|_{p_1,p_2}=0$), then  $\MS(\gimel\otimes\upchi)$ is the zero section. Applying Prop. \ref{geometry11} first to the variable $q_1$ (for each fixed $q_2\in\MD_\rho^\times$ and $\arg q_2$) and then to the variable $q_2$, we conclude that for each $w\in\Wbb$ and $\lambda_1,\lambda_2\in\Cbb$ and $l_1,l_2\in\Nbb$, the coefficient of $q_1^{\lambda_1}q_2^{\lambda_2}(\log q_1)^{l_1}(\log q_2)^{l_2}$ in \eqref{eq64a} is zero. When $l_1=l_2=0$, this means that
\begin{align*}
\sum_{\alpha\in\fk A_{\lambda_\blt}}\gimel(\wch e_{\lambda_\blt}(\alpha)\otimes w)\cdot\bk{\upchi,e_{\lambda_\blt}(\alpha)}=0
\end{align*}
Since $\upchi$ is a linear functional $\boxtimes_\ff(\Wbb)\rightarrow\Cbb$, we can define its restriction $\upchi_{\lambda_\blt}:\boxtimes_\ff(\Wbb)_{[\lambda_\blt]}\rightarrow\Cbb$. Then $\upchi_{\lambda_\blt}$ is an element of $\bbs_\ff(\Wbb)_{[\lambda_\blt]}$, and we clearly have
\begin{align*}
\upchi_{\lambda_\blt}=\sum_{\alpha\in\fk A_{\lambda_\blt}}\bk{\upchi,e_{\lambda_\blt}(\alpha)}\cdot \wch e_{\lambda_\blt}(\alpha)
\end{align*}
It follows that $\gimel(\upchi_{\lambda_\blt}\otimes w)=0$ for all $w\in\Wbb$. Since $\gimel$ is partially injective (cf. Rem. \ref{lb15}), we must have $\upchi_{\lambda_\blt}=0$ for all $\lambda_\blt$. Hence $\upchi=0$. This finishes the proof that $\Gamma$ is injective.

\section{The general sewing-factorization theorems}

\subsection{The SF theorems for compact Riemann surfaces}

The goal of this section is to prove Thm. \ref{lb40} and \ref{lb43}, two nearly equivalent versions of the sewing-factorization theorem.

\subsubsection{The setting}\label{lb44}

Let $N,K,R\in\Nbb$. In this section, we let 
\begin{gather*}
\fk F=\big(x_1',\dots,x_R'; \theta_1',\dots,\theta_R'\big|C_1\big|x_1,\dots,x_N;\theta_1,\dots,\theta_N\big)\\
\fk G=\big(y_1',\dots,y_R';\mu_1',\dots,\mu_R'\big|C_2\big| y_1,\dots,y_K;\mu_1,\dots,\mu_K \big)
    \end{gather*}
be respectively $(R,N)$-pointed and $(R,K)$-pointed compact Riemann surfaces with local coordinates, cf. Def. \ref{lb12}. Recall that $\ff\sqcup\fg$ is the disjoint union of $\ff,\fg$, cf. \eqref{eq86}. For each $1\leq j\leq R$, let $V'_j$ (resp. $V''_j$) be a neighborhood on which $\theta_j'$ (resp. $\mu_j'$) is defined. Assume that $V'_1,\dots,V_R',x_1,\dots,x_N$ are mutually disjoint, and $V''_1,\dots,V_R'',y_1,\dots,y_K$ are mutually disjoint. We also assume that for each $j$ there exist $r_j,\rho_j\in(0,+\infty]$ such that
\begin{align}
\theta_j'(V_j')=\MD_{r_j}\qquad \mu_j'(V_j'')=\MD_{\rho_j}
\end{align} 

\begin{df}
Let $\MS(\ff\sqcup\fg)$ be the sewing of $\ff\sqcup\fg$ along the pairs of points $(x_j',y_j')$ (for all $1\leq j\leq R$) with sewing radii $(r_j,\rho_j)$ using the local coordinates $(\theta_j',\mu_j')$. Cf. Subsec. \ref{lb9}.  Then $\MS(\ff\sqcup\fg)$ has base manifold $\MD_{r_\blt\rho_\blt}=\MD_{r_1\rho_1}\times\cdots\times\MD_{r_R\rho_R}$, and the fibers over $\MD_{r_\blt\rho_\blt}^\times=\MD_{r_1\rho_1}^\times\times\cdots\times\MD_{r_R\rho_R}^\times$ are smooth. Recall that the marked points $x_\blt,y_\blt$ and their local coordinates $\theta_\blt,\mu_\blt$ are extended constantly to $\MS(\ff\sqcup\fg)$, and we continue to denote them by the same symbol. So we can write
\begin{align}\label{eq105}
\MS(\ff\sqcup\fg)=\big(\pi:\MC\rightarrow\MD_{r_\blt\rho_\blt}\big|x_1,\dots,x_N,y_1,\dots,y_K;\theta_1,\dots,\theta_N,\mu_1,\dots,\mu_K  \big)
\end{align}
We still let $q_j$ be the standard coordinate of $\MD_{r_j\rho_j}$. Then $q_\blt=(q_1,\dots,q_R)$ is a set of coordinates of $\MD_{r_\blt\rho_\blt}$.
\end{df}

We assume that each component of $C_1$ contains one of $x_1,\dots,x_N$ (so that (dual) fusion products along $\ff$ can be defined), and each component of $C_2$ contains one of $y_1,\dots,y_K,y_1',\dots,y_R'$ (so that conformal blocks associated to $\fg$ can be defined).

Associate $\Wbb\in\Mod(\Vbb^{\otimes N})$ to the ordered incoming marked points $x_1,\dots,x_N$ of $\ff$. Similarly, associate $\Xbb\in\Mod(\Vbb^{\otimes K})$ to $y_1,\dots,y_K$. Associate $\Mbb\in\Mod(\Vbb^{\otimes R})$ to $x_1',\dots,x_R'$, and associate its contragredient $\Mbb'$ to $y_1',\dots,y_R'$. 

We view $\Mbb\otimes\Wbb$ as associated to $x_\blt',x_\blt$, and $\Mbb'\otimes\Xbb$ as associated to $y_\blt',y_\blt$. Associate $\Wbb\otimes\Xbb$ to the ordered marked points $x_1,\dots,x_N,y_1,\dots,y_K$ of $\MS(\ff\sqcup\fg)$. Let
\begin{align*}
\upphi\in\ST_\ff^*(\Mbb\otimes\Wbb)\qquad\upchi\in\ST^*_\fg(\Mbb'\otimes\Xbb)
\end{align*}

\begin{df}\label{lb42}
The \textbf{sewing of $\upphi$ and $\upchi$} is defined by contracting the $\Mbb$-component of $\upphi$ with the $\Mbb'$-component of $\upchi$. More precisely,
\begin{gather}
\begin{gathered}
\MS(\upphi\otimes\upchi):\Wbb\otimes\Xbb\rightarrow\Cbb\{q_\blt\}[\log q_\blt]\\[0.5ex]
\bigbk{\MS(\upphi\otimes\upchi),w\otimes \tipaz}=\sum_{\lambda_\blt\in\Cbb^R}\sum_{\alpha\in\fk A_{\lambda_\blt}}\bigbk{\upphi,q_\blt^{L_\blt(0)}e_{\lambda_\blt}(\alpha)\otimes w}\cdot \bigbk{\upchi,\wch e_{\lambda_\blt}(\alpha)\otimes\tipaz}
\end{gathered}
\end{gather}
(where $w\in\Wbb,\tipaz\in\Xbb$). Here, $(e_{\lambda_\blt}(\alpha))_{\alpha\in\fk A_{\lambda_\blt}}$ is a basis of $\Mbb_{[\lambda_\blt]}$ with dual basis $(\wch e_{\lambda_\blt}(\alpha))_{\alpha\in\fk A_{\lambda_\blt}}$ in $\Mbb'_{[\lambda_\blt]}$, and
\begin{align}\label{eq102}
q_\blt^{L_\blt(0)}=q_1^{L_1(0)}\cdots q_R^{L_R(0)}
\end{align}
See \cite[Sec. 4.1]{GZ2} for more explanations. Clearly we also have
\begin{gather}
\bigbk{\MS(\upphi\otimes\upchi),w\otimes \tipaz}=\sum_{\lambda_\blt\in\Cbb^R}\sum_{\alpha\in\fk A_{\lambda_\blt}}\bigbk{\upphi,e_{\lambda_\blt}(\alpha)\otimes w}\cdot \bigbk{\upchi,q_\blt^{L_\blt(0)}\wch e_{\lambda_\blt}(\alpha)\otimes \tipaz}
\end{gather}
\end{df}

\begin{rem}
By \cite[Thm. 4.9]{GZ2}, the formal series $\MS(\upphi\otimes\upchi)$ converges a.l.u. on $\MD_{r_\blt\rho_\blt}^\times$ in the sense of Def. \ref{lb19}, and
\begin{align}\label{eq112}
\MS(\upphi\otimes\upchi)\big|_{\MD_{r_\blt\rho_\blt}^\times}\in H^0\big(\wht\MD_{r_\blt\rho_\blt}^\times,\ST^*_{\MS(\ff\sqcup\fg)}(\Wbb\otimes\Xbb)\big)
\end{align}
where $\wht\MD_{r_\blt\rho_\blt}^\times$ is the universal cover of $\MD_{r_\blt\rho_\blt}^\times$. By \cite[Rem. 3.1]{GZ2}, \eqref{eq112} is equivalent to that  for each $p_\blt=(p_1,\dots,p_R)\in\MD_{r_\blt\rho_\blt}^\times$ with fixed $\arg p_1,\dots,\arg p_R$, we have
\begin{align*}
\MS(\upphi\otimes\upchi)\big|_{p_\blt}\in \ST^*_{\MS(\ff\sqcup\fg)_{p_\blt}}(\Wbb\otimes\Xbb)
\end{align*}
\end{rem}

\subsubsection{The sewing-factorization theorems}

Let $(\bbs_\ff(\Wbb),\gimel)$ be a dual fusion product of $\Wbb$ along $\ff$. If each component of $C_2$ contains one of $y_\blt$, we have a dual fusion product $(\bbs_\fg(\Xbb),\daleth)$ of $\Xbb$ along $\fg$. So
\begin{align*}
\gimel:\bbs_\ff(\Wbb)\otimes\Wbb\rightarrow\Cbb\qquad\daleth:\bbs_\fg(\Xbb)\otimes\Xbb\rightarrow\Cbb
\end{align*}
are the canonical conformal blocks.

\begin{thm}\label{lb40}
Assume that each component of $C_1$ contains one of $x_1,\dots,x_N$, and each component of $C_2$ contains one of $y_1,\dots,y_K$.  Let $p_\blt\in\MD_{r_\blt\rho_\blt}^\times$ with fixed $\arg p_1,\dots,\arg p_R$. Then the following linear map is an isomorphism.
\begin{gather}\label{eq95}
\begin{gathered}
\Hom_{\Vbb^{\otimes R}}\big(\boxtimes_\fg(\Xbb),\bbs_\ff(\Wbb)\big)\xlongrightarrow{\simeq} \ST_{\MS(\fk F\sqcup\fk G)_{p_\blt}}^*(\Wbb\otimes \Xbb)\\[0.5ex]
T\mapsto \MS\big((\gimel\circ(T\otimes \id_\Wbb))\otimes \daleth\big)\big|_{p_\blt}
\end{gathered}
\end{gather}
\end{thm}

Following Def. \ref{lb42}, the sewing in \eqref{eq95} is defined by contracting the $\boxtimes_\fg(\Xbb)$-component of $\gimel\circ(T\otimes\id_\Wbb):\boxtimes_\fg(\Xbb)\otimes\Wbb\rightarrow\Cbb$ with the $\bbs_\fg(\Xbb)$-component of $\daleth$. Moreover, note that by taking transpose, \eqref{eq95} becomes an isomorphism
\begin{gather}\label{eq100}
\begin{gathered}
\Hom_{\Vbb^{\otimes R}}\big(\boxtimes_\ff(\Wbb),\bbs_\fg(\Xbb)\big)\xlongrightarrow{\simeq} \ST_{\MS(\fk F\sqcup\fk G)_{p_\blt}}^*(\Wbb\otimes \Xbb)\\[0.5ex]
T\mapsto \MS\big(\gimel\otimes (\daleth\circ(T\otimes \id_\Xbb))\big)\big|_{p_\blt}
\end{gathered}
\end{gather}
where the sewing is defined by contracting the $\bbs_\ff(\Wbb)$-component of $\gimel$ with the $\boxtimes_\ff(\Wbb)$-component of $\daleth\circ(T\otimes\id_\Xbb):\boxtimes_\ff(\Wbb)\otimes\Xbb\rightarrow\Cbb$.

\begin{proof}
Step 1. Recall \eqref{eq87} for the $(2,0)$-pointed surface $\fn=\big(\infty,0;1/\zeta,\zeta\big|\Pbb^1\big)$. Let $\upomega\in\ST^*_\fn(\bbs_\fq(\Vbb))$ be defined by \eqref{eq88}, where $\fq=\big(\infty,0;1/\zeta,\zeta\big|\Pbb^1\big|1;\zeta-1\big)$. The goal of this step is to construct  the linear isomorphism \eqref{eq91} by applying the sewing-factorization Cor. \ref{lb32} iteratively $R$ times.

For each $1\leq \ell\leq R$, choose $p^+_\ell\in\MD^\times_{r_\ell}$ and $p^-_\ell\in\MD^\times_{\rho_\ell}$ such that $p_\ell=p^+_\ell p^-_\ell$. Let $\fz^\ell$ be the family obtained by sewing $\ff\sqcup\fg$ along $(x_1',y_1'),\dots,(x_\ell',y_\ell')$ with sewing radii $(r_1,\rho_1),\dots,(r_\ell,\rho_\ell)$. In particular, we have $\fz^0=\ff\sqcup\fg$ and $\fz^R=\MS(\ff\sqcup\fg)$. The base manifold of $\fz^\ell$ is $\MD_{r_1\rho_1}\times\cdots\times\MD_{r_\ell\rho_\ell}$. Note that the fiber $\fz^\ell_{p_1,\dots,p_\ell}$ has marked points
\begin{align*}
x_{\ell+1}',y_{\ell+1}',\dots,x_R',y_R',x_1,\dots,x_N,y_1,\dots,y_K
\end{align*}
Assuming this order, we associate $(\boxtimes_\fq(\Vbb))^{\otimes(R-\ell)}\otimes\Wbb\otimes\Xbb$ to these marked points.

For each $0\leq\ell<R$, consider the sewing
\begin{align*}
\MS(\fz^\ell_{p_1,\dots,p_\ell}\sqcup\fn)
\end{align*}
of $\fz^\ell_{p_1,\dots,p_\ell}\sqcup\fn$ along the pairs of points $(x_{\ell+1}',\infty)$ and $(y'_{\ell+1},0)$ with sewing radii $r_{\ell+1},1$ and $\rho_{\ell+1},1$. This family has base manifold $\MD_{r_{\ell+1}\rho_{\ell+1}}$.  Cor. \ref{lb32} yields an isomorphism
\begin{gather*}
\ST^*_{\fz^\ell_{p_1,\dots,p_\ell}}\big(\boxtimes_\fq(\Vbb)^{\otimes(R-\ell)}\otimes\Wbb\otimes\Xbb\big)\rightarrow \ST^*_{\MS(\fz^\ell_{p_1,\dots,p_\ell}\sqcup\fn)_{p^+_{\ell+1},p^-_{\ell+1}}}\big(\boxtimes_\fq(\Vbb)^{\otimes(R-\ell-1)}\otimes\Wbb\otimes\Xbb\big)\\
\uppsi\mapsto \MS(\uppsi\otimes\upomega)\big|_{p^+_{\ell+1},p^-_{\ell+1}}
\end{gather*}
where the sewing is defined by contracting the first tensor component $\boxtimes_\fq(\Vbb)$ of $\uppsi$ with $\upomega:\bbs_\fq(\Vbb)\rightarrow\Cbb$. By \eqref{eq61}, we have a canonical isomorphism $\fz^{\ell+1}_{p_1,\dots,p_{\ell+1}}\simeq\MS(\fz^\ell_{p_1,\dots,p_\ell}\sqcup\fn)_{p^+_{\ell+1},p^-_{\ell+1}}$. Thus, the above linear isomorphism becomes
\begin{gather*}
\ST^*_{\fz^\ell_{p_1,\dots,p_\ell}}\big(\boxtimes_\fq(\Vbb)^{\otimes(R-\ell)}\otimes\Wbb\otimes\Xbb\big)\longrightarrow \ST^*_{\fz^{\ell+1}_{p_1,\dots,p_{\ell+1}}}\big(\boxtimes_\fq(\Vbb)^{\otimes(R-\ell-1)}\otimes\Wbb\otimes\Xbb\big)\\
\uppsi\mapsto \MS(\uppsi\otimes\upomega)\big|_{p^+_{\ell+1},p^-_{\ell+1}}
\end{gather*}

Therefore, by induction on $\ell$, we have a linear isomorphism
\begin{gather}\label{eq91}
\begin{gathered}
\ST^*_{\ff\sqcup\fg}\big(\boxtimes_\fq(\Vbb)^{\otimes R}\otimes\Wbb\otimes\Xbb\big)\longrightarrow \ST^*_{\MS(\ff\sqcup\fg)_{p_\blt}}\big(\Wbb\otimes\Xbb\big)\\
\uppsi\mapsto \MS(\uppsi\otimes\upomega^{\otimes R})\big|_{p^+_1,p^-_1,\dots,p^+_R,p^-_R}\equiv \MS(\uppsi\otimes\upomega^{\otimes R})\big|_{p^\pm_\blt}
\end{gathered}
\end{gather}
where $\MS(\uppsi\otimes\upomega^{\otimes R})$ is defined by contracting the component $\boxtimes_\fq(\Vbb)^{\otimes R}$ of $\uppsi$ with $\upomega^{\otimes R}:\bbs_\fq(\Vbb)^{\otimes R}\rightarrow\Cbb$.\\[-1ex]

Step 2. The goal of this step is to obtain the isomorphism \eqref{eq92}. By Thm. \ref{lb16} and the universal property for dual fusion products, we have an isomorphism
\begin{gather}\label{eq98}
\begin{gathered}
\Hom_{\Vbb^{\otimes(2R)}}\big(\boxtimes_\fq(\Vbb)^{\otimes R},\bbs_\ff(\Wbb)\otimes\bbs_\fg(\Xbb) \big)\longrightarrow \ST^*_{\ff\sqcup\fg}\big(\boxtimes_\fq(\Vbb)^{\otimes R}\otimes\Wbb\otimes\Xbb\big)\\
A\mapsto (\gimel\otimes\daleth)\circ (A\otimes\id_\Wbb\otimes\id_\Xbb)
\end{gathered}
\end{gather}
Here, by reordering the tensor components in the domain, we view $\gimel\otimes\daleth$ as a linear functional
\begin{align*}
\gimel\otimes\daleth:\bbs_\ff(\Wbb)\otimes\bbs_\fg(\Xbb)\otimes\Wbb\otimes\Xbb\rightarrow\Cbb
\end{align*}
Moreover, $A$ is understood to be a linear map $\boxtimes_\fq(\Vbb)^{\otimes R}\rightarrow\bbs_\ff(\Wbb)\otimes\bbs_\fg(\Xbb)$ satisfying
\begin{gather*}
A\big(\fk v_1\otimes\cdots\otimes Y_+(u,z)\fk v_j\otimes\cdots\otimes\fk v_R\big)=\big(Y_j(u,z)\otimes\id_{\bbs_\fg(\Xbb)}\big)A\big(\fk v_1\otimes\cdots\otimes\fk v_R\big)\\
A\big(\fk v_1\otimes\cdots\otimes Y_-(u,z)\fk v_j\otimes\cdots\otimes\fk v_R\big)=\big(\id_{\bbs_\ff(\Wbb)}\otimes Y_j(u,z)\big)A\big(\fk v_1\otimes\cdots\otimes\fk v_R\big)
\end{gather*}
for all $\fk v_1,\dots,\fk v_R\in\boxtimes_\fq(\Vbb)$ and $u\in\Vbb$.

Combining \eqref{eq98} with \eqref{eq91}, we obtain a linear isomorphism
\begin{gather*}
\Hom_{\Vbb^{\otimes(2R)}}\big(\boxtimes_\fq(\Vbb)^{\otimes R},\bbs_\ff(\Wbb)\otimes\bbs_\fg(\Xbb) \big)\longrightarrow\ST^*_{\MS(\ff\sqcup\fg)_{p_\blt}}\big(\Wbb\otimes\Xbb\big)\\
A\mapsto \MS\big(((\gimel\otimes\daleth)\circ (A\otimes\id_\Wbb\otimes\id_\Xbb))\otimes\upomega^{\otimes R}\big)\big|_{p^\pm_\blt}=\MS\big((\gimel\otimes\daleth)\otimes(\upomega^{\otimes R}\circ A^\tr)\big)\big|_{p^\pm_\blt}
\end{gather*}
and hence a linear isomorphism
\begin{gather}\label{eq92}
\begin{gathered}
\Hom_{\Vbb^{\otimes(2R)}}\big(\boxtimes_\ff(\Wbb)\otimes\boxtimes_\fg(\Xbb),\bbs_\fq(\Vbb)^{\otimes R}\big)\longrightarrow\ST^*_{\MS(\ff\sqcup\fg)_{p_\blt}}\big(\Wbb\otimes\Xbb\big)\\
B\mapsto \MS\big((\gimel\otimes\daleth)\otimes(\upomega^{\otimes R}\circ B)\big)\big|_{p^\pm_\blt}
\end{gathered}
\end{gather}
where the sewing is defined by contracting the component $\bbs_\ff(\Wbb)\otimes\bbs_\fg(\Xbb)$ of  $\gimel\otimes\daleth$ with $\upomega^{\otimes R}\circ B:\boxtimes_\ff(\Wbb)\otimes\boxtimes_\fg(\Xbb)\rightarrow\Cbb$.\\[-1ex]

Step 3. Let $\fn^R=\fn_1\sqcup\cdots\sqcup\fn_R$ be the disjoint union of $R$ pieces of $\fn$, where $\fn_j\simeq\fn$ is written as
\begin{align*}
\fn_j=\big(\infty_j,0_j;1/\zeta,\zeta|\Pbb^1_j\big)
\end{align*}
Note that $\upomega^{\otimes R}:\bbs_\fq(\Vbb)^{\otimes R}\rightarrow\Cbb$ is a conformal block associated to $\fn^R$. (Here, we associate $\bbs_\fq(\Vbb)^{\otimes R}$ to the ordered marked points $\infty_1',0_1',\dots,\infty_R',0_R'$.) Then we have a linear map
\begin{gather}\label{eq90}
\begin{gathered}
\Hom_{\Vbb^{\otimes(2R)}}\big(\boxtimes_\ff(\Wbb)\otimes\boxtimes_\fg(\Xbb),\bbs_\fq(\Vbb)^{\otimes R}\big)\longrightarrow \ST^*_{\fn^R}\big(\boxtimes_\ff(\Wbb)\otimes\boxtimes_\fg(\Xbb)\big)\\
B\mapsto \upomega^{\otimes R}\circ B
\end{gathered}
\end{gather}
where $\boxtimes_\ff(\Wbb)\otimes\boxtimes_\fg(\Xbb)$ is associated to the ordered marked points $\infty_1',\dots,\infty_R',0_1',\dots,0_R'$ of $\fn^R$.

We claim that \eqref{eq90} is an isomorphism. Suppose this is true. Then, combining \eqref{eq90} with \eqref{eq92}, we get a linear isomorphism
\begin{gather}\label{eq93}
\begin{gathered}
\ST^*_{\fn^R}(\boxtimes_\ff(\Wbb)\otimes\boxtimes_\fg(\Xbb))\longrightarrow\ST^*_{\MS(\ff\sqcup\fg)_{p_\blt}}\big(\Wbb\otimes\Xbb\big)\\
\uptau\mapsto \MS\big(\gimel\otimes\daleth\otimes\uptau\big)\big|_{p^\pm_\blt}
\end{gathered}
\end{gather}
By Rem. \ref{lb39}, an element of $\ST^*_{\fn^R}(\boxtimes_\ff(\Wbb)\otimes\boxtimes_\fg(\Xbb))$ is precisely a linear functional $\uptau:\boxtimes_\ff(\Wbb)\otimes\boxtimes_\fg(\Xbb)\rightarrow\Cbb$ such that $\uptau(Y_j'(u,z)\fk w\otimes\fk z)=\uptau(\fk w\otimes Y_j(u,z)\fk z)$ holds in $\Cbb[[z^{\pm1}]]$ for all $1\leq j\leq R$, equivalently, a linear map $\uptau^\sharp:\boxtimes_\fg(\Xbb)\rightarrow \ovl{\bbs_\ff(\Wbb)}$ satisfying $\uptau^\sharp\circ Y_j(u)_n=Y_j(u)_n\circ\uptau^\sharp$ for all $u\in\Vbb,j$ and $n\in\Zbb$. In particular, $[L_j(0),\uptau^\sharp]=0$. 

Choose any $\lambda_\blt,\lambda'_\blt\in\Cbb^N$. Since $L_j(0)-\lambda_j$ is nilpotent on $\boxtimes_\fg(\Xbb)_{[\lambda_\blt]}$, there exists $a\in\Zbb_+$ such that $(L_j(0)-\lambda_j)^a\boxtimes_\fg(\Xbb)_{[\lambda_\blt]}=0$. Since $L_j(0)$ preserves each generalized eigenspace of $L(0)$, we have $[L_j(0),P_{\lambda_\blt}]=[L_j(0),P_{\lambda'_\blt}]=0$. Therefore $(L_j(0)-\lambda_j)^a P_{\lambda'_\blt}\uptau^\sharp P_{\lambda_\blt}=P_{\lambda'_\blt}\uptau^\sharp P_{\lambda_\blt}(L_j(0)-\lambda_j)^a=0$. Thus $P_{\lambda'_\blt}\uptau^\sharp P_{\lambda_\blt}$ has range in $\bbs_\ff(\Wbb)_{[\lambda_\blt]}$. This proves that $\uptau^\sharp$ sends $\boxtimes_\fg(\Xbb)_{[\lambda_\blt]}$ into $\bbs_\ff(\Wbb)_{[\lambda_\blt]}$. In particular, $\tau^\sharp$ has range in $\bbs_\ff(\Wbb)$.

Therefore, we have a linear isomorphism
\begin{gather}\label{eq94}
\begin{gathered}
\Hom_{\Vbb^{\otimes R}}(\boxtimes_\fg(\Xbb),\bbs_\ff(\Wbb))\longrightarrow\ST^*_{\fn^R}(\boxtimes_\ff(\Wbb)\otimes\boxtimes_\fg(\Xbb))\\
T\mapsto \Big(\fk w\otimes\fk z\mapsto\bk{\fk w,T\fk z}  \Big)
\end{gathered}
\end{gather}
The composition of \eqref{eq93} with \eqref{eq94} yields the isomorphism \eqref{eq95}, thereby completing the proof.\\[-1ex]

Step 4. It remains to prove that \eqref{eq90} is an isomorphism. Let $\fq^R=\fq_1\sqcup\cdots\sqcup\fq_R$ be a disjoint union of $\fq$, where $\fq_j\simeq\fq$ is 
\begin{align*}
\fq_j=\big(\infty_j,0_j;1/\zeta,\zeta\big|\Pbb^1_j\big|1_j,\zeta-1\big)
\end{align*}
Recall that $(\bbs_\fq(\Vbb),\aleph)$ is a dual fusion product of $\Vbb$ along $\fq$. By Thm. \ref{lb16}, $(\bbs_\fq(\Vbb)^{\otimes R},\aleph^{\otimes R})$ is a dual fusion product of $\Vbb^{\otimes R}$ along $\fq^R$. In particular,
\begin{align*}
\aleph^{\otimes R}:\bbs_\fq(\Vbb)^{\otimes R}\otimes\Vbb^{\otimes R}\rightarrow\Cbb
\end{align*}
is the canonical conformal block, where $\bbs_\fq(\Vbb)^{\otimes R}\otimes\Vbb^{\otimes R}$ is associated to the marked points of $\fq^R$ in the order $\infty_1,0_1,\dots,\infty_R,0_R,1_1,\dots,1_R$. Therefore, by the universal property,
\begin{gather}\label{eq96}
\begin{gathered}
\Hom_{\Vbb^{\otimes(2R)}}\big(\boxtimes_\ff(\Wbb)\otimes\boxtimes_\fg(\Xbb),\bbs_\fq(\Vbb)^{\otimes R}\big)\longrightarrow \ST^*_{\fq^R}\big(\boxtimes_\ff(\Wbb)\otimes\boxtimes_\fg(\Xbb)\otimes\Vbb^{\otimes R}\big)\\
B\mapsto \aleph^{\otimes R}\circ (B\otimes\id_{\Vbb^{\otimes R}})
\end{gathered}
\end{gather}
is a linear isomorphism, where $\boxtimes_\ff(\Wbb)\otimes\boxtimes_\fg(\Xbb)\otimes\Vbb^{\otimes R}$ is associated to the marked points of $\fq^R$ in the order $\infty_1,\dots,\infty_R,0_1,\dots,0_R,1_1,\dots,1_R$.

The propagation of conformal blocks (cf. \cite[Cor. 2.44]{GZ1}) gives an isomorphism
\begin{gather}\label{eq97}
\begin{gathered}
\ST^*_{\fq^R}\big(\boxtimes_\ff(\Wbb)\otimes\boxtimes_\fg(\Xbb)\otimes\Vbb^{\otimes R}\big)\longrightarrow \ST^*_{\fn^R}\big(\boxtimes_\ff(\Wbb)\otimes\boxtimes_\fg(\Xbb)\big)\\
\tipxphi\mapsto \tipxphi(-\otimes\idt^{\otimes R})
\end{gathered}
\end{gather}
where we note that $\idt^{\otimes R}\in\Vbb^{\otimes R}$. Since $\upomega^{\otimes R}$ equals $\aleph^{\otimes R}(-\otimes\idt^{\otimes R})$ (cf. \eqref{eq88}), the composition of \eqref{eq97} with \eqref{eq96} equals \eqref{eq90}. Therefore \eqref{eq90} is an isomorphism.
\end{proof}

Note that the assumption on the marked points of $C_2$ in the following sewing-factorization theorem is weaker than that in Thm. \ref{lb40}. 

\begin{thm}\label{lb43}
Assume that each component of $C_1$ contains one of $x_1,\dots,x_N$, and each component of $C_2$ contains one of $y_1,\dots,y_K,y_1',\dots,y_R'$.   Let $p_\blt\in\MD_{r_\blt\rho_\blt}^\times$ with fixed $\arg p_1,\dots,\arg p_R$. Then the following linear map is an isomorphism.
\begin{gather}\label{eq99}
\begin{gathered}
\ST_{\fg}^*\big(\boxtimes_\ff(\Wbb)\otimes \Xbb\big)\xlongrightarrow{\simeq} \ST_{\MS(\fk F\sqcup\fk G)_{p_\blt}}^*\big(\Wbb\otimes \Xbb\big)\\[0.5ex]
\upchi \mapsto \MS(\gimel\otimes \upchi)\big|_{p_\blt}
\end{gathered}
\end{gather}
\end{thm}

The sewing in \eqref{eq99} is defined by contracting the $\bbs_\ff(\Wbb)$-component of $\gimel:\bbs_\ff(\Wbb)\otimes\Wbb\rightarrow\Cbb$ with the $\boxtimes_\ff(\Wbb)$-component of $\upchi:\boxtimes_\ff(\Wbb)\otimes\Xbb\rightarrow\Cbb$.

\begin{proof}
Since we are not assuming that each component of $C_2$ contains one of the incoming marked points $y_\blt$, we may not be able to define the dual fusion product $\bbs_\fg(\Xbb)$. However, we can remedy this by adding additional incoming marked points $y_{K+1},\dots,y_{K+L}$ (with local coordinates $\mu_{K+1},\dots,\mu_{K+L}$) to $\fg$. More precisely, let
\begin{align*}
\wht\fg=\big(y_1',\dots,y_R';\mu_1',\dots,\mu_R'\big|C_2\big| y_1,\dots,y_{K+L};\mu_1,\dots,\mu_{K+L} \big)
\end{align*}
where $y_1,\dots,y_{K+L}\in C_2$ are mutually disjoint, and each component of $C_2$ contains one of $y_1,\dots,y_{K+L}$. Associate $\Xbb\otimes\Vbb^{\otimes L}$ to the ordered marked points $y_1,\dots,y_{K+L}$. Then we have a dual fusion product $(\bbs_{\wht\fg}(\Xbb\otimes\Vbb^{\otimes L}),\tsadi)$ where $\bbs_{\wht\fg}(\Xbb\otimes\Vbb^{\otimes L})\in\Mod(\Vbb^{\otimes R})$ is associated to the ordered marked points $y_1',\dots,y_R'$, and
\begin{align*}
\tsadi:\bbs_{\wht\fg}(\Xbb\otimes\Vbb^{\otimes L})\otimes\Xbb\otimes\Vbb^{\otimes L}\rightarrow\Cbb
\end{align*}
is the canonical conformal block.

By Thm. \ref{lb40} (more precisely, by Eq. \eqref{eq100}), we have an isomorphism
\begin{gather*}
\begin{gathered}
\Hom_{\Vbb^{\otimes R}}\big(\boxtimes_\ff(\Wbb),\bbs_{\wht\fg}(\Xbb\otimes\Vbb^{\otimes L})\big)\longrightarrow \ST_{\MS(\ff\sqcup\wht\fg)_{p_\blt}}^*(\Wbb\otimes \Xbb\otimes\Vbb^{\otimes L})\\[0.5ex]
T\mapsto \MS\big(\gimel\otimes (\tsadi\circ(T\otimes \id_\Xbb\otimes\id_{\Vbb^{\otimes L}}))\big)\big|_{p_\blt}
\end{gathered}
\end{gather*}
where the sewing is defined by contracting the $\bbs_\ff(\Wbb)$-component of $\gimel$ with the $\boxtimes_\ff(\Wbb)$-component of $\tsadi\circ(T\otimes \id_\Xbb\otimes\id_{\Vbb^{\otimes L}})$. The universal property of $(\bbs_{\wht\fg}(\Xbb\otimes\Vbb^{\otimes L}),\tsadi)$ yields the isomorphism
\begin{gather*}
\Hom_{\Vbb^{\otimes R}}\big(\boxtimes_\ff(\Wbb),\bbs_{\wht\fg}(\Xbb\otimes\Vbb^{\otimes L})\big)\longrightarrow\ST^*_{\wht\fg}\big(\boxtimes_\ff(\Wbb)\otimes\Xbb\otimes\Vbb^{\otimes L}\big)\\
T\mapsto \tsadi\circ(T\otimes\id_\Xbb\otimes\id_{\Vbb^{\otimes L}})
\end{gather*}
Therefore, we obtain the isomorphism
\begin{gather}\label{eq101}
\begin{gathered}
\ST^*_{\wht\fg}\big(\boxtimes_\ff(\Wbb)\otimes\Xbb\otimes\Vbb^{\otimes L}\big)\longrightarrow\ST_{\MS(\ff\sqcup\wht\fg)_{p_\blt}}^*(\Wbb\otimes \Xbb\otimes\Vbb^{\otimes L})\\
\uppsi\mapsto\MS(\gimel\otimes\uppsi)\big|_{p_\blt}
\end{gathered}
\end{gather}

The propagation of conformal blocks (cf. \cite[Cor. 2.44]{GZ1}) gives isomorphisms
\begin{gather*}
\ST^*_{\wht\fg}\big(\boxtimes_\ff(\Wbb)\otimes\Xbb\otimes\Vbb^{\otimes L}\big)\simeq\ST^*_\fg\big(\boxtimes_\ff(\Wbb)\otimes\Xbb\big)\\
\ST_{\MS(\ff\sqcup\wht\fg)_{p_\blt}}^*(\Wbb\otimes \Xbb\otimes\Vbb^{\otimes L})\simeq \ST_{\MS(\ff\sqcup\fg)_{p_\blt}}^*(\Wbb\otimes \Xbb)
\end{gather*}
defined by inserting $\idt^{\otimes L}$ into the $\Vbb^{\otimes L}$-components of the conformal blocks. With the help of these two isomorphisms, the map \eqref{eq101} becomes \eqref{eq99}. Therefore \eqref{eq99} is an isomorphism.
\end{proof}

\subsection{The coend version of the SF theorem}

In this section, we continue to assume the setting in Subsec. \ref{lb44}. Moreover, as in Thm. \ref{lb43}, we assume that each component of $C_1$ contains one of $x_1,\dots,x_N$, and each component of $C_2$ contains one of $y_1,\dots,y_K,y_1',\dots,y_R'$.   Let $p_\blt\in\MD_{r_\blt\rho_\blt}^\times$ with fixed $\arg p_1,\dots,\arg p_R$. 

We refer the reader to Sec. \ref{lb56} or \cite[Sec. 2]{SF-Hopf-Frob} for the definition of coends. Let $\Vect$ be the category of finite-dimensional $\Cbb$-linear spaces.

\begin{lm}\label{lb53}
The family of linear maps
\begin{gather}\label{eq109}
\begin{gathered}
\Hom_{\Vbb^{{\otimes R}}}\big(\Mbb,\bbs_\ff(\Wbb)\big)\otimes_\Cbb\ST^*_\fg\big(\Mbb'\otimes\Xbb\big)\longrightarrow \ST^*_\fg\big(\boxtimes_\ff(\Wbb)\otimes\Xbb\big)\\
T\otimes\upchi\mapsto \upchi\circ(T^\tr\otimes\id_\Xbb)
\end{gathered}
\end{gather}
for $\Mbb\in\Mod(\Vbb^{\otimes R})$ is a coend in $\Vect$.
\end{lm}

In other words, \eqref{eq109} realizes a linear isomorphism
\begin{align*}
\int^{\Mbb\in\Mod(\Vbb^{\otimes R})}\Hom_{\Vbb^{{\otimes R}}}\big(\Mbb,\bbs_\ff(\Wbb)\big)\otimes_\Cbb\ST^*_\fg\big(\Mbb'\otimes\Xbb\big)\simeq\ST^*_\fg\big(\boxtimes_\ff(\Wbb)\otimes\Xbb\big)
\end{align*}
Note that the linear functional $\upchi:\Mbb'\otimes\Xbb\rightarrow\Cbb$ in \eqref{eq109} is a conformal block, where $\Mbb'\otimes\Xbb$ is associated to the ordered marked points $y_1',\dots,y_R',y_1,\dots,y_K$ of $\fg$. 

\begin{proof}
The map \eqref{eq109} is clearly dinatural. By \cite[Prop. 4]{FS-coends-CFT}, if $\scr D$ is a $\Cbb$-linear category, and if $G:\scr D\rightarrow\Vect$ is a $\Cbb$-linear functor, then for any $b\in\scr D$, the family of linear maps
\begin{gather*}
\Hom_{\scr D}(d,b)\otimes_\Cbb G(d)\rightarrow G(b)\qquad T\otimes \xi\mapsto G(T)\xi
\end{gather*}
for $d\in\scr D$ is a coend. Apply this result to $\scr D=\Mod(\Vbb^{\otimes R})$, $b=\bbs_\ff(\Wbb)$, and
\begin{align*}
G:\Mod(\Vbb^{\otimes R})\rightarrow\Vect\qquad \Mbb\mapsto \ST^*_\fg\big(\Mbb'\otimes\Xbb\big)
\end{align*}
Then the proof is complete.
\end{proof}

The following theorem is the coend version of the sewing-factorization theorem. In the case where $\ff$, $\fg$, and $\MS(\ff\sqcup\fg)|_{p_\blt}$ are all genus $0$ surfaces, a similar theorem was obtained in \cite{Moriwaki22-CB}.

\begin{thm}\label{lb54}
The family of linear maps
\begin{gather}\label{eq111}
\begin{gathered}
\ST^*_\ff\big(\Mbb\otimes\Wbb\big)\otimes_\Cbb\ST^*_\fg\big(\Mbb'\otimes\Xbb\big)\longrightarrow \ST_{\MS(\fk F\sqcup\fk G)_{p_\blt}}^*\big(\Wbb\otimes \Xbb\big)\\
\uppsi\otimes\upchi\mapsto  \MS(\uppsi\otimes \upchi)\big|_{p_\blt}
\end{gathered}
\end{gather}
for $\Mbb\in\Mod(\Vbb^{\otimes R})$ is a coend in $\Vect$.
\end{thm}

In other words, the sewing map realizes a linear isomorphism
\begin{align}
\int^{\Mbb\in\Mod(\Vbb^{\otimes R})}\ST^*_\ff\big(\Mbb\otimes\Wbb\big)\otimes_\Cbb\ST^*_\fg\big(\Mbb'\otimes\Xbb\big)\simeq\ST_{\MS(\fk F\sqcup\fk G)_{p_\blt}}^*\big(\Wbb\otimes \Xbb\big)
\end{align}

\begin{proof}
By the universal property for dual fusion products, the linear map
\begin{gather*}
\Hom_{\Vbb^{{\otimes R}}}\big(\Mbb,\bbs_\ff(\Wbb)\big)\rightarrow\ST^*_\ff\big(\Mbb\otimes\Wbb\big)\qquad T\mapsto \gimel\circ(T\otimes\id_\Wbb)
\end{gather*}
implements a natural equivalence between the contravariant functors $\Mbb\in\Mod(\Vbb^{\otimes R})\mapsto\Hom_{\Vbb^{\otimes R}}(\Mbb,\bbs_\ff(\Wbb))$ and $\Mbb\in\Mod(\Vbb^{\otimes R})\mapsto\ST^*_\ff\big(\Mbb\otimes\Wbb\big)$. Combining this fact with Lem. \ref{lb53}, we see that the family of linear maps
\begin{gather}\label{eq110}
\begin{gathered}
\ST^*_\ff\big(\Mbb\otimes\Wbb\big)\otimes_\Cbb\ST^*_\fg\big(\Mbb'\otimes\Xbb\big)\rightarrow\ST^*_\fg\big(\boxtimes_\ff(\Wbb)\otimes\Xbb\big)\\ \uppsi\otimes\upchi\mapsto \upchi\circ(T_\uppsi^\tr\otimes\id_\Xbb)
\end{gathered}
\end{gather}
for $\Mbb\in\Mod(\Vbb^{\otimes R})$ is a coend, where $T_\uppsi:\Mbb\rightarrow\bbs_\ff(\Wbb)$ is the unique homomorphism such that $\gimel\circ(T_\uppsi\otimes\id_\Wbb)=\uppsi$. 

By the sewing-factorization Thm. \ref{lb43}, the linear map \eqref{eq99} is an isomorphism. Hence $\eqref{eq99}\circ\eqref{eq110}$ is also a coend. We compute that $\eqref{eq99}\circ\eqref{eq110}$ sends $\uppsi\otimes\upchi$ to
\begin{align*}
\MS\big(\gimel\otimes(\upchi\circ (T_\uppsi^\tr\otimes\id_\Xbb))\big)\big|_{p_\blt}=\MS\big((\gimel\circ (T_\uppsi\otimes\id_\Wbb))\otimes\upchi\big)\big|_{p_\blt}=\MS(\uppsi\otimes\upchi)\big|_{p_\blt}
\end{align*}
Therefore, \eqref{eq111} is a coend.
\end{proof}

\subsection{The SF theorems for families of compact Riemann surfaces}

\subsubsection{The setting}\label{lb50}

In this section, we generalize the sewing-factorization theorems \ref{lb40} and \ref{lb43} to families of compact Riemann surfaces. Let $R,N,K\in\Nbb$. Let
\begin{gather*}
\fk F=\big(x_1',\dots,x_R'; \theta_1',\dots,\theta_R'\big|\wtd\pi_1:\wtd\MC_1\rightarrow\wtd\MB\big|x_1,\dots,x_N;\theta_1,\dots,\theta_N\big)\\
\fk G=\big(y_1',\dots,y_R';\mu_1',\dots,\mu_R'\big|\wtd\pi_2:\wtd\MC_2\rightarrow\wtd\MB\big| y_1,\dots,y_K;\mu_1,\dots,\mu_K \big)
\end{gather*}
be $(R,N)$-pointed and $(R,K)$-pointed \textbf{families of compact Riemann surfaces} with common base manifold $\wtd\MB$. Therefore, $\wtd\MC_1,\wtd\MB$ are complex manifolds, and $\wtd\pi_1$ is a proper holomorphic submersion  such that for each $b\in\wtd\MB$, the fiber $\wtd\pi_1^{-1}(b)$ has pure dimension $1$. Each $x_i:\wtd\MB\rightarrow\wtd\MC_1$ is a holomorphic section, i.e., a holomorphic map such that $\wtd\pi_1\circ x_i=\id_{\wtd\MB}$. Moreover, $\theta_i$ is a \textbf{local coordinate} at $x_i(\wtd\MB)$, i.e., a holomorphic map $\theta_i:\wtd U_i\rightarrow\Cbb$ (where $\wtd U_i$ is an open set containing $x_i(\wtd\MB)$) which is univalent on  $\wtd U_{i,b}=\wtd U_i\cap\wtd\pi_1^{-1}(b)$ for each $b\in\wtd\MB$, and which maps $x_i(\wtd\MB)$ to $0$. Similarly, each $x_j'$ is a holomorphic section, and $\theta_j'$ is a local coordinate at $x_j'(\wtd\MB)$. We assume that $x_1(\wtd\MB),\dots,x_N(\wtd\MB),x_1'(\wtd\MB),\dots,x_R'(\wtd\MB)$ are mutually disjoint. A similar description applies to $\fg$.

For each $1\leq j\leq R$, let $V_j'$ be a neighborhood of $x_j'(\wtd\MB)$ on which $\theta_j'$ is defined (and univalent), and let $V_j''$ be a neighborhood of $x_j''(\wtd\MB)$ on which $\theta_j''$ is defined. We assume that
\begin{gather*}
x_1(\wtd\MB),\dots,x_N(\wtd\MB),V_1',\dots,V_R'\quad\text{are mutually disjoint}\\
y_1(\wtd\MB),\dots,y_K(\wtd\MB),V_1'',\dots,V_R''\quad\text{are mutually disjoint}
\end{gather*}
We also assume that there are $0<r_1,\dots,r_R,\rho_1,\dots,\rho_R\leq+\infty$ such that
\begin{align}
(\wtd\pi_1,\theta_j')(V_j')=\wtd\MB\times\MD_{r_j} \qquad (\wtd\pi_2,\mu_j')(V_j'')=\wtd\MB\times\MD_{\rho_j}
\end{align}

\begin{df}
Let $\MS(\ff\sqcup\fg)$ be the family obtained by sewing $\ff\sqcup\fg$ along the pairs of marked points $(x_j',y_j')$ (for all $1\leq j\leq R$) with sewing radii $(r_j,\rho_j)$ using their local coordinates $\theta_j',\mu_j'$. The sections $x_i,y_j$ and the local coordinates $\theta_i,\mu_j$ of $\ff,\fg$ are extended constantly to those of $\MS(\ff\sqcup\fg)$, and we continue to denote them by the same symbols. See \cite[Subsec. 1.2.1]{GZ2} for the detailed construction. Let
\begin{gather}
\begin{gathered}
\MB=\wtd\MB\times\MD_{r_\blt\rho_\blt}=\wtd\MB\times\MD_{r_1\rho_1}\times\cdots\times\MD_{r_R\rho_R}\\
\MB^\times=\wtd\MB\times\MD_{r_\blt\rho_\blt}^\times=\wtd\MB\times\MD^\times_{r_1\rho_1}\times\cdots\times\MD^\times_{r_R\rho_R}
\end{gathered}
\end{gather}
Then $\MS(\ff\sqcup\fg)$ has base manifold $\MB$ and has smooth fibers over $\MB^\times$. We write
\begin{align}\label{eq106}
\MS(\ff\sqcup\fg)=\big(\pi:\MC\rightarrow\MB\big|x_1,\dots,x_N,y_1,\dots,y_K;\theta_1,\dots,\theta_N,\mu_1,\dots,\mu_K  \big)
\end{align}
As in Subsec. \ref{lb44}, we let $q_j$ be the standard coordinate of $\MD_{r_j\rho_j}$.
\end{df}

We assume that each component of each fiber $\wtd\pi_1^{-1}(b)$ intersects $x_1(\wtd\MB)\cup\cdots\cup x_N(\wtd\MB)$, and each component of each fiber $\wtd\pi_2^{-1}(b)$ intersects $y_1(\wtd\MB)\cup\cdots\cup y_K(\wtd\MB)\cup y_1'(\wtd\MB)\cup\dots\cup y_R'(\wtd\MB)$.

As in Subsec. \ref{lb44}, we associate $\Wbb\in\Mod(\Vbb^{\otimes N})$ to the ordered incoming marked points $x_1,\dots,x_N$ of $\ff$. Similarly, associate $\Xbb\in\Mod(\Vbb^{\otimes K})$ to $y_1,\dots,y_K$. Associate $\Mbb\in\Mod(\Vbb^{\otimes R})$ to $x_1',\dots,x_R'$, and associate its contragredient $\Mbb'$ to $y_1',\dots,y_R'$. 

We view $\Mbb\otimes\Wbb$ as associated to $x_\blt',x_\blt$, and $\Mbb'\otimes\Xbb$ as associated to $y_\blt',y_\blt$. Associate $\Wbb\otimes\Xbb$ to the ordered marked points $x_1(\MB),\dots,x_N(\MB),y_1(\MB),\dots,y_K(\MB)$ of $\MS(\ff\sqcup\fg)$.

\begin{df}
Suppose that
\begin{gather*}
\upphi\in H^0\big(\wtd\MB,\ST^*_\ff(\Mbb\otimes\Wbb)\big)\qquad \upchi\in H^0\big(\wtd\MB,\ST^*_\fg(\Mbb'\otimes\Xbb)\big)
\end{gather*}
The \textbf{sewing of $\upphi$ and $\upchi$} is defined by contracting the $\Mbb$-component of $\upphi$ with the $\Mbb'$-component of $\upchi$. More precisely,
\begin{subequations}
\begin{gather}
\begin{gathered}
\MS(\upphi\otimes\upchi):\Wbb\otimes\Xbb\rightarrow\MO(\wtd\MB)\{q_\blt\}[\log q_\blt]\\[0.5ex]
\bigbk{\MS(\upphi\otimes\upchi),w\otimes \tipaz}=\sum_{\lambda_\blt\in\Cbb^R}\sum_{\alpha\in\fk A_{\lambda_\blt}}\bigbk{\upphi,q_\blt^{L_\blt(0)}e_{\lambda_\blt}(\alpha)\otimes w}\cdot \bigbk{\upchi,\wch e_{\lambda_\blt}(\alpha)\otimes\tipaz}
\end{gathered}
\end{gather}
(where $w\in\Wbb,\tipaz\in\Xbb$). Here, $(e_{\lambda_\blt}(\alpha))_{\alpha\in\fk A_{\lambda_\blt}}$ is a basis of $\Mbb_{[\lambda_\blt]}$ with dual basis $(\wch e_{\lambda_\blt}(\alpha))_{\alpha\in\fk A_{\lambda_\blt}}$ in $\Mbb'_{[\lambda_\blt]}$, and $q_\blt^{L_\blt(0)}=\eqref{eq102}$. See \cite[Sec. 4.1]{GZ2} for more explanations. Clearly we also have
\begin{gather}
\bigbk{\MS(\upphi\otimes\upchi),w\otimes \tipaz}=\sum_{\lambda_\blt\in\Cbb^R}\sum_{\alpha\in\fk A_{\lambda_\blt}}\bigbk{\upphi,e_{\lambda_\blt}(\alpha)\otimes w}\cdot \bigbk{\upchi,q_\blt^{L_\blt(0)}\wch e_{\lambda_\blt}(\alpha)\otimes \tipaz}
\end{gather}
\end{subequations}
\end{df}

\begin{rem}\label{lb45}
By \cite[Thm. 4.9]{GZ2}, $\MS(\upphi\otimes\upchi)$ converges a.l.u. on $\MB^\times$ in the sense of Def. \ref{lb19}, and
\begin{align}\label{eq113}
\MS(\upphi\otimes\upchi)\big|_{\MB^\times}\in H^0\big(\wtd\MB\times\wht\MD_{r_\blt\rho_\blt}^\times,\ST^*_{\MS(\ff\sqcup\fg)}(\Wbb\otimes\Xbb)\big)
\end{align}
where $\wht\MD_{r_\blt\rho_\blt}^\times$ is the universal cover of $\MD_{r_\blt\rho_\blt}^\times$. By \cite[Rem. 3.1]{GZ2}, \eqref{eq113} is equivalent to that for each $b=(b_0,p_1,\dots,p_R)\in\MB^\times$ with chosen $\arg p_1,\dots,\arg p_R$, we have
\begin{align*}
\MS(\upphi\otimes\upchi)\big|_b\in \ST^*_{\MS(\ff\sqcup\fg)_b}(\Wbb\otimes\Xbb)
\end{align*}
\end{rem}

\begin{df}[\cite{GZ2} Def. 3.18]\label{lb52}
A \textbf{dual fusion product} of $\Wbb$ along $\ff$ is a pair $(\bbs_\ff(\Wbb),\gimel)$ where $\bbs_\ff(\Wbb)\in\Mod(\Vbb^{\otimes R})$, and $\gimel\in H^0(\wtd\MB,\scr T^*_\ff(\bbs_\fx(\Wbb)\otimes\Wbb))$ satisfies that for each $b\in\wtd\MB$, the pair $(\bbs_\ff(\Wbb),\gimel|_b)$ is a dual fusion product of $\Wbb$ along $\ff_b$. The contragredient $\Vbb^{\otimes R}$-module of $\bbs_\ff(\Wbb)$ is denoted by $\boxtimes_\ff(\Wbb)$ and called the \textbf{fusion product} of $\Wbb$ along $\ff$.
\end{df}

\begin{rem}\label{lb106}
Dual fusion products, when they exist, are unique up to unique holomorphic isomorphisms. See \cite[Cor. 3.21]{GZ2} for a detailed description. Moreover, by \cite[Cor. 3.23]{GZ2}, a dual fusion product $(\bbs_\ff(\Wbb),\gimel)$ exists provided that $\wtd\MB$ is a \textbf{polydisk}, i.e., $\wtd\MB=\MD_{\eps_1}\times\cdots\times\MD_{\eps_m}$ for some $m\in\Nbb$ and $\eps_1,\dots,\eps_m\in(0,+\infty]$. 
\end{rem}

\subsubsection{The sewing-factorization theorems}

For each $p_\blt\in\MD_{r_\blt\rho_\blt}^\times$, we let $\MS(\ff\sqcup\fg)_{\wtd\MB\times p_\blt}$ be the restriction of $\MS(\ff\sqcup\fg)$ to the submanifold $\wtd\MB\times p_\blt$. Equivalently, it is the pullback of $\MS(\ff\sqcup\fg)$ along the map $b\in\wtd\MB\mapsto (b,p_\blt)\in\MB$. In this way, we may regard $\MS(\ff\sqcup\fg)_{\wtd\MB\times p_\blt}$ as a family with base manifold $\wtd\MB$. We view
\begin{align*}
\Hom_{\Vbb^{\otimes R}}\big(\boxtimes_\fg(\Xbb),\bbs_\ff(\Wbb)\big)\otimes_\Cbb \MO_{\wtd\MB}
\end{align*}
as the sheaf of (germs of) holomorphic functions on $\wtd\MB$ whose values are homomorphisms $\boxtimes_\fg(\Xbb)\rightarrow\bbs_\ff(\Wbb)$.

We now prove the sewing-factorization theorems for families, generalizing Thm. \ref{lb40} and Thm. \ref{lb43}. We emphasize that the dual fusion products mentioned in the following theorems exist whenever $\wtd\MB$ is a polydisk.

\begin{thm}\label{lb46}
For each $b\in\wtd\MB$, assume that every component of $\wtd\pi_1^{-1}(b)$ intersects $x_1(\wtd\MB)\cup\cdots\cup x_N(\wtd\MB)$, and every component of $\wtd\pi_2^{-1}(b)$ intersects $y_1(\wtd\MB)\cup\cdots \cup y_K(\wtd\MB)$. Assume that $(\bbs_\ff(\Wbb),\gimel)$ is a dual fusion product of $\Wbb$ along $\ff$, and $(\bbs_\fg(\Xbb),\daleth)$ is a dual fusion product of $\Xbb$ along $\fg$. Let $p_\blt\in\MD_{r_\blt\rho_\blt}^\times$ with fixed $\arg p_1,\dots,\arg p_R$. Then the following sheaf map is an $\MO_{\wtd\MB}$-module isomorphism.
\begin{gather}\label{eq103}
\begin{gathered}
\Hom_{\Vbb^{\otimes R}}\big(\boxtimes_\fg(\Xbb),\bbs_\ff(\Wbb)\big)\otimes_\Cbb\MO_{\wtd\MB}\xlongrightarrow{\simeq} \ST_{\MS(\fk F\sqcup\fk G)_{\wtd\MB\times p_\blt}}^*(\Wbb\otimes \Xbb)\\[0.5ex]
T\mapsto \MS\big((\gimel\circ(T\otimes \id_\Wbb))\otimes \daleth\big)\big|_{\wtd\MB\times p_\blt}
\end{gathered}
\end{gather}
\end{thm}

Note that \eqref{eq103} is equivalent to
\begin{align}
\begin{gathered}
\Hom_{\Vbb^{\otimes R}}\big(\boxtimes_\ff(\Wbb),\bbs_\fg(\Xbb)\big)\otimes_\Cbb\MO_{\wtd\MB}\xlongrightarrow{\simeq} \ST_{\MS(\fk F\sqcup\fk G)_{\wtd\MB\times p_\blt}}^*(\Wbb\otimes \Xbb)\\[0.5ex]
T\mapsto \MS\big(\gimel\otimes (\daleth\circ(T\otimes \id_\Xbb))\big)\big|_{\wtd\MB\times p_\blt}
\end{gathered}
\end{align}

\begin{proof}
By Rem. \ref{lb45}, the map $\Phi:=$\eqref{eq103} is well-defined and is clearly an $\MO_{\wtd\MB}$-module morphism. Denote the source and the target of $\Phi$ by $\scr E$ and $\scr F$, which are locally free. Thus, $\Phi$ can be viewed as a morphism of vector bundles. In particular, each $b_0\in\wtd B$ is contained in a neighborhood $\Omega\subset\wtd\MB$ such that $\Phi|_\Omega$ can be viewed as a holomorphic matrix-valued function.  For each $b\in\Omega$, by \cite[Thm. 3.13]{GZ2} (see also Rem. \ref{lb49}), we have a linear isomorphism
\begin{align*}
\scr F_b/\mk_{\wtd\MB,b}\scr F_b\xlongrightarrow{\simeq}\ST^*_{\MS(\ff\sqcup\fg)_{b\times p_\blt}}(\Wbb\otimes\Xbb)
\end{align*}
defined by sending each $\uppsi$ of the stalk $\scr F_b$ to $\uppsi|_b$. Therefore, by Thm. \ref{lb40}, the restriction of $\Phi$ to the fiber at $b$ is a linear isomorphism, i.e., the induced linear map $\scr E/\mk_{\wtd\MB,b}\scr E\rightarrow \scr F/\mk_{\wtd\MB,b}\scr F$ is an isomorphism. Therefore, if we view $\Phi|_\Omega$ as  a matrix-valued function, then for each $b\in\Omega$, the matrix $\Phi|_b$ is invertible. This proves that $\Phi$ is an $\MO_{\wtd\MB}$-module isomorphism.
\end{proof}

\begin{thm}\label{lb47}
For each $b\in\wtd\MB$, assume that every component of $\wtd\pi_1^{-1}(b)$ intersects $x_1(\wtd\MB)\cup\cdots\cup x_N(\wtd\MB)$, and every component of $\wtd\pi_2^{-1}(b)$ intersects $y_1(\wtd\MB)\cup\cdots \cup y_K(\wtd\MB)\cup y_1'(\wtd\MB)\cup\cdots\cup y_R'(\wtd\MB)$. Assume that $(\bbs_\ff(\Wbb),\gimel)$ is a dual fusion product of $\Wbb$ along $\ff$. Let $p_\blt\in\MD_{r_\blt\rho_\blt}^\times$ with fixed $\arg p_1,\dots,\arg p_R$. Then the following sheaf map is an $\MO_{\wtd\MB}$-module isomorphism.
\begin{gather}
\begin{gathered}
\ST_{\fg}^*\big(\boxtimes_\ff(\Wbb)\otimes \Xbb\big)\xlongrightarrow{\simeq} \ST_{\MS(\fk F\sqcup\fk G)_{\wtd\MB\times p_\blt}}^*\big(\Wbb\otimes \Xbb\big)\\[0.5ex]
\upchi \mapsto \MS(\gimel\otimes \upchi)\big|_{\wtd\MB\times p_\blt}
\end{gathered}
\end{gather}
\end{thm}

\begin{proof}
This theorem follows from Thm. \ref{lb43} by the same reasoning as the derivation of Thm. \ref{lb46} from Thm. \ref{lb40}.
\end{proof}

\subsection{Transitivity of fusion products as an SF theorem}

We continue to work in the setting of Subsec. \ref{lb50}. Let $\fh$ denote the same surface $\fg$ but with incoming and outgoing marked points exchanged. Namely,
\begin{gather*}
\fk F=\big(x_1',\dots,x_R'; \theta_1',\dots,\theta_R'\big|\wtd\pi_1:\wtd\MC_1\rightarrow\wtd\MB\big|x_1,\dots,x_N;\theta_1,\dots,\theta_N\big)\\
\fh=\big(y_1,\dots,y_K;\mu_1,\dots,\mu_K\big|\wtd\pi_2:\wtd\MC_2\rightarrow\wtd\MB\big|y_1',\dots,y_R';\mu_1',\dots,\mu_R'  \big)
\end{gather*}
Let $\fh\circ\ff$ denote the same family as $\MS(\ff\sqcup\fg)=\eqref{eq106}$, except with the marked points $y_1,\dots,y_K$ reinterpreted as outgoing marked points. Namely,
\begin{align*}
\fh\circ\ff=\big(y_1,\dots,y_K;\mu_1,\dots,\mu_K\big|\pi:\MC\rightarrow\MB\big|x_1,\dots,x_N;\theta_1,\dots,\theta_N  \big)
\end{align*}
is a family of $(K,N)$-pointed compact Riemann surfaces with local coordinates. For each $p_\blt\in\MD^\times_{r_\blt\rho_\blt}$, define
\begin{align*}
\fh\circ_{p_\blt}\ff~=\text{ the restriction of $\fh\circ\ff$ to $\wtd\MB\times p_\blt$}
\end{align*}
Equivalently, $\fh\circ_{p_\blt}\ff$ is the pullback of $\fh\circ\ff$ along the map $b\in\wtd\MB\mapsto (b,p_\blt)\in\MB$. Our notation suggests that we are viewing $\fh,\ff$ as morphisms in a cobordism category, with $\fh\circ_{p_\blt}\ff$ representing their composition.

\subsubsection{Transitivity of fusion products}

We now prove the transitivity of fusion products, which can be viewed as another version of the sewing-factorization theorem. Roughly speaking, this property says that
\begin{align}
\boxtimes_{\fh\circ_{p_\blt}\ff}(\Wbb)\simeq\boxtimes_{\fh}(\boxtimes_\ff(\Wbb))
\end{align}
Note that by Rem. \ref{lb106}, the dual fusion products in the following theorem exist whenever $\wtd\MB$ is a polydisk.

\begin{thm}\label{lb51}
For each $b\in\wtd\MB$, assume that every component of $\wtd\pi_1^{-1}(b)$ intersects $x_1(\wtd\MB)\cup\cdots\cup x_N(\wtd\MB)$, and every component of $\wtd\pi_2^{-1}(b)$ intersects $y_1'(\wtd\MB)\cup\cdots\cup y_R'(\wtd\MB)$. 

Associate $\Wbb\in\Mod(\Vbb^{\otimes N})$ to the ordered marked points $x_1(\wtd\MB),\dots,x_N(\wtd\MB)$. Assume that $(\bbs_\ff(\Wbb),\gimel)$ is a dual fusion product of $\Wbb$ along $\ff$. Associate $\boxtimes_\ff(\Wbb)$ to the ordered marked points $y_1'(\wtd\MB),\cdots,y_R'(\wtd\MB)$. Assume that $\big({\bbs_{\fh}(\boxtimes_\ff(\Wbb))},\mem\big)$ is a dual fusion product of $\boxtimes_\ff(\Wbb)$ along $\fh$. Choose $p_\blt\in\MD_{r_\blt\rho_\blt}^\times$ with fixed $\arg p_1,\dots,\arg p_R$. Then
\begin{align}
\big(~{\bbs_{\fh}(\boxtimes_\ff(\Wbb))},\MS(\mem\otimes\gimel)\big|_{\wtd\MB\times p_\blt}~\big)
\end{align}
is a dual fusion product of $\Wbb$ along $\fh\circ_{p_\blt}\ff$.
\end{thm}

Note that $\mem:{\bbs_{\fh}(\boxtimes_\ff(\Wbb))}\otimes \boxtimes_\ff(\Wbb)\rightarrow\Cbb$ and $\gimel:\bbs_\ff(\Wbb)\otimes\Wbb\rightarrow\Cbb$ are canonical conformal blocks, and the sewing
\begin{align*}
\MS(\mem\otimes\gimel):{\bbs_{\fh}(\boxtimes_\ff(\Wbb))}\otimes\Wbb\rightarrow\MO(\wtd\MB)\{q_\blt\}[\log q_\blt]
\end{align*}
is defined by contracting the $\boxtimes_\ff(\Wbb)$-component of $\mem$ with the $\bbs_\ff(\Wbb)$-component of $\gimel$.

\begin{proof}
By Def. \ref{lb52}, it suffices to prove that for each $b\in\wtd\MB$, the $\Vbb^{\otimes K}$-module ${\bbs_{\fh}(\boxtimes_\ff(\Wbb))}$, together with the conformal block $\MS(\mem\otimes\gimel)|_{b\times p_\blt}=\MS(\mem|_b\otimes\gimel|_b)|_{p_\blt}$, is a dual fusion product of $\Wbb$ along $(\fh\circ_{p_\blt}\ff)_b=(\fh\circ\ff)_{b\times p_\blt}$. Therefore, we may assume without loss of generality that $\wtd\MB=\{0\}$.

Let $\Xbb\in\Mod(\Vbb^{\otimes K})$. We need to show that the linear map
\begin{gather}\label{eq107}
\begin{gathered}
\Hom_{\Vbb^{\otimes K}}\big(\Xbb,\bbs_\fh(\boxtimes_\ff(\Wbb))\big)\rightarrow \ST^*_{\fh\circ_{p_\blt}\ff}(\Xbb\otimes\Wbb)\qquad
T\mapsto \MS(\mem\otimes\gimel)\big|_{p_\blt}\circ(T\otimes\id_\Wbb)
\end{gathered}
\end{gather}
is an isomorphism. Note that
\begin{align*}
\MS(\mem\otimes\gimel)\big|_{p_\blt}\circ(T\otimes\id_\Wbb)=\MS\big((\mem\circ (T\otimes\id_{\boxtimes_\ff(\Wbb)}))\otimes\gimel\big)\big|_{p_\blt}
\end{align*}
where $\mem\circ (T\otimes\id_{\boxtimes_\ff(\Wbb)}):\Xbb\otimes\boxtimes_\ff(\Wbb)\rightarrow\Cbb$ is a conformal block associated to $\fh$. Therefore, \eqref{eq107} is the composition of
\begin{gather*}
\ST^*_\fh(\Xbb\otimes\boxtimes_\ff(\Wbb))\rightarrow\ST^*_{\fh\circ_{p_\blt}\ff}(\Xbb\otimes\Wbb)\qquad \upchi\mapsto \MS(\upchi\otimes\gimel)\big|_{p_\blt}
\end{gather*}
(which is an isomorphism by the sewing-factorization Thm. \ref{lb43}) with
\begin{gather*}
\Hom_{\Vbb^{\otimes K}}\big(\Xbb,\bbs_\fh(\boxtimes_\ff(\Wbb))\big)\rightarrow\ST^*_\fh(\Xbb\otimes\boxtimes_\ff(\Wbb))\qquad T\mapsto \mem\circ(T\otimes\id_{\boxtimes_\ff(\Wbb)})
\end{gather*}
(which is an isomorphism by the universal property of $\mem$). Therefore, \eqref{eq107} is an isomorphism.
\end{proof}

\subsubsection{Example: The associativity isomorphisms}\label{lb55}

In this subsection, we briefly explain how the associativity isomorphisms in the theory of Huang-Lepowsky-Zhang \cite{HLZ1,HLZ2}-\cite{HLZ8} can be constructed using our sewing-factorization Thm. \ref{lb51}.

Let $\zeta$ be the standard coordinate of $\Cbb$. For $z,z_1,z_2\in\Cbb^\times=\Cbb-\{0\}$, let
\begin{gather*}
\fk A=\big(\infty;1/\zeta\big|\Pbb^1\big|0;\zeta\big)\qquad \fk P_z=\big(\infty;1/\zeta\big|\Pbb^1\big|z,0;\zeta-z,\zeta \big)\\
\fk P_{z_1,z_2}=\big(\infty;1/\zeta\big|\Pbb^1\big|z_1,z_2,0;\zeta-z_1,\zeta-z_2,\zeta \big)
\end{gather*}
which are $(1,1),(1,2),(1,3)$-pointed spheres with local coordinates. For each $\Wbb,\Mbb\in\Mod(\Vbb)$, we associate $\Wbb\otimes\Mbb$ to the ordered marked points $z,0$ of $\fp_z$. Let
\begin{gather*}
(\Wbb\boxtimes_{\fp_z}\Mbb,\mc Y_z)
\end{gather*}
be a  fusion product of $\Wbb\otimes\Mbb$ along $\fp_z$, where $\mc Y_z:\Wbb\otimes\Mbb\rightarrow \ovl{\Wbb\boxtimes_{\fp_z}\Mbb}$ is the linear map such that
\begin{align*}
\fk w\otimes w_1\otimes w_2\in(\Wbb\bbs_{\fp_z}\Mbb)\otimes\Wbb\otimes\Mbb\mapsto \bk{\fk w,\mc Y_z(w_1,w_2)}
\end{align*}
is the canonical conformal block. (Here, $\Wbb\bbs_{\fp_z}\Mbb$ is the contragredient of $\Wbb\boxtimes_{\fp_z}\Mbb$.) 

In what follows, all canonical conformal blocks will be expressed as linear maps in this form. For example, $(\Wbb,\id_\Wbb)$ is a  fusion product of $\Wbb$ along $\fk A$.

Now, we assume that $0<|z_1-z_2|<|z_2|<|z_1|$. Then we have a canonical equivalence
\begin{align*}
\fp_{z_1}\circ_{1,1} (\fk A\sqcup\fp_{z_2})\simeq \fp_{z_1,z_2}
\end{align*}
Let $\Wbb_1,\Wbb_2,\Wbb_3\in\Mod(\Vbb)$. By Thm. \ref{lb16},
\begin{gather*}
\big(\Wbb_1\otimes(\Wbb_2\boxtimes_{\fp_{z_2}}\Wbb_3),\id_{\Wbb_1}\otimes\mc Y_{z_2} \big)
\end{gather*}
is a  fusion product of $\Wbb_1\otimes\Wbb_2\otimes\Wbb_3$ along $\fk A\sqcup\fp_{z_1}$. Therefore, by the sewing-factorization Thm. \ref{lb51},
\begin{align*}
\big(\Wbb_1\boxtimes_{\fp_{z_1}}(\Wbb_2\boxtimes_{\fp_{z_2}}\Wbb_3),\mc Y_{z_1}\circ(\id_{\Wbb_1}\otimes\mc Y_{z_2}) \big)
\end{align*}
is a  fusion product of $\Wbb_1\otimes\Wbb_2\otimes\Wbb_3$ along $\fp_{z_1,z_2}$. 

Similarly, we have a canonical equivalence
\begin{gather*}
\fp_{z_2}\circ_{1,1}(\fp_{z_1-z_2}\sqcup\fk A)\simeq\fp_{z_1,z_2}
\end{gather*}
By Thm. \ref{lb16},
\begin{align*}
\big((\Wbb_1\boxtimes_{\fp_{z_1-z_2}}\Wbb_2)\otimes\Wbb_3,\mc Y_{z_1-z_2}\otimes\id_{\Wbb_3}\big)
\end{align*}
is a  fusion product of $\Wbb_1\otimes\Wbb_2\otimes\Wbb_3$ along $\fp_{z_1-z_2}\sqcup\fk A$. Thus, by Thm. \ref{lb51},
\begin{align*}
\big((\Wbb_1\boxtimes_{\fp_{z_1-z_2}}\Wbb_2)\boxtimes_{\fp_{z_2}}\Wbb_3,\mc Y_{z_2}\circ(\mc Y_{z_1-z_2}\otimes\id_{\Wbb_3})\big)
\end{align*}
is a  fusion product of $\Wbb_1\otimes\Wbb_2\otimes\Wbb_3$ along $\fp_{z_1,z_2}$. Therefore, by the uniqueness of (dual) fusion products, there is a (unique) isomorphism 
\begin{gather*}
\Phi_{z_1,z_2}:\Wbb_1\boxtimes_{\fp_{z_1}}(\Wbb_2\boxtimes_{\fp_{z_2}}\Wbb_3)\xlongrightarrow{\simeq}  (\Wbb_1\boxtimes_{\fp_{z_1-z_2}}\Wbb_2)\boxtimes_{\fp_{z_2}}\Wbb_3
\end{gather*}
which, upon extension to a linear map between the algebraic completions of the respective modules, satisfies for all $w_1\in\Wbb_1,w_2\in\Wbb_2,w_3\in\Wbb_3$ the equality
\begin{gather}
\mc Y_{z_2}\circ(\mc Y_{z_1-z_2}(w_1\otimes w_2)\otimes w_3)=\Phi_{z_1,z_2}\circ \mc Y_{z_1}(w_1\otimes\mc Y_{z_2}(w_2\otimes w_3))\label{eq108}
\end{gather}

We assume that $z\mapsto\mc Y_z$ is holomorphic. Therefore, by \cite[Prop. 3.20]{GZ2}, if we vary $z_1,z_2$, then $\Phi_{z_1,z_2}$ is holomorphic with respect to $z_1,z_2$. Moreover, note that $\Pbb^1$ has a canonical projective structure, i.e., the one consisting of M\"obius transformations. Therefore, by \cite[Cor. 2.32]{GZ2}, for any family of spheres with marked points and local coordinates such that the restriction of each coordinate to each fiber is a M\"obius transformation, there is a unique connection $\nabla$ on the conformal block bundle that is compatible with the projective structure. If we further require that $z\mapsto\mc Y_z$ is parallel under this connection, then by \cite[Thm. 4.11]{GZ2}, both $\mc Y_{z_2}\circ(\mc Y_{z_1-z_2}\otimes\id_{\Wbb_3})$ and $\mc Y_{z_1}\circ(\id_{\Wbb_1}\otimes\mc Y_{z_2})$ are parallel sections of conformal blocks as $z_1,z_2$ vary. Therefore, by \eqref{eq108}, $(z_1,z_2)\mapsto\Phi_{z_1,z_2}$ must be a constant isomorphism $\Phi$. This yields the \textbf{associativity isomorphism} $\Phi$.

\footnotesize
	\bibliographystyle{alpha}

\noindent {\small \sc Yau Mathematical Sciences Center, Tsinghua University, Beijing, China.}

\noindent {\textit{E-mail}}: binguimath@gmail.com\qquad bingui@tsinghua.edu.cn\\

\noindent {\small \sc Yau Mathematical Sciences Center and Department of Mathematics, Tsinghua University, Beijing, China.}

\noindent {\textit{E-mail}}: zhanghao1999math@gmail.com \qquad h-zhang21@mails.tsinghua.edu.cn
\end{document}